\documentclass[12pt,reqno]{amsart}

\usepackage{amssymb,amsmath,exscale}
\usepackage[applemac]{inputenc}
\usepackage{graphics}
\usepackage{graphicx}
\usepackage{amsthm}
\usepackage{eufrak}

\def\paragraph#1{{\bf #1\ }}

\def\NN{\mbox{\rm I\hspace{-0.50ex}N} }
 \def\OO{\rm \hbox{O\kern-.34em\raise.47ex
         \hbox{$\scriptscriptstyle |$}\kern-.46em\raise.47ex
         \hbox{$\scriptscriptstyle |$}\kern+0.5 em }}
\def\RR{\mbox{\mathrm I\hspace{-0.50ex}R} }

\def\ZZ{\mbox{\mathsf Z\hspace{-0.80ex}Z} }

\def\hcboxcm#1#2{\hbox to #1{\hfill #2 \hfill}}

\def\null{\hbox{}}

\def\tn1{\widetilde n_1}
\def\tn2{\widetilde n_2}
\def\tn{\widetilde n }

\let\ds\displaystyle

\def\be{\begin{equation}}
\def\ee{\end{equation}}
\def\bea{\begin{eqnarray}}
\def\eea{\end{eqnarray}}
\def\bean{\begin{eqnarray*}}
\def\eean{\end{eqnarray*}}

\def\RR{{\mathrm{ I~\hspace{-1.15ex}R}}}
\def\NN{{\mathrm{ I~\hspace{-1.15ex}N}}}
\def\qquad{{\quad\quad}}

\def\={\, = \, }

\def\T{{\cal T}}

\def\Box{\leavevmode\vbox{\hrule
     \hbox{\vrule\kern4pt\vbox{\kern4pt}%
           \vrule}\hrule}}
\def\blackbox{\leavevmode\vrule height 5pt width 4pt depth 0pt\relax}
\catcode`@=11

\def\eqalign#1{\null\,\vcenter{\openup1\jot \m@th
   \ialign{\strut \hfil$\displaystyle{##}$ & $\displaystyle{{}##}$\hfil
      \crcr#1\crcr}}\,}
\def\eqalignrll#1{\null\,\vcenter{\openup1\jot \m@th
   \ialign{\strut \hfil$\displaystyle{##}$ & $\displaystyle{{}##}$\hfil
    & $\displaystyle{{}##}$\hfil
      \crcr#1\crcr}}\,}
\def\eqalignrcl#1{\null\,\vcenter{\openup1\jot \m@th
   \ialign{\strut \hfil$\displaystyle{##}$ &\hfil $\displaystyle{{}##}$\hfil
    & $\displaystyle{{}##}$\hfil
      \crcr#1\crcr}}\,}
\def\eqalignno#1{\displ@y \tabskip\@centering
  \halign to\displaywidth{\hfil$\@lign\displaystyle{##}$\tabskip\z@skip
    &$\@lign\displaystyle{{}##}$\hfil\tabskip\@centering
    &\llap{$\@lign##$}\tabskip\z@skip\crcr
    #1\crcr}}
\newcounter{appendix}
\newcounter{sectionz}
\def\appendix{\advance\c@appendix by 1
   \def\thesectionz {\Alph{appendix}}
\def\thesection{\Alph{appendix}} 
   \ifnum\c@appendix=1 \setcounter{section}{-1} \fi
   \@startsection {section}{1}{\z@}{-3.5ex plus -1ex minus 
  -.2ex}{2.3ex plus .2ex}{\large\bf}}

\catcode`@=12
\newtheorem{lemme}{Lemma}[section]  

\newtheorem{theorem}[lemme]{Theorem}

\newtheorem{corollary}[lemme]{Corollary}

\newtheorem{definition}[lemme]{Definition}

\newtheorem{proposition}[lemme]{Proposition}

\newtheorem{remark}[lemme]{Remark} 

\def\deblem{\begin{lemme}\it }
\def\finlem{\end{lemme}}
\def\debthm{\begin{theorem}\it }
\def\finthm{\end{theorem}}
\def\debprop{\begin{proposition} \it}
\def\finprop{\end{proposition}}
\def\debcor{\begin{corollary}\it }
\def\fincor{\end{corollary}}
\def\debdef{\begin{definition}\it}
\def\findef{\end{definition}}
\def\debrem{\begin{remark}\em}
\def\finrem{\null\hfill\blackbox\end{remark}}
\def\debproof{{\bf Proof: \ }}
\def\finproof{\null\hfill {$\blackbox$}\bigskip}  
\def\NN{\mathbb{N}}
\def\OO{\mathbb{O}}

\def\RR{\mathbb{R}}

\def\ZZ{\mathbb{Z}}

\def\bfu{\mathbf{u}}
\def\bfv{\mathbf{v}}
\def\bfx{\mathbf{x}}

\def\bfE{\mathbf{E}}
\def\bfB{\mathbf{B}}

\newcommand{\converge}[1]{%
\ensuremath{\displaystyle{\smash{\,\mathop{\rightharpoonup}\limits_{\epsilon \to0}^{#1}\,}}}}

\usepackage{amsfonts}
\usepackage{amssymb}
\usepackage{amsbsy}
\usepackage{color}
\usepackage{stmaryrd}
\usepackage{psfrag}
\usepackage[hang,center]{caption}
\usepackage{verbatim} 
\usepackage{float}
\usepackage{cite}
\usepackage[lofdepth,lotdepth]{subfig}
\usepackage{comment}

\title[Anisotropic Vlasov equation]{Numerical study of an anisotropic Vlasov equation arising in plasma physics}

\author[B. Fedele, C. Negulescu]{Baptiste Fedele, Claudia Negulescu}

\address{Universit\'e de Toulouse \& CNRS, UPS, Institut de Math\'ematiques de Toulouse UMR 5219, F-31062 Toulouse, France.}
\email{baptiste.fedele@math.univ-toulouse.fr\\
claudia.negulescu@math.univ-toulouse.fr}
\date{\today}

\begin{document}

\maketitle

\begin{abstract}
\noindent 
Goal of this paper is to investigate several numerical schemes for the resolution of two anisotropic Vlasov equations. These two toy-models are obtained from a kinetic description of a tokamak plasma confined by strong magnetic fields. The simplicity of our toy-models permits to better understand the features of each scheme, in particular to investigate their {\it asymptotic-preserving} properties, in the aim to choose then the most adequate numerical scheme for upcoming, more realistic simulations. 
\end{abstract}

\bigskip


%%%%%%%%%%%%%%%%%%%%%%%%%%%%%%%%%%%%%%%%%%%%%%
\section{Introduction} \label{SEC1}
%%%%%%%%%%%%%%%%%%%%%%%%%%%%%%%%%%%%%%%%%%%%%%
The present paper addresses a new approach for an efficient numerical resolution of anisotropic transport models, which simplified are of the type
\be \label{MODEL}
\left\{
\begin{array}{l}
\ds \partial_t f^\epsilon + { \bfu \over \epsilon} \cdot   \nabla f ^\epsilon=0\,, \qquad \forall (t,x,y) \in [0,T] \times \Omega\,,\\[3mm]
\ds f^\epsilon(0,x,y)=f_{in}(x,y)\,,
\end{array}
\right.
\ee
subject to appropriate boundary conditions (here periodic ones). The unknown $f^\epsilon$ stands for the quantity (distribution function) which is advected along the given (or self-consistently computed) field $\bfu$ in the domain $\Omega:=[0,L_x] \times [0,L_y]$ and the small scaling parameter $\epsilon \ll 1$ indicates that we have to deal with very strong advection fields $\bfu$ or equivalently  with the long-time asymptotics of $f^\epsilon$. Such anisotropic transport models arise very often in physics, as simplifications of more complex systems. In Section \ref{SEC10} we detail some examples coming from plasma physics, as the Vlasov equation for the ion dynamics in the gyrokinetic regime. There are however several other examples arising in physics and leading to a simplified transport equation as \eqref{MODEL}, for example when one studies the long-time asymptotics of the incompressible Euler 2D equations, \eqref{MODEL} representing then the vorticity equation, which has to be coupled (via $\bfu$) with a Poisson equation for the stream-function computation \cite{Majda}.\\

A numerical resolution of problems of the type \eqref{MODEL} is rather challenging in the regime $\epsilon \ll 1$, due to the singularity of the mathematical problem as $\epsilon \rightarrow 0$. Certainly, the exact solution of the simple transport-case \eqref{MODEL} is known for $\epsilon >0$, however not in general situations, when $\bfu$ is self-consistently computed via $f^\epsilon$ and when other (not-stiff) terms are present. These general situations require then an efficient numerical treatment of \eqref{MODEL}. From a physical point of view we can say that we have to cope with a multiscale problem, the parameter $\epsilon$ being the stiffness parameter. Standard schemes (explicit hyperbolic approaches) require very restrictive CFL-conditions (dependent on $\epsilon$) in order to accurately account for the microscopic $\epsilon$-scales. Very often in such situations people are impliciting the stiff term \cite{filjin2}, in order to avoid these too restrictive CFL-conditions. This can work in some situations, for example when the grid is aligned with the anisotropy, and only for a certain range of $\epsilon$-values. However in more general configurations, not-aligned grids and $\epsilon$-values covering all the interval $[0,1]$, impliciting the stiff term is no more sufficient, as shall be seen in this paper. We propose thus in this work a new numerical procedure, based on {\it Asymptotic-Preserving} arguments, being able to solve \eqref{MODEL} in an efficient manner, uniformly accurate and stable in $\epsilon$, and this on a simple, Cartesian grid. {\it Asymptotic-Preserving} methods are efficient, as they are designed in order to mimic on the discrete level the asymptotic behavior of the singularly perturbed problem solutions (see \cite{Jin_rev,PE} for a detailed introduction).\\
This paper was initiated by the repetitive remarks/questions one of the authors got during conferences, meaning that impliciting the stiff term in \eqref{MODEL} is enough to get an efficient AP-scheme, which behaves well even in the limit $\epsilon \rightarrow 0$. The aim of this paper is to prove the contrary, AP-schemes are more than impliciting the stiff term. In order to understand in detail the main features of the here proposed AP-scheme, we preferred to keep the investigated model as simple as possible, so that a detailed numerical analysis is possible, permitting to perceive the differences of our scheme when compared to standard (implicit) schemes. We hope that doing so, we are able to resolve some of the confusion that surround AP-schemes. However, even if the here presented results are based on a simplified model as \eqref{MODEL}, the same {\it Asymptotic-Preserving} approach can be used for more involved anisotropic transport problems, such as those presented in Section \ref{SEC10} and which shall be the objective of an upcoming work. \\
The AP-procedure we propose here was employed in other contexts by the authors (elliptic \cite{DDLNN,DLNN}, parabolic \cite{MN}). The present setting is more stimulating, as we have to cope with highly oscillating problems when $\epsilon \ll 1$ and no more dissipative ones. In the present oscillating case, the limit (weak) $\epsilon \rightarrow 0$ is more challenging, and has to be defined adequately. We refer the reader to \cite{NL,NICOO,lm} for other AP-scheme references.\\

This paper is laid as follows. Section \ref{SEC10} deals with the presentation of a physical situation leading, after scaling and simplification,  to the anisotropic transport equation \eqref{MODEL}. Two simplified models which will be studied in the following, are presented. Section \ref{SEC2} reviews the mathematical framework necessary to study the first toy model, and investigates the asymptotic limit $\epsilon \to 0$. Section \ref{SEC21} introduces several numerical schemes that we shall apply for the resolution of the first toy model. Then, we present the numerical results obtained with these schemes in Section \ref{SEC24} and the numerical analysis in Section \ref{SEC25}. The last section is dedicated to the mathematical and numerical study of the second toy model which considers variable coefficients. A conclusion gives some hints for our upcoming work, concerning the more realistic Vlasov equation \eqref{V_C}.

%%%%%%%%%%%%%%%%%%%%%%%%%%%%%%%%%%%%%%%%%%%%%%%
\section{Physical motivation and toy models} \label{SEC10}
%%%%%%%%%%%%%%%%%%%%%%%%%%%%%%%%%%%%%%%%%%%%%%%
Let us shortly say here some words about the physical motivation of the present work and introduce the two simplified models we shall investigate numerically in the next sections. These simplified models are caricatures of typical asymptotic regimes encountered in plasma physics, as for example the gyro-kinetic regime, and contain all the numerical difficulties arising in the more complex real physical systems.\\

The core tokamak plasma can be considered as collisionless, such that  the most appropriate model for the description of its dynamics is the Vlasov equation for each particle species ($\alpha=e$ for electrons and $\alpha=i$ for ions), {\it i.e.} 
\begin{equation} \label{Boltz}
\partial_t f_{\alpha} + \bfv \cdot \nabla_x f_{\alpha} + \frac{e_{\alpha}} {m_\alpha} (\bfE+ \bfv \times \bfB) \cdot \nabla_v f_{\alpha} =0\,,
\end{equation}
where $e_\alpha = \pm e$ resp. $m_\alpha$ are the particle elementary charge resp. mass and $\bfE(t,\bfx)$ resp. $\bfB(t,\bfx)$ are the electric respectively magnetic fields, determined self-consistently from Maxwell's equations. In the electrostatic case (given field $\bfB$), Maxwell's equations have to be replaced by Poisson's equation
\be \label{Poi}
- \epsilon_0 \Delta \Phi = \rho\,, \qquad \rho(t,\bfx):=\sum_{\alpha} e_\alpha \int_{\RR^3} f_{\alpha}(t,\bfx,\bfv)\, d\bfv\,,
\ee
where $\Phi$ is the electrostatic potential, related to the electric field $\bfE$ by $\bfE(t,\bfx)= - \nabla \Phi(t,\bfx)$. For more details about the modelling of magnetically confined fusion plasmas, we refer the interested reader to the textbooks \cite{CHENF,hazeltine_meiss,Ruther}.\\

From a numerical point of view, solving the system \eqref{Boltz}-\eqref{Poi} is rather arduous, due among others to its high dimensionality ($6$ dimensional in the phase space $(\bfx,\bfv)$) and to the presence of several time and space scales in the dynamics, introduced for ex. by the strong magnetic field $\bfB$ which confines the plasma in the tokamak. We shall be concerned in the present work with the multi-scale aspects of the kinetic problem, difficulties which are described mathematically by the following rescaling of the Vlasov equation for the ions (see \cite{mb,Xav,Grand,Max,CN} for the gyrokinetic scaling)
\be \label{V_C}
\partial_t f + \bfv \cdot \nabla_x f + \Big[\bfE + {1
  \over \epsilon} \left( \bfv \times \bfB \right)\Big] \cdot \nabla_v f = 0\,,\ee
where $\epsilon$ stands for the ratio of the particle cyclotron period to the observation time. The electrons experience the appearance of a second small parameter, related to the small electron to ion mass ratio $m_e/m_i$, leading to additional numerical burden, we shall not consider here (see \cite{DNP}).
The effect of the intense magnetic field on the particle dynamics is that it introduces a strong anisotropy, the motion of the charged particles being splitted into a fast gyration around the magnetic field lines and a slow dynamics along these lines, separation which necessarily causes numerical complications.\\

%The aim of the present paper is to propose a new numerical procedure for an efficient resolution of \eqref{V_C} uniformly in $\epsilon$. However, we deliberately simplified equation \eqref{V_C} in order to be able to do a complete numerical analysis and to understand in all details the features of the here introduced AP-schemes.\\
%The two simplified toy-models we shall treat in this paper still contain all the numerical difficulties of the initial model. Based on the conclusions of the present study, we shall pick up the most appropriate numerical scheme to apply it to the initial problem \eqref{V_C}, posed in a realistic physical context. This last part shall be done in a second paper, due to lengthy reasons.\\

Let us introduce now two simplified toy models, which contain all the numerical difficulties of the initial model.
In the rest of this paper we shall consider a homogeneous magnetic field $\bfB= b\, {\EuFrak b} $ with fixed direction ${\EuFrak b}:=e_z$ and constant magnitude $|\bfB| =b\equiv 1$. Furthermore, let us also introduce the following notation
$$
\bfv_{||}= (0,0,v_z)^t\,, \quad
\bfv_{\perp}=(v_x,v_y,0)^t\,, \quad ^\perp \bfv:=(v_y,-v_x,0)^t=\bfv \times \bfB\,.
$$
Sometimes it is more convenient to shift in \eqref{V_C} from Cartesian coordinates to polar coordinates for the velocity, {\it i.e.}
$$
v=(v_x,v_y,v_z) \Leftrightarrow (r,\theta,v_z)\,, \quad \left\{
\begin{array}{l}
v_x:=r \cos(\theta)\\
v_y:=r \sin(\theta)
\end{array}
\right.\,, \quad
\begin{array}{l}
\theta \in [0, 2 \pi)\\
r \ge 0
\end{array}\,.
$$
The Vlasov equation \eqref{V_C}, written in polar coordinates, has then the form
\be \label{V_P}
\begin{array}{l}
\partial_t F + v_z \partial_{z} F + E_z \partial_{v_z} F +\left(E_x
  \cos\theta + E_y \sin\theta\right) \partial_r F- \ds{1
  \over r } \left(E_x
  \sin\theta - E_y \cos\theta  \right)\partial_{\theta} F \\[3mm]
\hspace{4cm} \ds + r \left( \cos\theta \partial_x F +
  \sin\theta \partial_y F \right) - { 1 \over
  \epsilon} \partial_{\theta} F =0\,,
\end{array}
\ee
where the unknown is now $F(t,x,y,z,r,\theta,v_z)$.\\

The two formulations, \eqref{V_C} resp. \eqref{V_P}, corresponding to a Cartesian (not field-aligned) resp. polar (field-aligned) configuration, are different from a numerical point of view, and different numerical schemes are usually employed for their resolution. To understand this difference better, we shall investigate in the present work in detail some numerical schemes for simplified versions of \eqref{V_C} and \eqref{V_P}.  We deliberately simplified these equations in order to be able to do a complete numerical analysis and to understand in all details the features of the here introduced AP-schemes.

%%%%%%%%%%%%%%%%%%%%%%%%%%%%%%%%%%%%%%%%%%%%%%%
\subsection{First toy model - Polar, field-aligned configuration} \label{SEC101}
%%%%%%%%%%%%%%%%%%%%%%%%%%%%%%%%%%%%%%%%%%%%%%%
Let us start from the Vlasov equation \eqref{V_C}, assume here that $\bfE \equiv 0$, $\bfB=e_z$ and consider furthermore only the dynamics in the perpendicular plane $(x,y)$, {\it i.e.}
\begin{equation} \label{Vlas4D}
\partial_t f + \bfv_{\perp} \cdot \nabla_x  f + { 1 \over \epsilon} (\bfv \times \bfB) \cdot \nabla_v f =0\,,
\end{equation}
where $\epsilon \ll 1$ accounts as usual for very strong magnetic fields.
In order to simplify the computations, one often shifts to  polar coordinates for the velocity, leading to 
\be \label{V_P_bis}
\begin{array}{l}
\partial_t F + r\, \cos\theta \,\partial_x F + r\,
  \sin\theta\, \partial_y F  - \ds{ 1 \over
  \epsilon}\, \partial_{\theta} F =0\,,
\end{array}
\ee
where the unknown now is $F(t,x,y,r,\theta)$. We recognize thus a simple $3D$ anisotropic transport equation, the variable $r$ being considered as a parameter in \eqref{V_P_bis}.\\

Choosing an initial condition $F_{in}$ independent on the variable $y$, would even lead to a more simpler $2D$ transport model
\be \label{ADV_2D}
\partial_t F + r\, \cos\theta \,\partial_x F - { 1 \over
  \epsilon}\, \partial_{\theta} F =0\,.
\ee
This problem represents the simplest example of an anisotropic advection equation, to be understood in detail before designing an efficient scheme for the resolution of the Vlasov equation in the gyrokinetic regime \eqref{V_C}. It is sufficiently difficult in order to study the behavior of the various schemes we shall introduce, and shall be the starting point of Section \ref{SEC2}.

%%%%%%%%%%%%%%%%%%%%%%%%%%%%%%%%%%%%%%%%%%%%%%%
\subsection{Second toy model - Cartesian, not field-aligned configuration} \label{SEC102}
%%%%%%%%%%%%%%%%%%%%%%%%%%%%%%%%%%%%%%%%%%%%%%%
In this second part, we shall differently simplify our Vlasov equation in order to study a different behavior. In particular, setting $\bfE \equiv 0$, $\bfB \equiv e_z$ and taking an initial condition independent on the space variable, yields the following 2D equation, in Cartesian coordinates
\begin{equation} \label{Vlas4D}
\partial_t f + { 1 \over \epsilon} (\bfv \times \bfB) \cdot \nabla_v f =0\,,
\end{equation}
or equivalently
\begin{equation} \label{Vlas4D_bis}
\partial_t f + { v_y \over \epsilon} \partial_{v_x} f -{ v_x \over \epsilon} \partial_{v_y} f =0\,.
\end{equation}

\noindent The difference of this model to the previous one is that this time the characteristics are no more straight lines but curves, such that the numerical schemes will behave differently. As mentioned earlier, these two models correspond to simplified versions of a field-aligned, polar coordinate framework
%, as the one used in the GYSELA code \cite{Grand}
, as well as a not field-aligned, Cartesian framework, both associated to the Vlasov equation \eqref{V_C} in the gyro-kinetic regime.

%%%%%%%%%%%%%%%%%%%%%%%%%%%%%%%%%%%%%%%%%%%%%%%
\subsection{Aim of the present paper} \label{SEC103}
%%%%%%%%%%%%%%%%%%%%%%%%%%%%%%%%%%%%%%%%%%%%%%%
The main points we are interested in within this study are the following:
\begin{itemize}
\item design of AP-schemes for an efficient numerical resolution of anisotropic Vlasov equations of type \eqref{ADV_2D}, \eqref{Vlas4D}. Important properties we are asking from the schemes are: (a) stability independent on $\epsilon$; (b) numerical diffusion/accuracy independent on $\epsilon$; (c) discretization of the limit model as $\epsilon \rightarrow 0$; 
\item show that taking the stiff term ${ 1 \over \epsilon} (\bfv \times \bfB) \cdot \nabla_v f$ in \eqref{V_C} implicitly is not sufficient for having an AP-scheme, meaning that AP-schemes are more than taking ``implicitly'' the suitable terms. AP-schemes have to mimic at the discrete level the precise asymptotic behavior of the solution in the limit $\epsilon \rightarrow 0$;
\item perform a detailed numerical analysis of the presented schemes in the framework of the two simplified toy-models \eqref{ADV_2D}, \eqref{Vlas4D} and identify exactly which are the particularities of each scheme and each equation;
\item understand the difference between a field-aligned framework \eqref{ADV_2D} and a Cartesian framework \eqref{Vlas4D}, and this from a numerical point of view;
\item prepare the foundation for a future, more realistic work, dealing with the resolution of the initial Vlasov equation \eqref{V_C} in the gyro-kinetic regime.
\end{itemize}

%Let us underline at this point one important fact. Our final aim is to design an efficient and simple numerical scheme working on a Cartesian grid, grid which has not to be adapted on the direction of the anisotropy field lines or its strength.\\

Finally, let us say some words about {\it Asymptotic-Preserving} schemes. In general, inaccuracy in numerical simulations can result from applying unstable algorithms to well-conditioned problems or stable algorithms to ill-conditioned problems. 
Dealing with singularly-perturbed problems is a hard task, as they are ill-conditioned from the beginning. A standard, stable discretization (implicit in this case) often results in inaccurate results. The essence of AP-procedures is to replace singularly-perturbed problems by equivalent problems, which are regularly perturbed, well-conditioned problems, leading to uniformly accurate results, if stable algorithms are used (AP-approach).
%%%%%%%%%%%%%%%%%%%%%%%%%%%%%%%%%%%%%%%%%%%%%%%
\section{First anisotropic Vlasov toy model and its mathematical study} \label{SEC2}
%%%%%%%%%%%%%%%%%%%%%%%%%%%%%%%%%%%%%%%%%%%%%%%
Let us investigate now in detail the following simplified toy model, corresponding to a field-aligned anisotropic Vlasov equation
\be \label{AV}
(V)_\epsilon\,\,\, 
\left\{
\begin{array}{l}
\ds \partial_t f^\epsilon + a\, \partial_x f^\epsilon + { b \over \epsilon}\,  \partial_y f ^\epsilon=0\,, \quad \forall (t,x,y) \in [0,T] \times [0,L_x] \times [0,L_y]\,,\\[3mm]
\ds f^\epsilon(0,x,y)=f_{in}(x,y)\,,
\end{array}
\right.
\ee
where $f_{in}$ is a given initial condition, $a>0$ and $b>0$ are for the moment constants  and $0<\epsilon \ll 1$ is a parameter representing the strong anisotropy/stiffness of the problem. Our computational domain is a doubly periodic box $\Omega :=[0,L_x] \times [0,L_y]$.\\
We shall review here some  standard numerical schemes as well as introduce some new ones for the resolution of such a singularly perturbed problem and discuss finally their advantages and disadvantages. In particular, one is interested in numerical schemes capable to solve \eqref{AV} uniformly accurate in $\epsilon$, so-called ``Asymptotic-Preserving'' schemes.\\
Let us however start with a detailed mathematical study of the behavior of \eqref{AV}.

%%%%%%%%%%%%%%%%%%%%%%%%%%%%%%%%%%%%%%%%%%%%%%%
\subsection{Singularly perturbed problem} \label{SEC20}
%%%%%%%%%%%%%%%%%%%%%%%%%%%%%%%%%%%%%%%%%%%%%%%
Equation  \eqref{AV} is a simple advection problem, whose exact solution is given by the characteristic method, {\it i.e.}
\be \label{SOL}
f^\epsilon(t,x,y)=f_{in}(x-at,y-{b  \over \epsilon}\, t)\,, \quad \forall (t,x,y) \in [0,T] \times \Omega\,.
\ee
\noindent Remark that this function is $L_x$-periodic in the variable $x$, $L_y$-periodic in the variable $y$. Concerning the time-variable, two time-scales are present in the problem, a slow time-scale $t$ and a rapid one $t/\epsilon$.\\
%For the time variable, if we have $\ds \frac{L_x /a}{L_y \epsilon \, b} = \frac{p_1}{p_2}$ with $(p_1,p_2) \in (\mathbb{N}^{*})^2$, then the previous function is $L_t$-periodic with $L_t = p_1 L_y \epsilon / b = p_2 L_x /a$. As in a computer all numbers are rational (finite number of decimal places), we have in practice always to cope with a $L_t$-periodic function, where $L_t \rightarrow_{\epsilon \rightarrow 0} 0$, meaning that the functions become more and more oscillating in time, as $\epsilon \rightarrow 0$.\\

\noindent The term ${b \over \epsilon} \partial_y f^\epsilon$ in \eqref{AV} is the dominant term in the case where $\epsilon \ll 1$, such that passing formally to the limit $\epsilon \rightarrow 0$, yields
\begin{align}
(R)\,\,\,
\left\lbrace
\begin{array}{ll}
 \partial_y f = 0, \quad \forall(t,x,y) \in [0,T] \times [0, L_x] \times [0,L_y],  \\ \\
f(0,x,y) = f_{in}(x,y).
\end{array}\right.
\label{Red}
\end{align}

\noindent This system, called "reduced system", is ill-posed. Depending on the initial condition, it can admit or an infinite number of solutions, namely if $\partial_y f_{in} =0$, or no regular solution (if $\partial_y f_{in} \neq 0$). From a numerical point of view, this ill-posedness in the limit is translated into the singularity of the matrix of the linear system obtained by discretization of this problem. In particular, trying to solve \eqref{AV} in a standard manner will necessarily lead to a linear system which degenerates in the limit $\epsilon \rightarrow 0$. This shall induce sever numerical problems.\\
More adequate schemes are hence necessary for an efficient resolution of \eqref{AV}, as for example ``Asymptotic-Preserving'' schemes which are uniformly stable and accurate independently on the small parameter $\epsilon$, and are additionally able to capture the limit model as $\epsilon \rightarrow 0$.

%%%%%%%%%%%%%%%%%%%%%%%%%%%%%
\subsection{Limit model} \label{SEC29}
%%%%%%%%%%%%%%%%%%%%%%%%%%%%%
For a better comprehension of our singularly-perturbed problem as well as for the construction of efficient ``Asymptotic-Preserving'' schemes, we have to identify the limit problem $(V)_0$ of \eqref{AV} and its solution denoted by $f^0$. The information we get from the reduced model is that the limit-function $f^0$ has to be $y$-independent. With this information we introduce now the average of the function $f^{\epsilon}$ with respect to the direction $y$
\begin{align*}
\bar{f}^{\epsilon}(t,x) := \frac{1}{L_y} \int_0^{L_y} f^{\epsilon}(t,x,y) dy.
\end{align*}

\noindent Integration of  the equation \eqref{AV} with respect to $y$ yields $\partial_t \bar{f}^{\epsilon} + a \partial_x \bar{f}^{\epsilon} = 0$, which is an $\epsilon$-independent problem. Passing then to the limit $\epsilon \rightarrow 0$ leads to the advection equation
\begin{align}
(V)_{0}\,\,\,
\left\lbrace
\begin{array}{ll}
\ds \partial_t {f}^{0} + a \partial_x {f}^{0} = 0\,, \quad\forall (t,x) \in [0,T] \times [0,L_x], \\[3mm]
\ds {f^{0}}(0,x) = \bar{f}_{in}(x)\,, \quad \forall x \in [0,L_x]\,,
\end{array}\right.
\label{eq4}
\end{align}
with solution 
$$
{f^{0}}(t,x) = \bar{f}_{in}(x-at)\,, \quad \forall (t,x) \in [0,T] \times [0,L_x]\,.
$$
The system $(V)_0$ is what we call ``limit-system'' of the anisotropic Vlasov equation $(V)_\epsilon$, as shall be proved in the next section.

%%%%%%%%%%%%%%%%%%%%%%%%%%%%%
 \subsection{Weak convergence} \label{SEC29}
 %%%%%%%%%%%%%%%%%%%%%%%%%%%%%
 So far, we proved the existence of a unique solution $f^\epsilon$ for the system $(V)_{\epsilon}$ resp. $f^0$  for the limit system $(V)_0$. The next step is now to show the weak-convergence of $f^\epsilon$ towards $f^0$ as $\epsilon \rightarrow 0$, and this in a certain sense. To define this sense, we have to introduce the right mathematical framework. In the sequel the symbol $\sharp$ shall underline the periodicity of the considered space.\\

\begin{theorem} \label{THM_weak}
Let the initial condition $f_{in} \in H^1_\sharp(\Omega)$. Then the unique solutions to $(V)_{\epsilon}$ resp. $(V)_0$ satisfy $f^\epsilon \in W^{1,\infty}(0,T; L^2_\sharp(\Omega)) \cap L^{\infty}(0,T; H^1_\sharp(\Omega))$ resp. $f^0 \in W^{1,\infty}(0,T; L^2_\sharp(0,L_x)) \cap L^{\infty}(0,T; H^1_\sharp(0,L_x))$. Moreover, we have the weak-$\star$ limit
\begin{equation} \label{Weak}
f^\epsilon \converge{*} f^0 \quad \textrm{in} \,\,\, L^{\infty}(0,T; L^2_\sharp(\Omega))\,. 
\end{equation}
\end{theorem}
\begin{proof}
To prove \eqref{Weak}, which signifies
$$
\int_0^T \int_\Omega \left(f^\epsilon(t,x,y) - f^0(t,x) \right)  \phi(t,x,y) \,dx\,dy\,dt \underset{\epsilon \to 0}{\longrightarrow}  0 \quad \forall \phi \in L^1(0,T;L^2_\sharp(\Omega))\,,
$$
we shall introduce first a primitive of the function $f_{in}(x,.) - \bar{f_{in}}(x)$, {\it i.e.} 
 \begin{align*}
g(x,y) := \int_0^y \Big(f_{in}(x,z) - \bar{f_{in}}(x) \Big)dz.
\end{align*}
It follows that the function $g$ belongs to $H^1_\sharp(\Omega)$ such that $g^\epsilon(t,x,y):=g(x-at,y-b\,t/\epsilon)$ belongs to $W^{1,\infty}(0,T; L^2_\sharp(\Omega)) \cap L^{\infty}(0,T; H^1_\sharp(\Omega))$. The $L_y$-periodicity of $g$ is seen by the simple computation
\begin{align*}
g(x,y+L_y) &= \int_0^{y+L_y} \Big(f_{in}(x,z) - \bar{f_{in}}(x) \Big)dz = \int_{-L_y}^y f_{in}(x,z)dz - \bar{f_{in}}(x)(y+L_y) \\
&= \int_0^y f_{in}(x,z)dz - \bar{f_{in}}(x)\, y + \int_{0}^{L_y}f_{in}(x,z)dz - \bar{f_{in}}(x)\, L_y = g(x,y)\,.
\end{align*}
 Taking now an arbitrary test function $\phi \in C^1_0(0,T;L^2_\sharp(\Omega))$ and introducing for simplicity for each $f,g \in L^2_\sharp(\Omega)$ the bracket $\langle f,g \rangle:= \int_\Omega f\, g \, dx\, dy$, we have
 \begin{align*}
&\int_0^T \langle f_{in} (x-at,y-\frac{b}{\epsilon}t) - \bar{f}_{in}(x-at), \phi(t)\rangle dt = \int_0^T \langle \big(\partial_y g\big)\big(x-at,y-\frac{b}{\epsilon}t\big), \phi(t) \rangle dt  \\
&= - \frac{\epsilon}{b} \Bigg[ \int_0^T \langle \partial_t \Big[g\big(x-at,y-\frac{b}{\epsilon}t\big)\Big]+a  \big(\partial_x g\big)\big(x-at,y-\frac{b}{\epsilon}t\big),  \phi(t)\rangle dt  \Bigg] \\
&=  \frac{\epsilon }{b}  \int_0^T \langle  g\big(x-at,y-\frac{b}{\epsilon}t\big),  \phi '(t)\rangle  dt - \frac{\epsilon a}{b}\int_0^T \langle   \big(\partial_x g \big)\big(x-at,y-\frac{b}{\epsilon}t\big), \phi(t) \rangle dt .
\end{align*}

\noindent As $g^\epsilon \in W^{1,\infty}(0,T; L^2_\sharp(\Omega)) \cap L^{\infty}(0,T; H^1_\sharp(\Omega))$, we can estimate
\begin{align*}
\forall \phi \in C^1_0(0,T;L^2_\sharp(\Omega)),  \; \;  & \left| \int_0^T \langle \Big[ f_{in} (x-at,y-\frac{b}{\epsilon}t) - \bar{f}_{in}(x-at) \Big],  \phi(t) \rangle dt \right| \leqslant C \epsilon\,,
\end{align*} 
\noindent where $C>0$ is a constant independent on $\epsilon$. Therefore,
\begin{align*}
\forall \phi \in C^1_0(0,T;L^2_\sharp(\Omega)),  \; \;  &\int_0^T \langle \Big[ f_{in} (x-at,y-\frac{b}{\epsilon}t) - \bar{f}_{in}(x-at) \Big], \phi(t) \rangle dt \underset{\epsilon \to 0}{\longrightarrow} 0\,,
\end{align*}
which concludes the proof due to the dense injection $C^1_0(0,T;L^2_\sharp(\Omega)) \subset L^1(0,T;L^2_\sharp(\Omega))$.
\end{proof}
%%%%%%%%%%%%%%%%%%%%%%%%%%%%%%%%%%%%%%%%%%%%%%%
\section{Numerical schemes for the anisotropic Vlasov equation} \label{SEC21}
%%%%%%%%%%%%%%%%%%%%%%%%%%%%%%%%%%%%%%%%%%%%%%%
In this section we shall now introduce several numerical schemes for the resolution of \eqref{AV} and examine them in more details. Firstly, different time semi-discretizations will be presented and then some words mentioned about a standard upwind space-discretization. The time-discretization is the most important step in the construction of  AP-schemes.\\
For this, let us first introduce the following homogeneous discretizations of our time interval $[0,T]$ as well as of our simulation domain $\Omega=[0,L_x] \times [0,L_y]$ :
\be
\begin{array}{lll}\label{discret}
\Delta t&:= T/N_t\,, \,\,\, N_t \in \NN\,; &\quad t_n:=n* \Delta t\,, \,\,\, n=0, \cdots, N_t\\[3mm]
\Delta x&:= L_x/(N_x-1)\,, \,\,\, N_x \in \NN\,;& \quad x_i:=(i-1)* \Delta x\,, \,\,\, i=1, \cdots, N_x \; \; \; \; \; \; \; \; \\[3mm] 
\Delta y&:= L_y/(N_y-1)\,, \,\,\, N_y \in \NN\,;& \quad y_j:=(j-1)* \Delta y\,, \,\,\, j=1, \cdots, N_y\,.
\end{array}
\ee
We denote by $Q_h$ the index domain  $Q_h := [0,N_t] \times [1,N_x] \times [1,N_y] \subset \mathbb{N}^3$. We shall denote further by $f^{\epsilon,n}$ resp. $f^{\epsilon,n}_{ij}$ the numerical approximation of $f^{\epsilon}(t^n,x,y)$ resp. $f^{\epsilon}(t^n,x_i,y_j)$. Recall also that we consider a doubly-periodic framework, such that
$$
f^{\epsilon,n}_{0,j}=f^{\epsilon,n}_{N_x-1,j}\,, \quad f^{\epsilon,n}_{1,j}=f^{\epsilon,n}_{N_x,j}\,, \quad f^{\epsilon,n}_{i,0}=f^{\epsilon,n}_{i,N_y-1}\,, \quad f^{\epsilon,n}_{i,1}=f^{\epsilon,n}_{i,N_y}\,, \quad \forall (n,i,j) \in Q_h\,.
$$
%%%%%%%%%%%%%%%%%
\subsection{Semi-discretization in time }
%%%%%%%%%%%%%%%%%
\subsubsection{ IMEX scheme}\label{211}
%%%%%%%%%%%%%%%%%
The first time semi-discretization we shall study will be the implicit-explicit (IMEX) Euler method, where the stiff term is taken implicitly, {\it i.e}
\be \label{EI}
(IMEX)_\epsilon \quad\,\, {f^{\epsilon,n+1}-f^{\epsilon,n} \over \Delta t} + a\, \partial_x f^{\epsilon,n} + {b \over \epsilon}\, \partial_y f^{\epsilon,n+1}=0\,, \quad \forall n \ge 0\,.
\ee
To study the behavior of this scheme, as $\epsilon$ becomes smaller, let us formally let $\epsilon$ go to zero in \eqref{EI} and get
\begin{align*}
\partial_y f^{0,n+1}(x,y) = 0\,, \quad \forall (x,y)\in \Omega\,.
\end{align*}
This equation admits an infinite amount of solutions, namely all periodic functions dependent only on $x$. This formal analysis permits hence to conclude that the IMEX scheme can not be an AP-scheme, as it does not capture correctly the asymptotic behavior of the problem, which is rather given by the limit problem $(V)_0$. This property shall be tested numerically in Section \ref{SEC24}.

%%%%%%%%%%%%%%%%%
\subsubsection{Fourier method/Micro-Macro method}\label{212}
%%%%%%%%%%%%%%%%%
A different way to solve \eqref{AV} is to use a partial Fourier transform in the variable $y$, which is possible here, as we are in a simplified periodic context with constant coefficients. Denoting indeed the Fourier coefficients by
$$
\hat{f^\epsilon_k} (t,x):= { 1 \over L_y} \int_0^{L_y} f^\epsilon(t,x,y)\, e^{-  {\mathbf i} \, \omega_y\, k\, y}\, dy\,, \quad \forall k \in \ZZ\,, \qquad \omega_y:={2 \, \pi \over L_y}\,,
$$
one has
\be \label{FTR}
f^\epsilon(t,x,y)= \sum_{k=-\infty}^\infty \hat{f^\epsilon_k} (t,x)\, e^{ {\mathbf i}\, \omega_y\,  k \,y}\,,
\ee
where the Fourier coefficients are solutions of the system
\be \label{Fourier}
\left\{
\begin{array}{ll}
\ds \partial_t \widehat{f^\epsilon_0} + a \; \partial_x \widehat{f^\epsilon_0} =0\,, &\quad \forall (t,x) \in [0,T] \times [0,L_x]\\[3mm]
\ds \partial_t \widehat{f^\epsilon_k} + a \; \partial_x \widehat{f^\epsilon_k} + {\mathbf i} \; \omega_y k \; {b \over \epsilon}  \widehat{f^\epsilon_k}= 0 \,, &\quad \forall k \neq 0\,, \,\,\, \forall (t,x) \in [0,T] \times [0,L_x]\,.
\end{array}
\right.
\ee
A simple discretization of this problem can be
$$
(F)_\epsilon \quad\,\, \left\{
\begin{array}{l}
\ds {\widehat{f^{\epsilon,n+1}_0}-\widehat{f^{\epsilon,n}_0} \over \Delta t} + a \; \partial_x \widehat{f^{\epsilon,n}_0} =0\,, \quad \forall n \ge 0\\[3mm]
\ds {\widehat{f^{\epsilon,n+1}_k}-\widehat{f^{\epsilon,n}_k} \over \Delta t} + a \; \partial_x \widehat{f^{\epsilon,n}_k} + {\mathbf i} \; \omega_y k {b \over \epsilon}  \; \widehat{f^{\epsilon,n+1}_k}= 0 \,, \quad \forall k \neq 0\,, \,\, \forall n \ge 0\,.
\end{array}
\right.
$$
Solving this system and using the inverse Fourier transform \eqref{FTR} permits to get the desired result, {\it i.e.} the values of the unknowns $f^{\epsilon,n}_{ij}$, solution of \eqref{AV}.\\ 

\noindent Let us investigate now  the behavior of this system when $\epsilon \rightarrow 0$. Formally we get
$$
(F)_0\quad\,\, \left\{
\begin{array}{l}
\ds {\widehat{f^{\epsilon,n+1}_0}-\widehat{f^{\epsilon,n}_0} \over \Delta t} + a \; \partial_x \widehat{f^{\epsilon,n}_0} =0\,, \quad \forall n \ge 0\\[3mm]
  \widehat{f^{\epsilon,n+1}_k}= 0 \,, \quad \forall k \neq 0\,, \,\, \forall n \ge 0\,.
\end{array}
\right.
$$

\noindent Therefore, we find a discretized version of the Vlasov limit problem $(V)_0$,  signifying that this method will be ``Asymptotic-Preserving''. \\

\noindent The Fourier method is very nice, however it can be applied only in a simplified periodic framework with constant coefficients. As a sort of generalization one can think at the micro-macro method \cite{NL}, which is based on the decomposition of each quantity in its mean part over the variable $y$, denoted by $H^\epsilon$ or simply $\bar{f^\epsilon}$, and the fluctuation part $h^\epsilon$ or simply $({f^\epsilon})'$, defined as follows
$$
H^\epsilon(t,x):= { 1 \over L_y} \int_0^{L_y} f^\epsilon(t,x,y)\, dy\,, \quad h^\epsilon(t,x,y):= f^\epsilon(t,x,y)- H^\epsilon(t,x)\,, \quad \bar{h^\epsilon}=0\,.
$$
Taking now the average of the advection equation \eqref{AV} over $y$ and subtracting the resulting equation then from the initial one, yields a system to be solved for the unknowns $(H^\epsilon,h^\epsilon)$, {\it i.e.}
\be \label{MM}
(MM)_\epsilon \quad \left\{
\begin{array}{ll}
\ds \partial_t H^\epsilon + a \partial_x H^\epsilon=0\,, & \quad \forall (t,x) \in [0,T] \times [0,L_x]\\[3mm] 
\ds \partial_t h^\epsilon + a \partial_x h^\epsilon + {b \over \epsilon} \partial_y h^\epsilon =0\,, & \quad \forall (t,x,y) \in [0,T] \times \Omega\\[3mm] 
\ds \bar{ h^\epsilon}=0\,, & \quad \forall (t,x) \in [0,T] \times [0,L_x]\,.
\end{array}
\right.
\ee

\noindent Let us study now  the behavior of this system when $\epsilon \rightarrow 0$. We have  formally 
\be \label{MM}
(MM)_0 \quad \left\{
\begin{array}{ll}
\ds \partial_t H^0 + a \partial_x H^0=0\,, & \quad \forall (t,x) \in [0,T] \times [0,L_x]\\[3mm] 
\partial_y h^0 =0\,, & \quad \forall (t,x,y) \in [0,T] \times \Omega\\[3mm] 
\ds \bar{ h^0}=0\,, & \quad \forall (t,x) \in [0,T] \times [0,L_x]\,.
\end{array}
\right.
\ee
\noindent The two last equations establish that $h^0 \equiv 0$. Hence the system $(MM)_0$ is nothing else than the Vlasov limit system $(V)_0$.  Again, we have created a scheme which is a regular perturbation of the asymptotic limit model, and shall be hence ``Asymptotic-Preserving''. 

This method is rather similar to Fourier method, however more general, as it can be applied in rather broad contexts. To understand this similitude, remark that $H^\epsilon$ is nothing else than the first Fourier coefficient $\widehat{f^\epsilon_0}$ and the fluctuation $h^\epsilon$ regroups the remaining Fourier modes. However, there is still a disadvantage or difficulty, namely the implementation of the constraint $\bar{ h^\epsilon}=0$, which is crucial for the passage to the limit $\epsilon \rightarrow 0$. It is this constraint which permits in the limit to get a unique $h^0$ and to have thus a well-posed limit problem $(MM)_0$. But averaging  along the anisotropy lines can be very difficult in more general contexts, for ex. when these lines are not aligned with the axes.
%%%%%%%%%%%%%%%%%
\subsubsection{Lagrange-multiplier method}\label{214}
%%%%%%%%%%%%%%%%%
The Lagrange-multiplier method is based on the idea to replace the stiff, dominant term ${b \over \epsilon} \partial_y f$ by a smoother one $\partial_y q$, yielding the system
\be \label{La}
(La)_\epsilon \quad \left\{
\begin{array}{l}
\ds \partial_t f^\epsilon + a \; \partial_x f^\epsilon + b \; \partial_y q^\epsilon=0\,, \quad \forall (t,x,y) \in   [0,T] \times \Omega \\[3mm]
\ds  \partial_y f^\epsilon = \epsilon \; \partial_y q^\epsilon\,, \quad \forall (t,x,y)\in   [0,T] \times \Omega \\[3mm]
\ds q^\epsilon_{| \Gamma_{in}}=0\,,
\end{array}
\right.
\ee
where the inflow boundary is defined as $\Gamma_{in}:=\{ (x,y) \in \partial \Omega\,\, / \,\, y=0 \}$. In the limit $\epsilon \rightarrow 0$ one remarks that $q^\epsilon$ is a sort of Lagrange multiplier corresponding to the constraint $\partial_y f^0=0$, where the name of the method.

\bigskip

\noindent First, we will prove the equivalence between the system $(La)_{\epsilon}$ and the Vlasov equation $(V)_{\epsilon}$, proving thus the well-posedness of the reformulation $(La)_{\epsilon}$. For this, let us first consider the unique solution $f^{\epsilon}$ of $(V_{\epsilon})$ and prove the existence of a function $q^{\epsilon}$ such that $(f^{\epsilon},q^{\epsilon})$ solves $(La)_{\epsilon}$. Since $f_{in} \in H_{\sharp}^1(\Omega)$, we have $f^{\epsilon} \in \mathcal{V} := W^{1,\infty}(0,T; L^2_\sharp(\Omega)) \cap L^{\infty}(0,T; H^1_\sharp(\Omega))$. The kernel of the dominant operator ${b \over \epsilon} \partial_y f$, denoted by $\mathcal{G}$, reads:
\begin{align*}
\mathcal{G} := \{ f^{\epsilon} \in \mathcal{V}, \ \partial_y f^{\epsilon} = 0 \}.
\end{align*}
\noindent Then we shall decompose $f^{\epsilon}$ in the following manner, which is somehow similar to a Hilbert Ansatz : 
\begin{align}
f^{\epsilon} = p^{\epsilon} + \epsilon q^{\epsilon},
\label{decompo}
\end{align}
\noindent with $(p^{\epsilon},q^{\epsilon}) \in \mathcal{G} \times \mathcal{V}$. To have a unique decomposition, we have to single out the $\mathcal{G}$-part of $q^\epsilon$, by fixing for example  $q^\epsilon$ on the inflow boundary $\Gamma_{in}$, choosing $q^\epsilon \in {\mathcal Q}$ with
\begin{align*}
\mathcal{Q} := \{ q^{\epsilon} \in \mathcal{V}, \ q^{\epsilon}_{|_{\Gamma_{in}}}=0 \}.
\end{align*}
Obviously, we have $\mathcal{G} \cap \mathcal{Q} = \{0_{\mathcal{V}}\}$, implying the uniqueness of the decomposition \eqref{decompo}. Replacing now this decomposition in the system $(V)_{\epsilon}$, we obtain directly the system $(La)_{\epsilon}$, which proves the existence of a solution to $(La)_{\epsilon}$. The converse is trivial, meaning that for $(f^{\epsilon},q^{\epsilon}) \in \mathcal{V} \times \mathcal{Q}$ solution to $(La)_{\epsilon}$, $f^{\epsilon}$ solves $(V)_{\epsilon}$. Altogether, we have proved the equivalence between both systems. 

\bigskip

\noindent Now let us consider the limit problem of $(La)_{\epsilon}$, obtained by letting formally $\epsilon \to 0$ in \eqref{La}
\be \label{La2}
(La)_0 \quad \left\{
\begin{array}{l}
\ds \partial_t f^0 + a \; \partial_x f^0 + b \; \partial_y q^0=0\,, \quad \forall (t,x,y)\in  [0,T] \times \Omega\\[3mm]
\ds  \partial_y f^0 = 0, \quad \forall (t,x,y)\in  [0,T] \times \Omega\\ [3mm]
\ds q^0_{| \Gamma_{in}}=0\,.
\end{array}
\right.
\ee

\noindent The second equation leads to $f^0 = \bar{f^0}$. Then, averaging the first equation of \eqref{La2} in the $y$-variable, yields
\begin{align}
\partial_t \bar{f^0} + a \; \partial_x \bar{f^0} = 0,
\end{align}
\noindent where we used that $q^0$ is $L_y$-periodic. This equation permits the determination of the limit function $f^0$. Furthermore, the remaining well-posed system 
\be \label{La3}
\quad \left\{
\begin{array}{l}
\ds  b \partial_y q^0= -\partial_t f^0 - a \; \partial_x f^0 \,, \quad \forall (t,x,y) \in [0,T] \times \Omega \\[3mm]
\ds  q^0_{| \Gamma_{in}} = 0 \,, \quad \forall (t,x) \in \times [0,T] \times [0,L_x],
\end{array}
\right.
\ee
can be solved to assure finally the existence of the unique solution $(f^0,q^0)$ for the limit problem $(La)_0$.\\
The Lagrangian scheme seems to be the most ``far-reaching'' AP-scheme . The only disadvantage of this method is that we have now two unknowns and hence two equations to be solved, meaning longer simulation times. However, we are no more forced to follow the anisotropy lines and can choose coarse Cartesian, not-field aligned grids.

%%%%%%%%%%%%%%%%%
\subsection{Space discretization for the IMEX scheme }
%%%%%%%%%%%%%%%%%%
For any numerical scheme presented above, we decided to consider the standard upwind method to discretize the transport terms in the equation \eqref{AV}. The idea behind this choice is that the space-discretization is not the important step in the construction of an AP-scheme, such that we opted for a simple discretization, in order not to embroil the further numerical analysis as well as the understanding of the main ideas of our methods. The same arguments incited us to select only first order discretizations in time. A Runge-Kutta coupled to a second-order space-discretization would be naturally more accurate, changes however nothing in the essential concept of our AP-strategies. As mentioned earlier, in a forthcoming paper we shall be concerned with a realistic, fusion plasma situation, such that we shall adapt the most adequate of the here presented schemes to more accurate second order techniques, to gain in accuracy.\\
Now, let us recall the first-order upwind forms

$$
a \; \partial_x f^{\epsilon,n}_{i,j} \approx a\, \frac{f^{\epsilon,n}_{i,j}-f^{\epsilon,n}_{i-1,j}}{\Delta x},  \; \text{if} \quad a>0, \quad a\, \partial_x f^{\epsilon,n}_{i,j} \approx a \; \frac{f^{\epsilon,n}_{i+1,j}-f^{\epsilon,n}_{i,j}}{\Delta x}, \;  \text{if} \; \; a<0, \quad \forall (n,i,j) \in Q_h.
$$
We have analogous formulae for the partial derivative in the $y$-variable. 
Denoting now  $\ds \alpha := \frac{a \Delta t}{\Delta x} > 0$ and $ \ds \beta := \frac{b \Delta t }{ \Delta y} >0$ and using the periodicity, {\it i.e.}
$$
f^{\epsilon,n}_{0,j}=f^{\epsilon,n}_{N_x-1,j}\,, \quad f^{\epsilon,n}_{1,j}=f^{\epsilon,n}_{N_x,j}\,, \quad f^{\epsilon,n}_{i,0}=f^{\epsilon,n}_{i,N_y-1}\,, \quad f^{\epsilon,n}_{i,1}=f^{\epsilon,n}_{i,N_y}\,, \quad \forall (n,i,j) \in Q_h\,.
$$
the completely discretized IMEX scheme writes finally : 
\begin{align*}
(IMEX)_{\epsilon} \quad \quad
(\epsilon + \beta) f_{i,j}^{\epsilon,n+1} - \beta f_{i,j-1}^{\epsilon,n+1} = \epsilon(1- \alpha) f_{i,j}^{\epsilon,n} + \epsilon \alpha f_{i-1,j}^{\epsilon,n}, 
\end{align*}
\noindent for all $(n,i,j) \in [0, N_t-1] \times [1,N_x-1] \times [1,N_y-1]$. We remark that we can rewrite this scheme like a system of $N_x-1$ equations : 
\be \label{SYST_IM}
\mathcal{A } \;  \mathcal{F}_i^{n+1} = \mathcal{B}_i^n, \; \; \forall n \geqslant 0, \quad \forall i \in [1,N_x-1],
\ee
\noindent where : 
\begin{align*}
 \mathcal{A} = \left(\begin{array}{ccccc}(\epsilon+\beta) & 0 & \hdots & 0 & - \beta \\- \beta & \ddots & 0 & 0 & 0 \\0 & \ddots  & \ddots & 0 & 0 \\0 & 0 & \ddots & \ddots & 0 \\0 & 0 & 0 & -\beta & (\epsilon + \beta)\end{array}\right), \; \;\quad \mathcal{F}^{n+1}_i = \left(\begin{array}{c}f^{\epsilon,n+1}_{i,1} \\f^{\epsilon,n+1}_{i,2} \\\vdots \\f^{\epsilon,n+1}_{i,N_y-2} \\ f^{\epsilon,n+1}_{i,N_y-1}\end{array}\right), 
 \end{align*}
 
 \begin{align*}
\mathcal{B}_i^n =  \left(\begin{array}{c}\epsilon(1- \alpha)f_{i,1}^{\epsilon,n} + \epsilon \alpha f^{\epsilon,n}_{i-1,1} \\\epsilon(1- \alpha)f_{i,2}^{\epsilon,n} + \epsilon \alpha f^{\epsilon,n}_{i-1,2} \\\vdots \\ \epsilon(1- \alpha)f_{i,N_y-2}^{\epsilon,n} + \epsilon \alpha f^{\epsilon,n}_{i-1,N_y-2} \\ \epsilon(1-\alpha)f^{\epsilon,n}_{i,N_y-1}+ \epsilon  \alpha f^{\epsilon,n}_{i-1,N_y-1}\end{array}\right)\,.
 \end{align*}
At each time step, we resolve this system $ \forall i \in [1,N_x-1]$, to get the unknowns $f^{\epsilon,n+1}_{i,j}$. Remark that $\mathcal{A}=\epsilon \, Id + \mathcal{C}_\beta$ is a regular perturbation of a singular, cyclic matrix $\mathcal{C}_\beta$.
%%%%%%%%%%%%%%%%%%%%%%%%%%%%%%%%%%%%%%%%%%%%%%%
\section{Numerical simulations} \label{SEC24}
%%%%%%%%%%%%%%%%%

In this part, we shall test numerically every scheme introduced in the previous Section for the resolution of the anisotropic Vlasov equation \eqref{AV}. The homogeneous time and phase-space discretization was previously introduced in \eqref{discret}  and we choose in the sequel the following parameters: $T=1$, $L_x = 2\pi$, $L_y= 2\pi$, $N_t = 101$, $N_x = N_y = 201$, $a=0.1$ and $b=1$. Changes in these parameters shall be explicitly mentioned. The initial condition we adopt is given by :
\begin{align*}
 f_{in}(x,y) := \sin( x) \big(\cos(2 y)+1 \big), \quad \forall (x,y) \in \Omega:=[0,L_x] \times [0,L_y].
 \end{align*}
\noindent We recall that the exact solution of \eqref{AV} is known and reads, for each $\epsilon >0$:
\begin{align*}
f^{\epsilon}_{ex}(t, x, y) =   \sin \big(x-a t\big) \Bigg[\cos \Big(2 \Big(y- \frac{b}{\epsilon}t\Big)\Big)+1 \Bigg], \quad \forall (t,x,y) \in [0,T] \times \Omega.
\end{align*}
In Figure \ref{figure1}, we reveal two graphics which contain on the one hand $f_{in}$ and on the other hand $f^{\epsilon}_{ex}$ at the final time $T=1$.

\begin{figure}[ht]
\centering
	\subfloat[$f_{in}(x,y)$]{
      \includegraphics[width=0.45\textwidth]{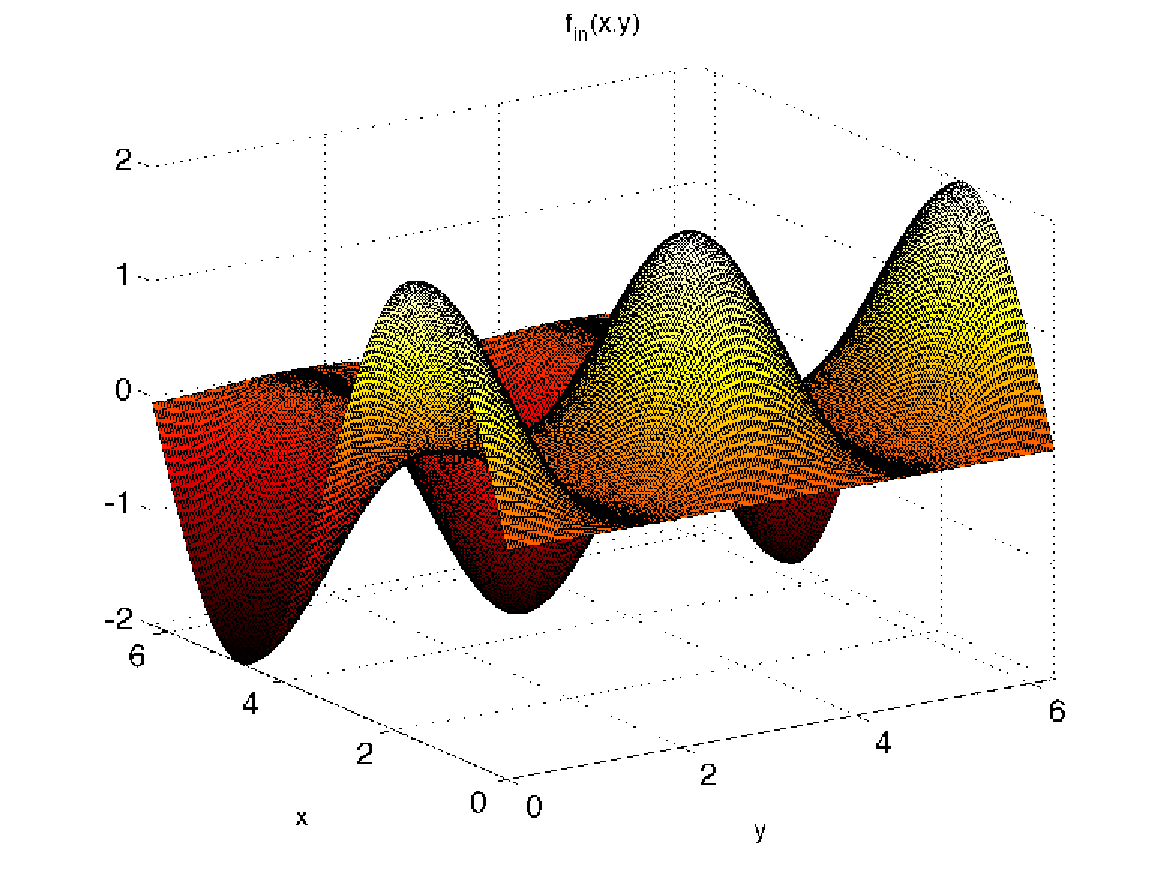}}
      \label{f0}    
	\subfloat[$f^{\epsilon}_{ex}(T,x,y)$]{
      \includegraphics[width=0.45\textwidth]{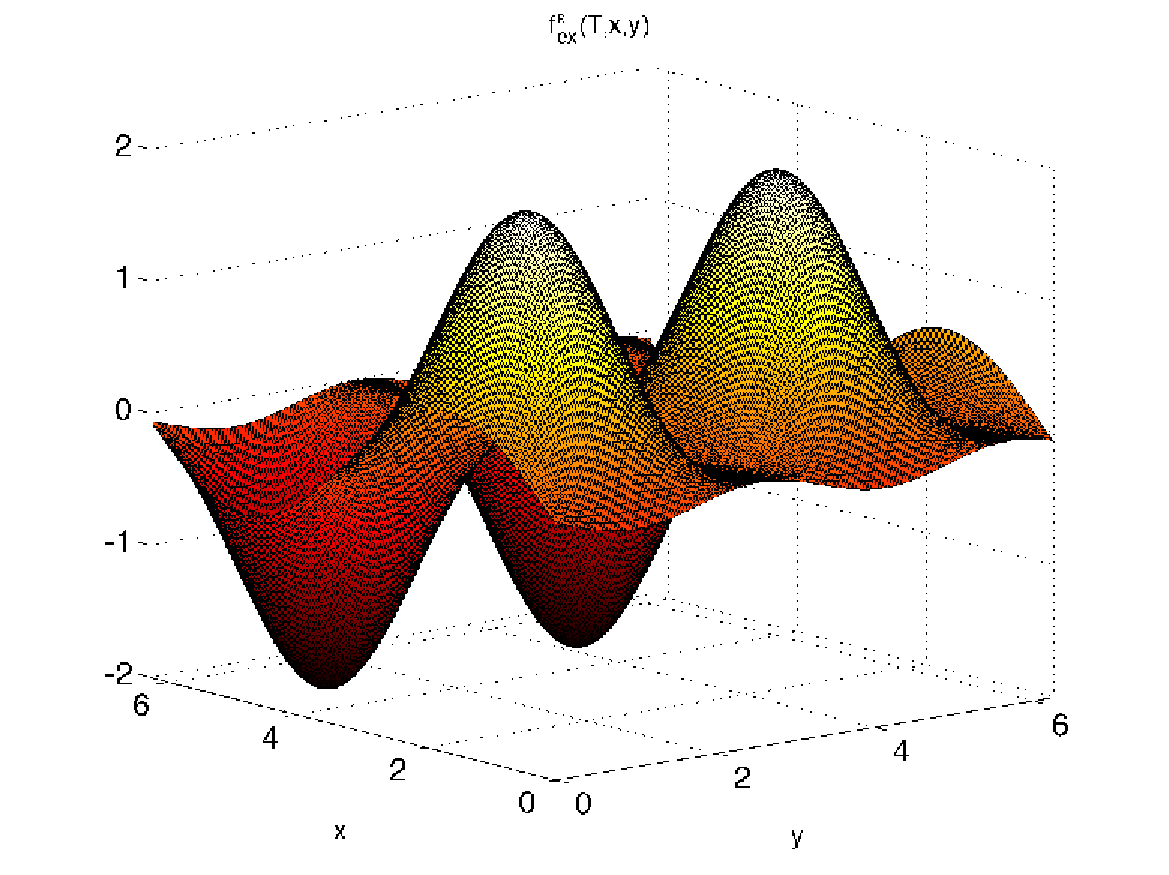}}
      \label{f_ana} 
\caption{\small{Representation of the initial condition $f_{in}$ (A) and the exact solution $f_{ex}^{\epsilon}$ at the final time $T=1$ (B). Here $\epsilon =1$.}}
\label{figure1}
\end{figure}

\noindent Furthermore, in order to better figure out our problem, we plotted in Figure \ref{figure2} the exact solution of the limiting Vlasov system \eqref{eq4} at the final time $T$, i.e. $\ds f^{0}_{ex}(T,x) = \bar{f_{in}}(x-aT)$. Remark that this solution is homogeneous in the $y$-variable. 

\bigskip

\noindent Finally, we show in Figure \ref{figure_fluctu} the time-evolution of the exact solution $f^{\epsilon}_{ex}$ at one point only, {\it i.e.} $(x_{N_x-1}, y_{N_y-1})$. We distinguish easily on the left plot (A) of Fig. \ref{figure_fluctu} the two periods, one linked with the $x$-variable, and the other one corresponding to the $y$-variable. This last one is $\epsilon$-dependent and we see that more $\epsilon$ is small, more the frequency of the time-oscillations becomes important. As the 2D situation is not so eloquent, we eliminate the $x$-variable in the problem and considered also a 1D problem, keeping only the term containing the parameter $\epsilon$ ({\it i.e.} $a=0$). The time-evolution of the exact solution at the point $y_{N_y-1}$ is now plotted in Fig. \ref{figure_fluctu} (B). One observes here more easily that with smaller becoming $\epsilon$, the frequency of the time-oscillations is increasing. In the limit $\epsilon \to 0$, $f^{\epsilon}(t,y_{N_y-1})$ converges weakly towards the average, which is here the constant $1$.

\begin{figure}[ht]
\centering
	  {
      \includegraphics[width=0.45\textwidth]{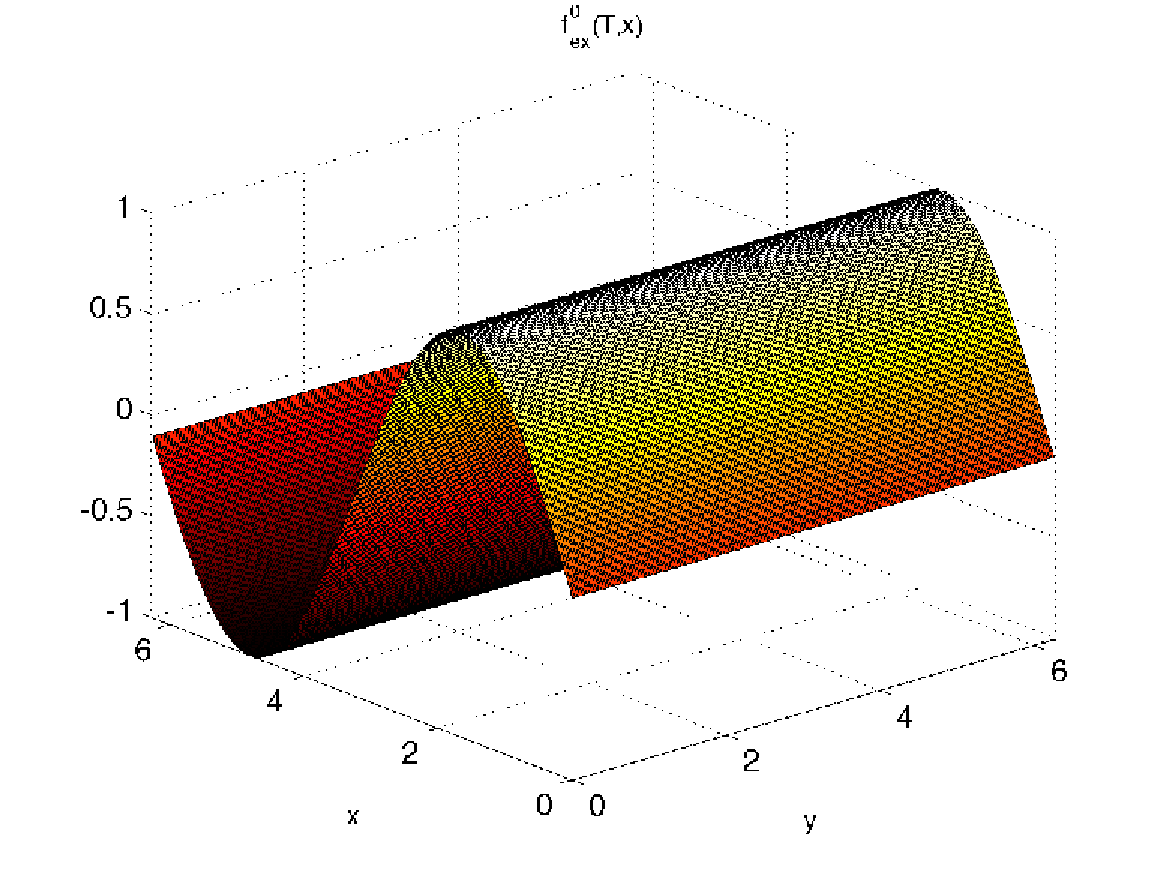}}
      \label{f0}    
\caption{\small{Representation of the exact limit solution $f^0_{ex}(t,x)$ at the final time $T$}. }
\label{figure2}
\end{figure}

\begin{figure}[ht]
\centering
	\subfloat[$f^{\epsilon}_{ex}(t,x_{N_x-1},y_{N_y-1})$]{
      \includegraphics[width=0.45\textwidth]{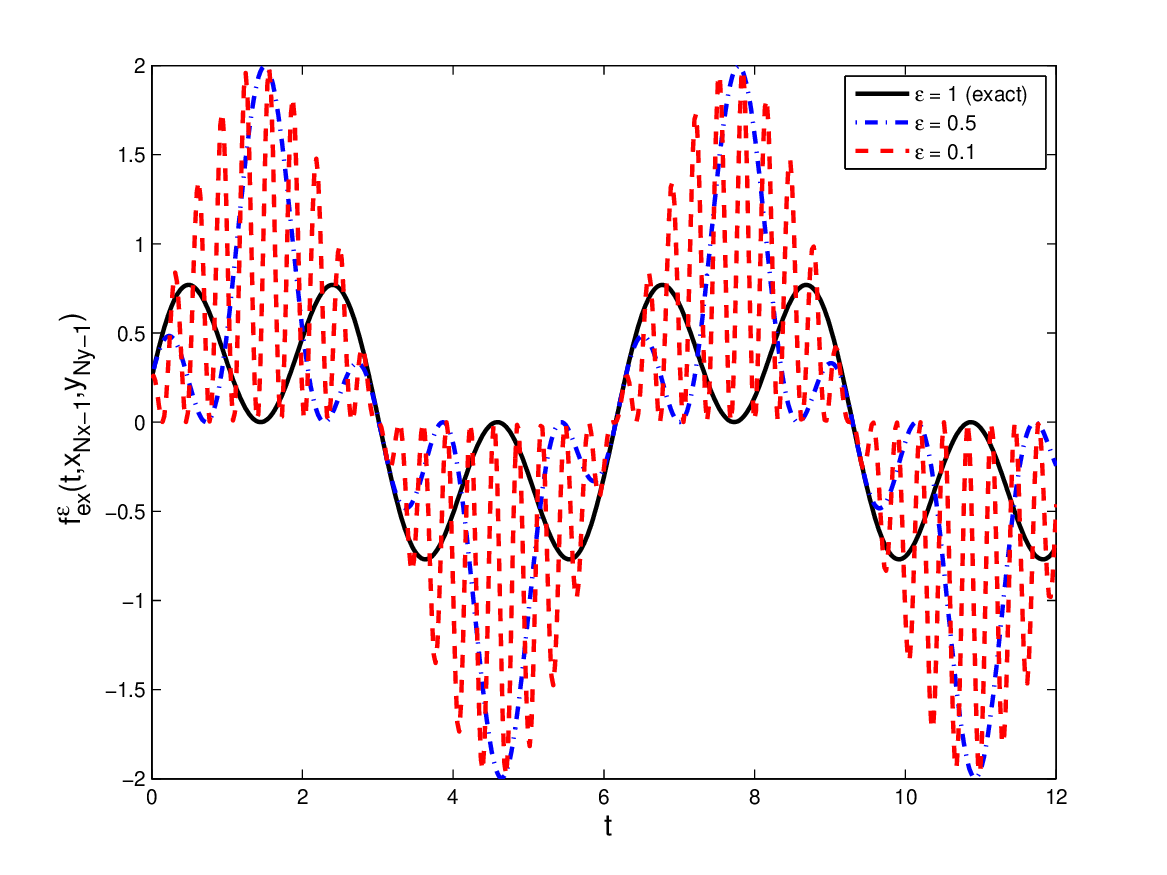}}
      \label{fluctu_ana}    
	\subfloat[$f^{\epsilon}_{ex}(t,y_{N_y-1})$]{
      \includegraphics[width=0.45\textwidth]{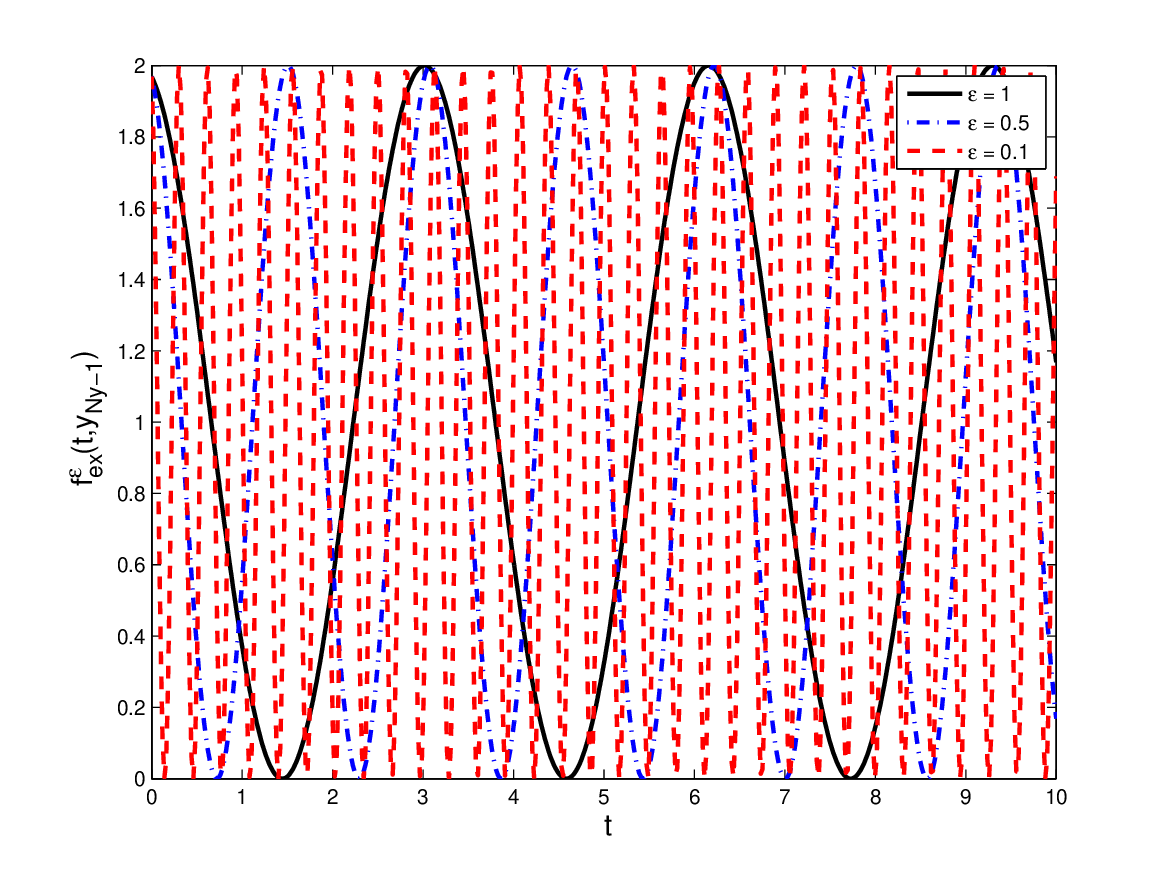}}
      \label{f_ana} 
\caption{\small{Time-evolution of the exact solution at point $(x_{N_x-1},y_{N_y-1})$ in the two dimensional case (A) with  $T=12$ and $N_t = 501$ ; resp. at point $y_{N_y-1}$ in the one dimensional case with $T=10$, $a=0$ and $N_t = 501$ (B). }}
\label{figure_fluctu}
\end{figure}
%%%%%%%%%%%%%%%%%
\subsection{Some results obtained with our schemes }\label{241}
%%%%%%%%%%%%%%%%
Now we examine how the different numerical schemes introduced above cope with such an asymptotic behavior.

\subsubsection{ IMEX scheme} We start by first showing in Fig. \ref{figure3} as well as in the left plot of  Fig. \ref{figure4} the numerical solution $f^{\epsilon}$ via the IMEX-scheme, for three different values of $\epsilon$, namely $\epsilon=1$, $\epsilon=0.1$ and $\epsilon=10^{-10}$, all of them at the final time $T=1$. 

\begin{figure}[ht]
\centering
	\subfloat[$\epsilon=1$]{
      \includegraphics[width=0.45\textwidth]{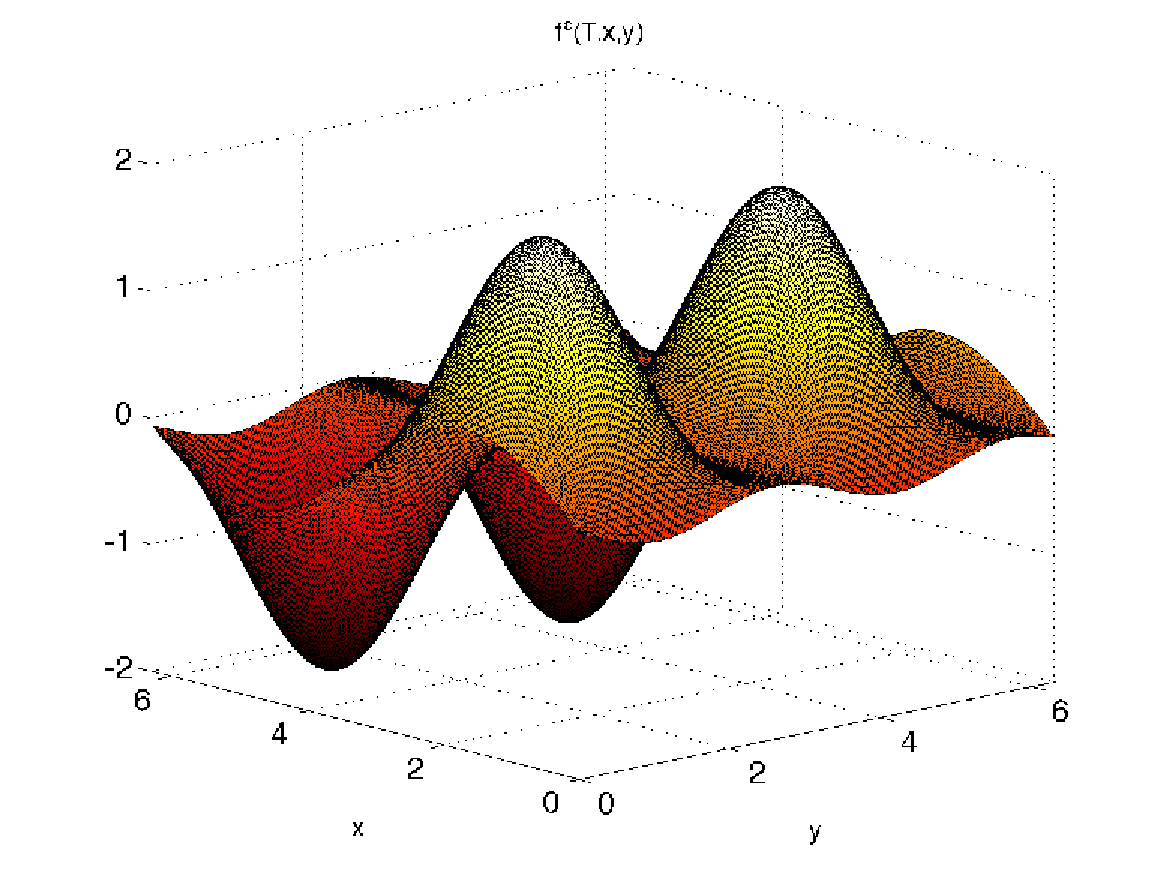}}
      \label{f0}    
	\subfloat[$\epsilon=10^{-1}$]{
      \includegraphics[width=0.45\textwidth]{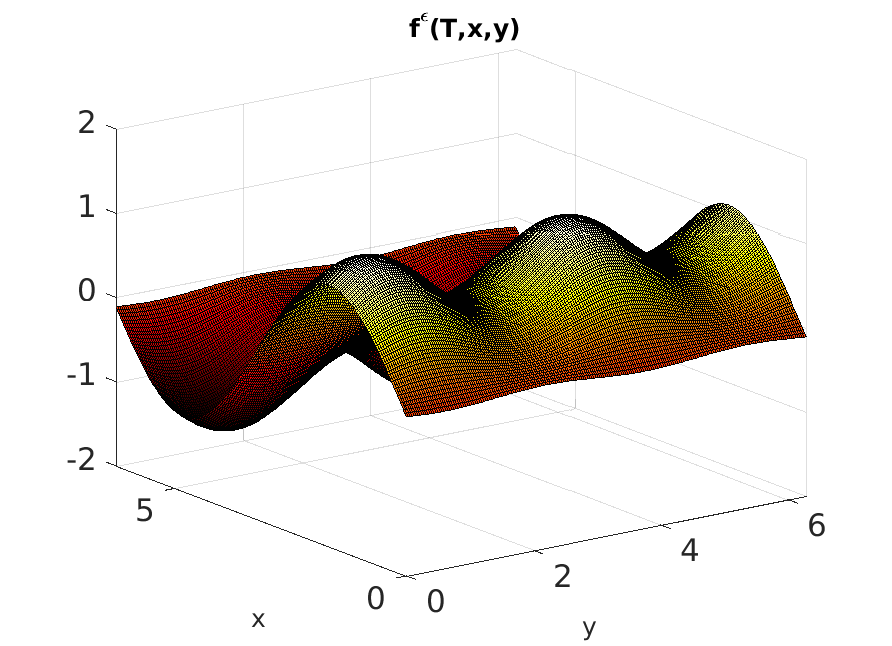}}
      \label{f_eps} 
\caption{\small{Representation of the numerical solution $f^{\epsilon}$ for two values of $\epsilon$, and at the final time $T$, corresponding to the IMEX scheme. }}
\label{figure3}
\end{figure}

\bigskip

\noindent For $\epsilon = 1$, we recognize an approximation of the exact solution (see Figure \ref{figure1}) and for $\epsilon=10^{-10}$, the limit solution is clearly obtained (see Figure \ref{figure2}). Briefly one can say that the numerical solution follows the weak-$\star$ convergence $f^\epsilon \converge{\star} f^0$ as $\epsilon$ becomes smaller and smaller. But, one can remark  a numerical diffusion which leads to a loss of amplitude, especially visible in the non-limit case $\epsilon=1$ or $\epsilon=10^{-1}$. To observe better this numerical diffusion, we show in the right plot of Fig. \ref{figure4} the time-evolution of just one point of the numerical solution, corresponding again to a 1D situation as the one plotted on the right of Fig. \ref{figure_fluctu}, and this for several values of $\epsilon$. 
\begin{figure}[ht]
\centering
	\subfloat[$\epsilon=10^{-10}$]{
        \includegraphics[width=0.45\textwidth]{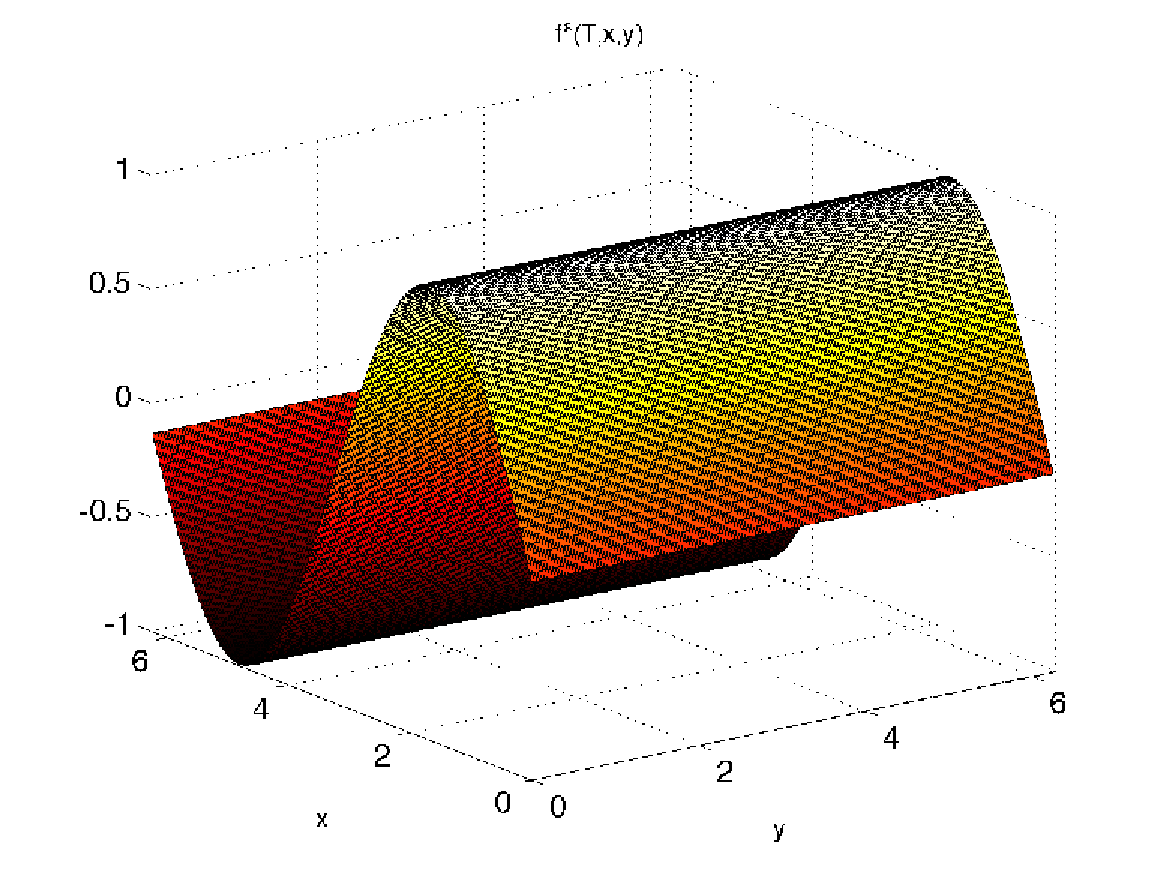}}
      \label{f0}    
	\subfloat[$f^{\epsilon}(t,y_{N_y-1})$]{
    \includegraphics[width=0.45\textwidth]{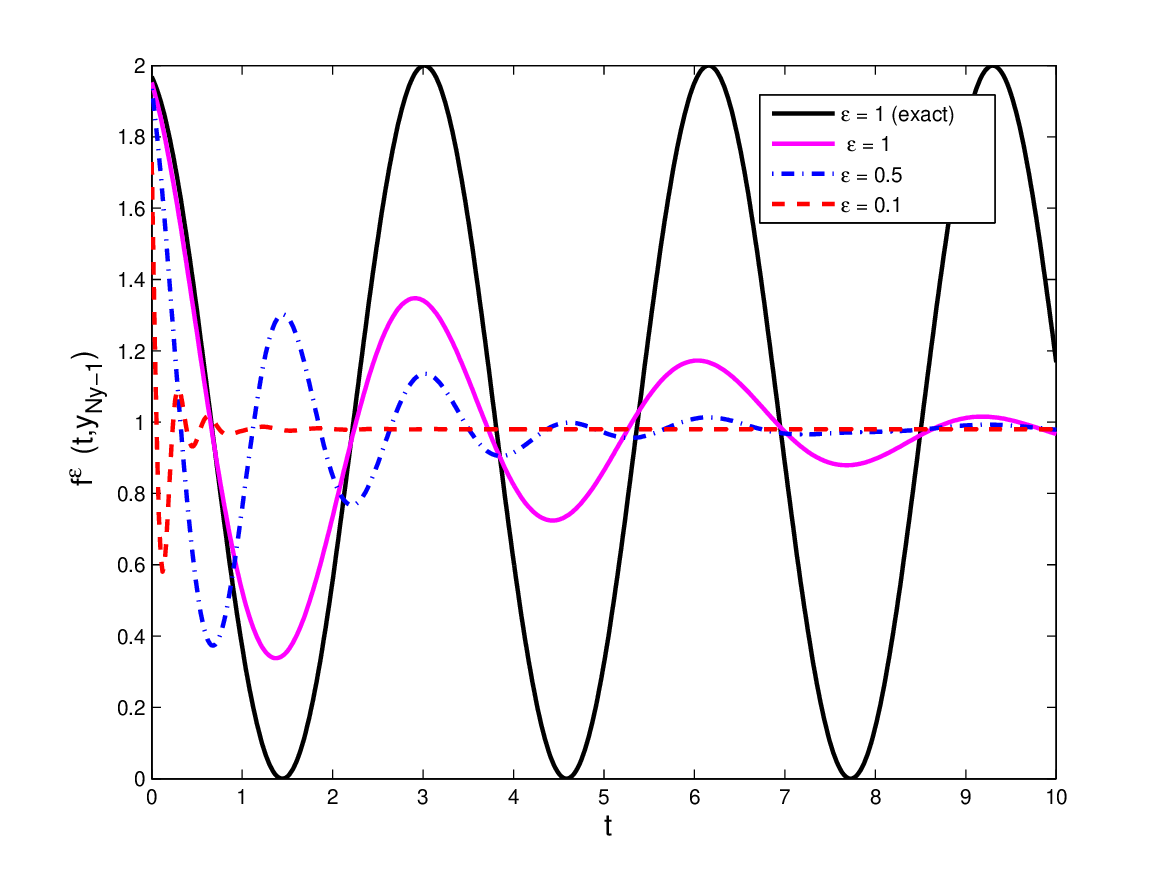}}
      \label{f_ana} 
\caption{\small{Left (A): Plot of the num. sol. $f^{\epsilon}$ for $\epsilon=10^{-10}$, at the final time $T$. Right (B): Time-evolution of the IMEX scheme sol. at point $y_{N_y-1}$ in the 1D case for $T=10$ and several $\epsilon$. We have added the exact solution for $\epsilon=1$.}}
\label{figure4}
\end{figure}
As one can observe, the damping is more and more pronounced if $\epsilon \to 0$. For small $\epsilon$-values the numerical solution recovers quasi immediately the weak limit solution, here the constant $1$. This damping phenomenon will be understood from the numerical analysis we shall fulfill in Section \ref{SEC25}.

\bigskip 
%%%%%%%
\subsubsection{Fourier, Micro-Macro and Lagrange-multiplier schemes} Let us now present analogous results for the remaining schemes, namely the Fourier, Micro-Macro and Lagrange-multiplier schemes. The 2D plots are rather similar to the ones presented for the IMEX-scheme (see Fig. \ref{figure3}-\ref{figure4}). To examine the difference between these methods, we preferred to plot in Fig. \ref{figure_fluct} only the time-evolution of the numerical solution in the 1D-context again.
\begin{figure}[ht]
\centering
	\subfloat[$f^{\epsilon}(t,y_{N_y-1})$]{
    \includegraphics[width=0.45\textwidth]{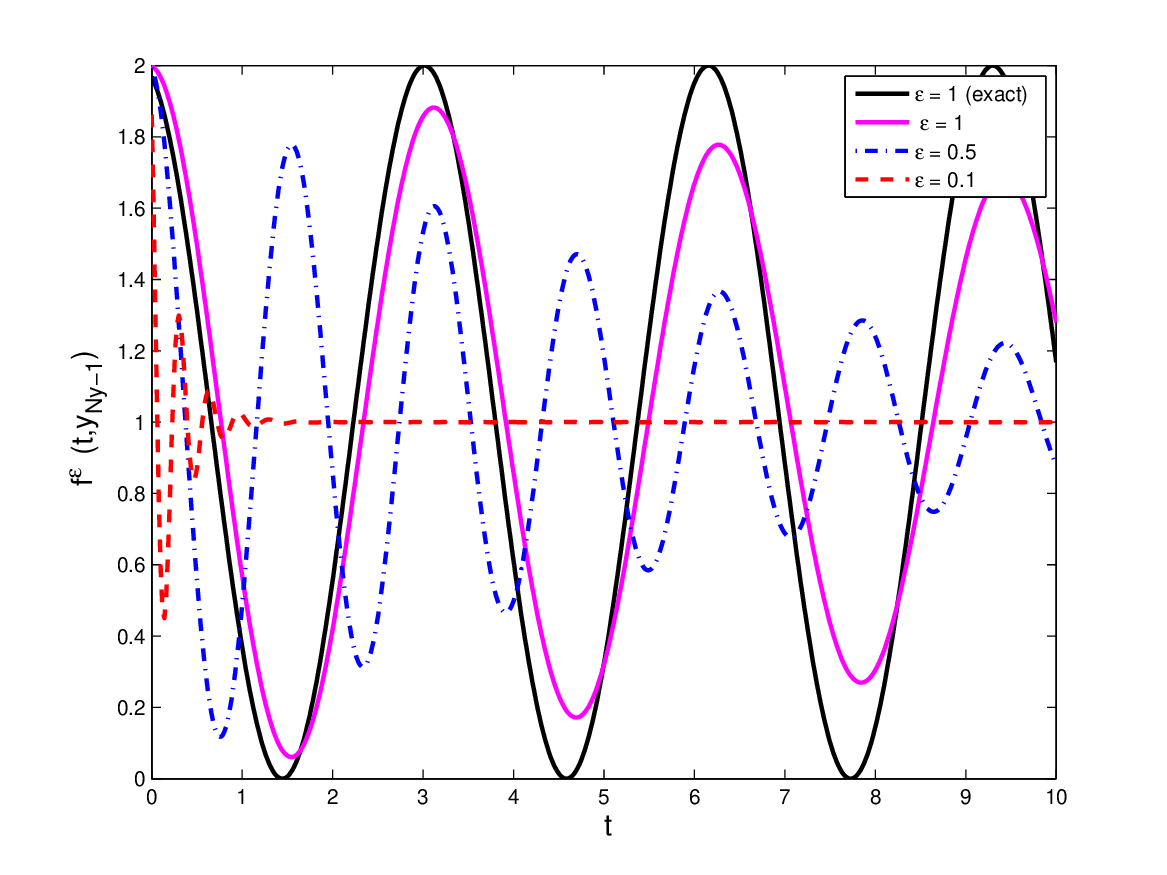}}
	\subfloat[$f^{\epsilon}(t,y_{N_y-1})$]{
    \includegraphics[width=0.45\textwidth]{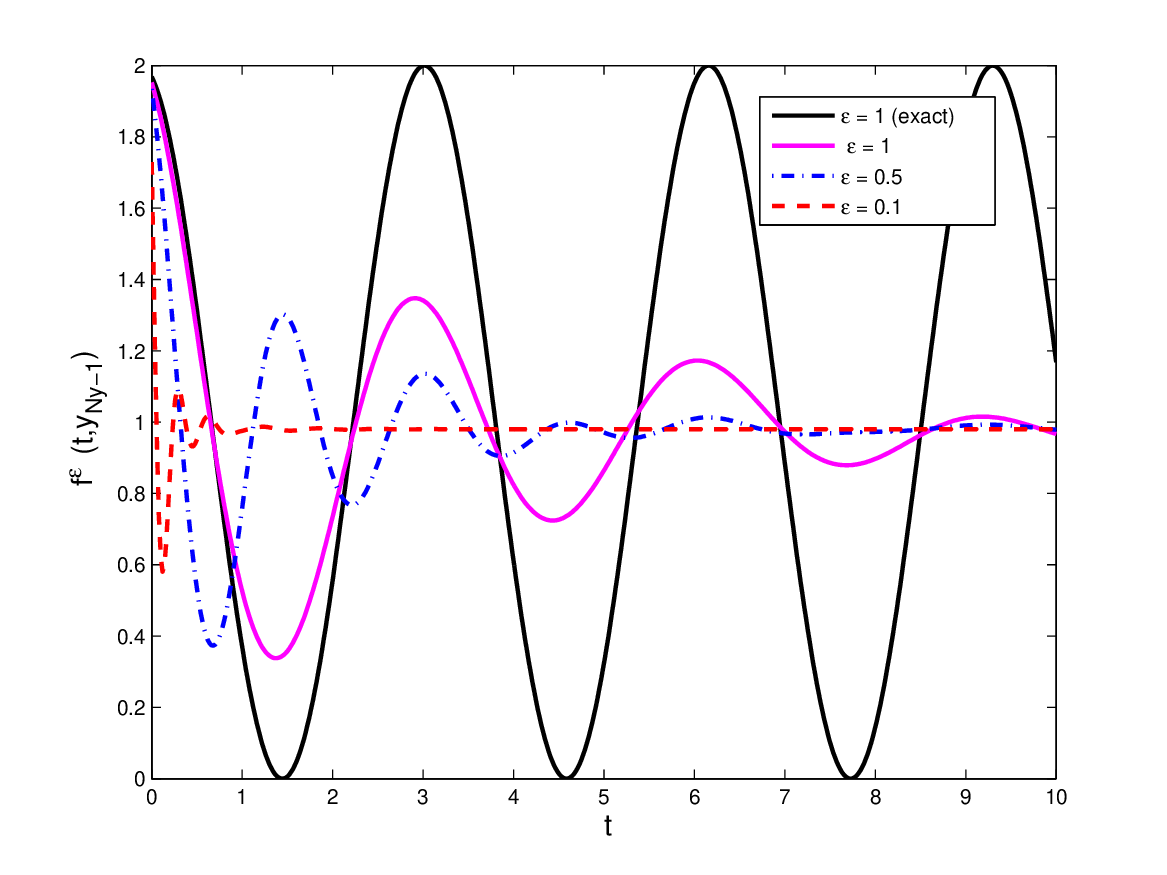}}
\caption{\small{Time-evolution of the solution via Fourier (A) and IMEX, MM- resp. Lagrange-multiplier schemes (B), at $y_{N_y-1}$ in 1D with $T=10$, $a=0$, $N_t = 501$. We have added in both cases the exact solution for $\epsilon=1$.} }
\label{figure_fluct}
\end{figure}
We remark that the damping of the Fourier method is more slowly than the ones  of the IMEX-scheme as well as Micro-Macro and Lagrange-multiplier scheme (which are completely overlapping). But, once again we observe that in the limit $\epsilon / t\to 0$, the fluctuations are completely damped out and we recover the weak limit solution. 
%%%%%%%%%%%%%%%%%%%%%%%%%%%%%
\subsection{Convergence of the schemes for fixed $\epsilon >0$}
%%%%%%%%%%%%%%%%%%%%%%%%%%%%%
Let us study now the convergence of the here presented schemes with respect to time and space, and this for fixed $\epsilon>0$, permitting to show their validity in the large $\epsilon$-regime. For this, fix $\epsilon >0$ and consider the error between exact and numerical solutions as a function of the mesh-size, at the final time T. Firstly, concerning the convergence with respect to $\Delta t$, we choose small space steps ($N_x=N_y = 501$) such that the space errors are much smaller than the time error and vary then the time step. Equally we apply the same strategy for the convergence with  respect to $\Delta x$ and $\Delta y$, by fixing a time step of $N_t=501$. In all cases, the parameter $\epsilon$ is fixed to $1$. In Figure \ref{tronca}, we have plotted curves in  log-log scale, showing the evolution of the errors as a function of $\Delta x$, $\Delta y$ and $\Delta t$, respectively.

\begin{figure}[ht]
\centering
	\subfloat[$\Delta x$]{
      \includegraphics[width=0.4\textwidth]{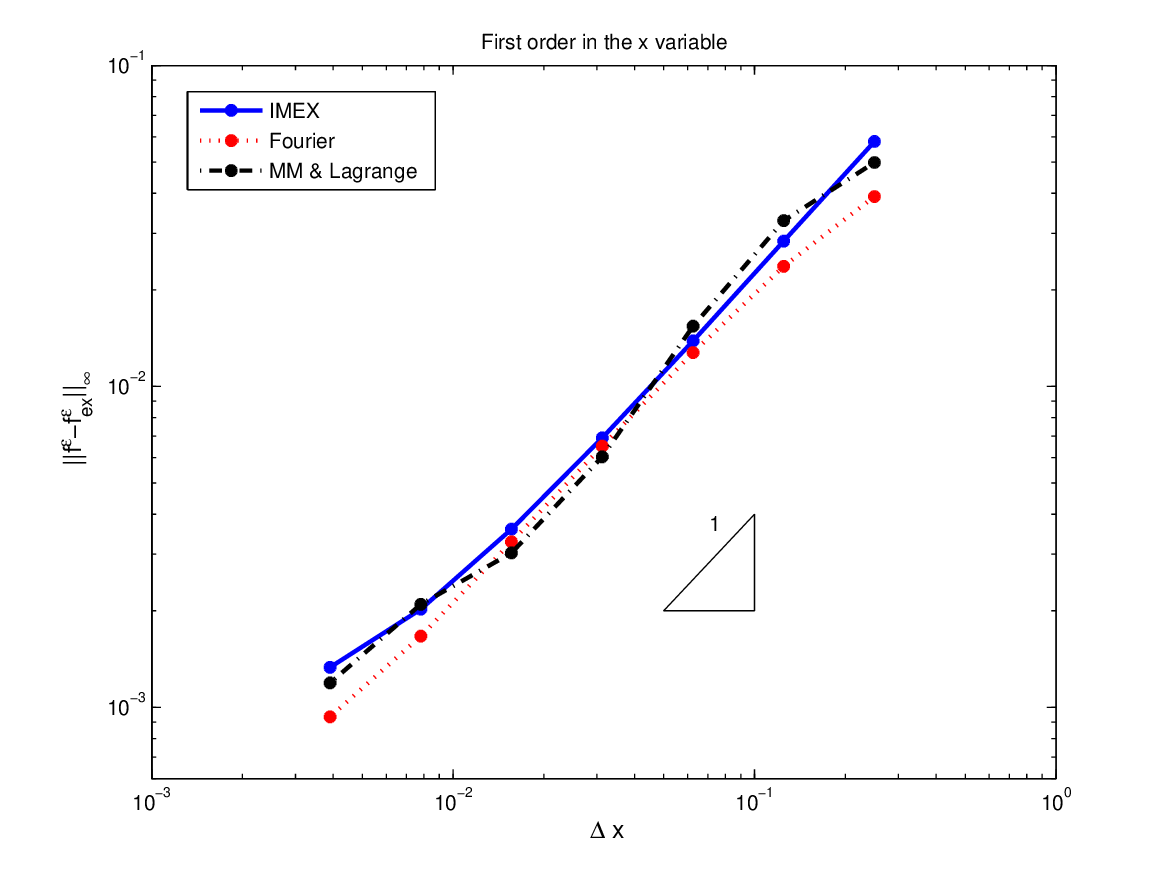}}
      \label{temps}    
	\subfloat[$\Delta y$]{
      \includegraphics[width=0.4\textwidth]{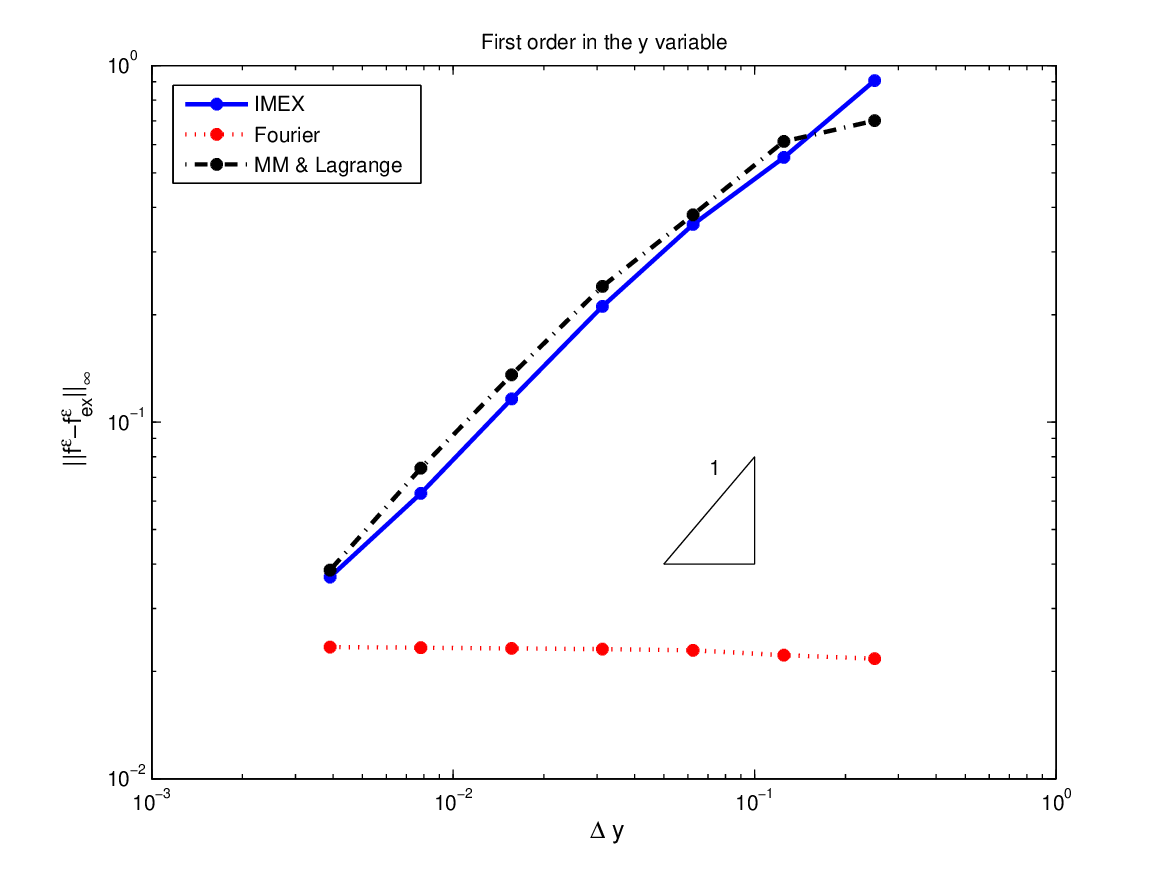}}
      \label{tronca_x} 
      \vfill
    \subfloat[$\Delta t$]{
      \includegraphics[width=0.4\textwidth]{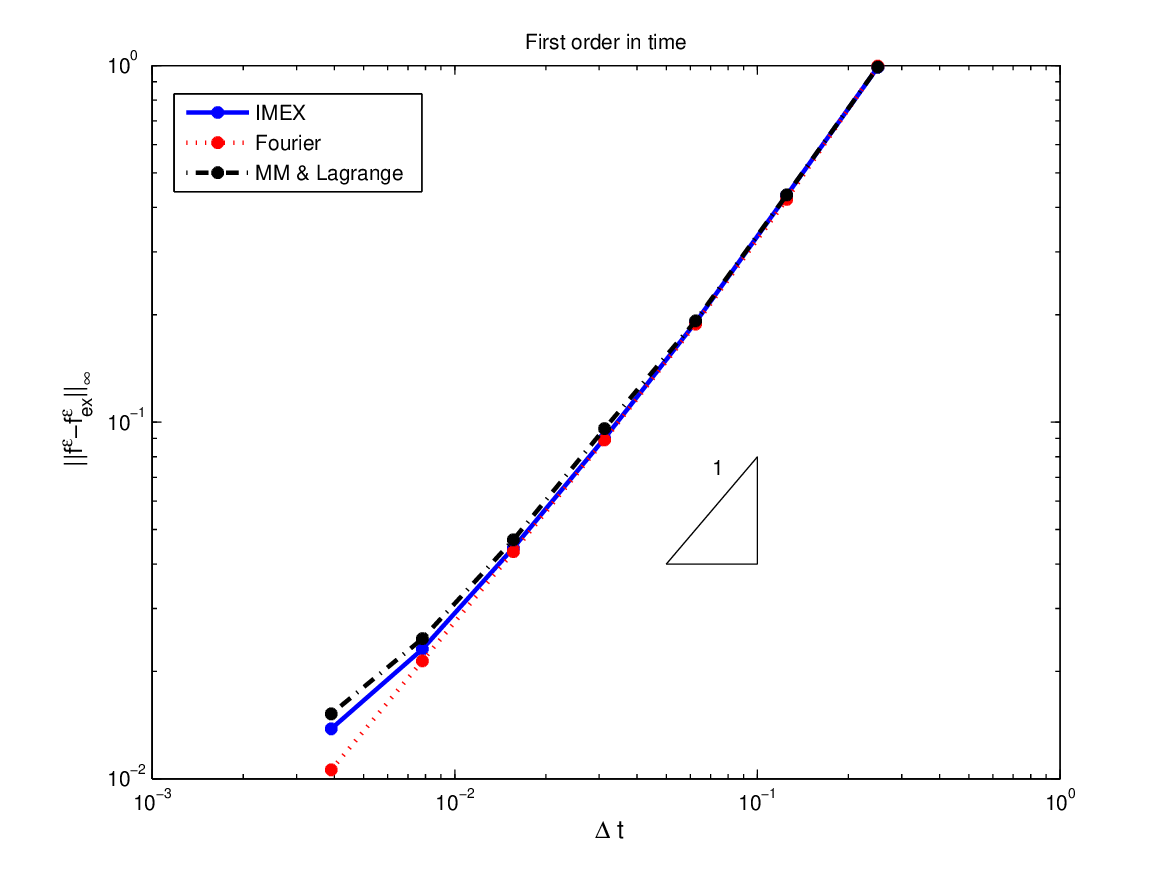}}
      \label{sub:renonc}
\caption{\small{Evolution of the $L^{\infty}$-error between
 $f^{\epsilon}_{ex}(t,\cdot)$ and $f^{\epsilon}(t,\cdot)$ at final time $T=1$ and for $\epsilon =1$, as a function of $\Delta x$
 (with $N_y = 15 001$, $N_t = 15 001$), $\Delta y$ (with $N_x = 15 001$, $N_t =
 15 001$) and $\Delta t$ (with $N_x=N_y=1 001$).}} 
\label{tronca}
\end{figure}

\bigskip 

\noindent As expected, we observe that all schemes are first order in time and space. Some comments are however necessary to understand Figure \ref{tronca}. First, the slop of the curves gets smaller than $1$ in the small-grid ranges. This is due to the fact that the error to be investigated (for ex. in $\Delta t$) becomes as small as the fixed error term (in $\Delta x,\, \Delta y$) and saturates. Secondly, the slope of the curves becomes also smaller in the large-grid ranges. This is usual, as for large discretization steps, the rest-terms in the Taylor series for the error analysis can no longer be neglected. Finally, we would like to draw the attention of the reader to the Fourier error curve, which has a constant slope in (B). This is completely natural, as the Fourier method has spectral accuracy.

%%%%%%%%%%%%%%%%%%%%%%%%%%%%%
\subsection{Asymptotic behavior as $\epsilon  \rightarrow 0$}
%%%%%%%%%%%%%%%%%%%%%%%%%%%%%
To begin the study of the asymptotic behavior, we define the following two errors
$$
\eta_\epsilon(t) = \max_{i,j}|f^{\epsilon}_{ex,i,j} - f^{\epsilon}_{num,i,j}|(t), \qquad
\gamma_\epsilon(t) = \max_{i,j}|f^{\epsilon}_{num,i,j} - f^{0}_{ex,i,j}|(t),
$$
\noindent where $\eta_\epsilon(t)$ represents the $L^{\infty}$- error between the exact and the numerical solution at instant $t$, for fixed $\epsilon > 0$, whereas $\gamma_\epsilon(t)$ denotes the $L^{\infty}$- error at instant $t$ between the numerical solution $f^{\epsilon}_{num}$ and the exact limit solution $f^{0}_{ex}$.

\bigskip 

\noindent We are interested in the evolution of these two errors at the final time $T$ as functions of $\epsilon$. The curves corresponding to  the different schemes are plotted in Figure \ref{figure11}.
\noindent As expected, we observe a decrease of $\eta_\epsilon(T)$ and an increase of $\gamma_\epsilon(T)$  when $\epsilon \to 1$. For $\epsilon \to 0$ the converse behavior is observed. This plot shows that each scheme approximates well either the exact solution $f_{ex}^{\epsilon}$ for large $\epsilon$, or the exact limit  solution $f^{0}_{ex}$ for small $\epsilon$.

\begin{figure}[h]
\begin{center}
\includegraphics[scale = 0.5]{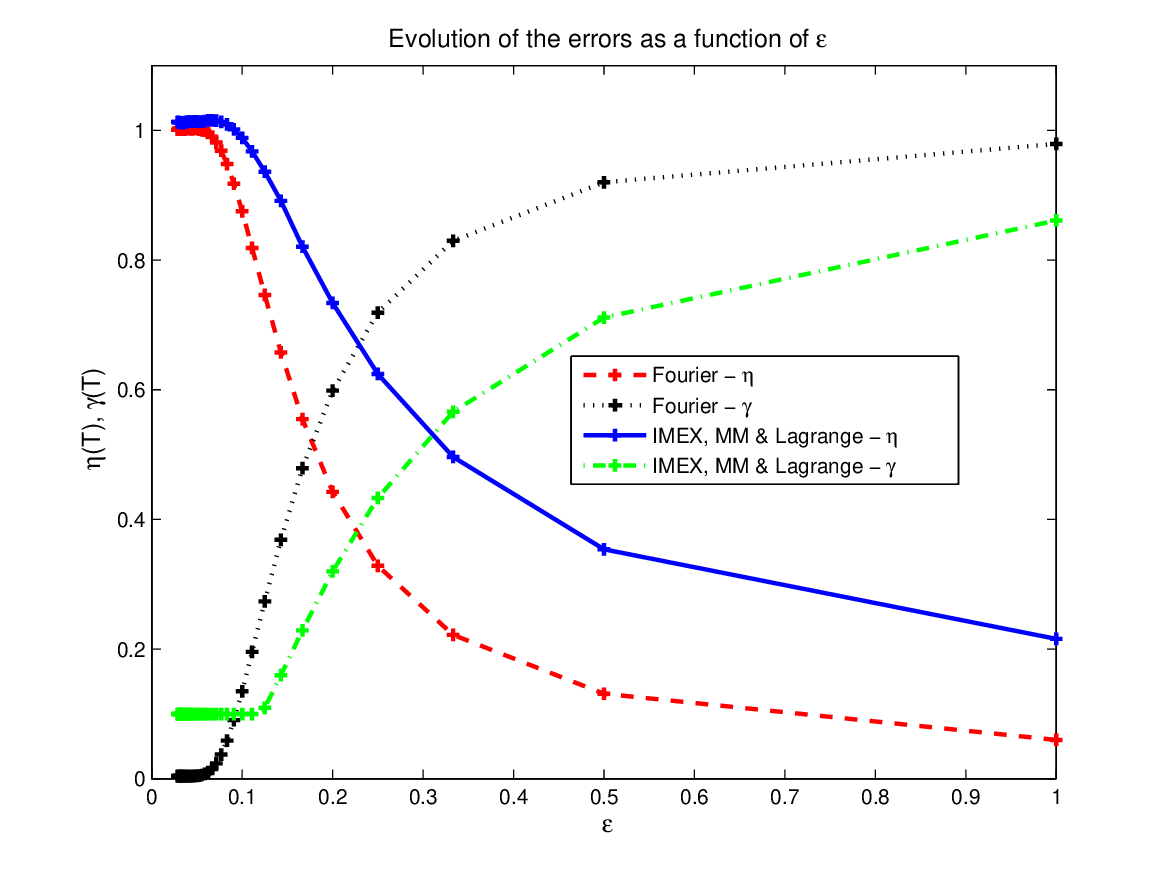} 
\end{center}
\caption{\small{Evolution of $\eta_\epsilon(T)$ and $\gamma_\epsilon(T)$ as a function of $\epsilon$ for each scheme.}}
\label{figure11}
\end{figure}

\bigskip

What can be said as a conclusion, is that all schemes seem to have the right asymptotic behavior  in this simple test case. Indeed, for fixed $\epsilon >0$, each numerical solution $f^{\epsilon}_{num}$ converges to the expected solution $f^{\epsilon}_{ex}$ as long as the grid is refined (Fig. \ref{tronca}). For fixed discretization steps, all numerical solutions $f^{\epsilon}_{num}$ converge towards the limit solution $f^{0}$ when $\epsilon$ becomes smaller and smaller, underlying the AP property of our methods.\\

It is worth mentioning however that the IMEX-scheme is no more working for $\epsilon$ smaller than $10^{-14}$, the matrix $\mathcal{A}$ of the IMEX linear-system \eqref{SYST_IM}, namely
$$
\mathcal{A } \;  \mathcal{F}_i^{n+1} = \mathcal{B}_i^n\,, \quad \quad 
\mathcal{A}=\epsilon \, Id + {\mathcal C}_\beta\,, \qquad \det \, {\mathcal C}_\beta=0\,,
$$
is becoming numerically singular in the limit $\epsilon \rightarrow 0$. This is not the case for the Micro-Macro as well as Lagrange-multiplier schemes, which give accurate results even for $\epsilon =0$.
This difference in the behavior can be observed also from the study of 
the condition-number of the discretization matrices, paying 
attention especially on the $\epsilon$-dependence. Remark here that an 
``Asymptotic-Preserving scheme'' must have an $\epsilon$-independent 
condition number, depending merely on the discretization parameters 
$\Delta x$, $\Delta y$.\\

%In order to understand more 
%thoroughly what befalls in detail, let us investigate the condition-number of the discretization matrices, paying 
%attention especially on the $\epsilon$-dependence. Indeed, an 
%``Asymptotic-Preserving scheme'' must have an $\epsilon$-independent 
%condition number, depending merely on the discretization parameters 
%$\Delta x$, $\Delta y$.\\

In Fig. \ref{cond} we plotted thus the matrix condition-number $cond(\mathcal{A 
}):=||\mathcal{A }^{-1}||_2\, ||\mathcal{A }||_2$ corresponding to the three 
schemes (IMEX, Micro-Macro and Lagrange-multiplier) as a function of $\epsilon$. What 
can be observed is that for the Micro-Macro and Lagrange-multiplier scheme, the 
condition-number is $\epsilon$-independent (for $\epsilon \le 10^{-2}$), which is a hint of the well-posedness of these two problems in the limit $\epsilon \rightarrow 0$, namely of $(MM)_0$ resp. $(La)_0$. 
On the other hand, for the IMEX-scheme  $cond(\mathcal{A 
})$ is proportional to $1/\epsilon$ (slope of the curve is approx. $-1$). This 
circumstance is the translation on the discrete level of the fact that 
the reduced model \eqref{Red}, obtained on the continuous level by letting 
formally $\epsilon \rightarrow 0$ in the IMEX time-discretization, is 
ill-posed, admitting an infinite amount of solutions.\\

\begin{figure}[h]
\begin{center}
       \includegraphics[width=0.45\textwidth]{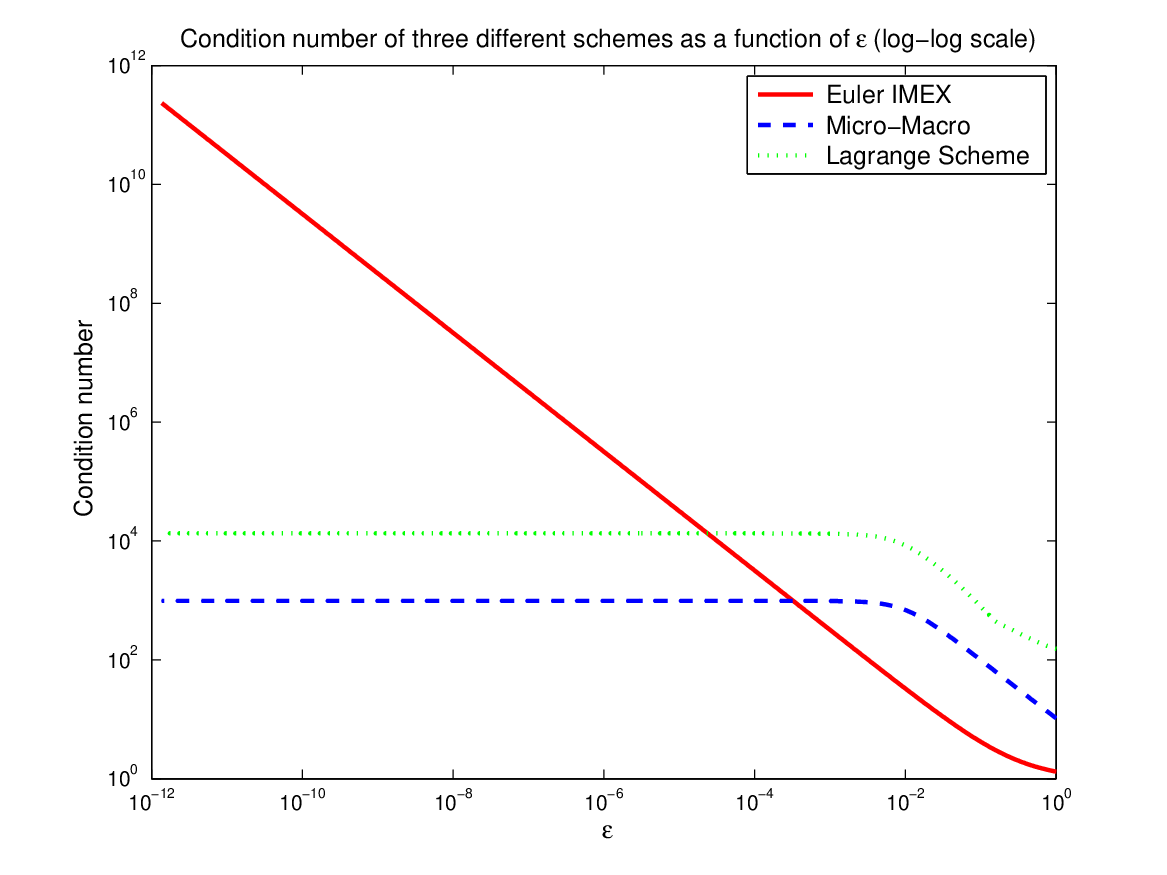}
       \label{cond_A}
\end{center}
\caption{\small{Condition number $cond(A)$ as a function of $\epsilon$ in  log-log scale. The three 
curves correspond to the IMEX, Micro-Macro and Lagrange-multiplier schemes.}}
\label{cond}
\end{figure}

% One can also assert, by analyzing the form of the IMEX linear-system, namely
% $$
% \mathcal{A } \;  \mathcal{F}_i^{n+1} = \mathcal{B}_i^n\,, \quad \quad 
% \mathcal{A}=\epsilon \, Id + {\mathcal C}_\beta\,, \qquad \det \, {\mathcal C}_\beta=0\,,
% $$
% that in the limit $\epsilon \rightarrow 0$ the discretization matrix tends 
% towards the singular matrix ${\mathcal C}_\beta$, such that for $\epsilon=0$ 
% the system has no more a unique solution. This explains the explosion of 
% $cond(\mathcal{A })$ as $\epsilon \rightarrow 0$.\\

However, even if these arguments show clearly that the IMEX-method 
should behave badly for very small $\epsilon$-values, it is not the case in our 
simplified toy model, in particular it does not seem to be affected by 
the bad condition number. This will no more be the case in our second toy-model. To understand in detail what happens, a more refined error study could be profitable and shall be done in the next section.  The final interpretation is postponed to Section \ref{Disc} after having estimated the truncation error. One can only say here that the functioning of the IMEX-scheme is due to the fact that the 
investigated problem is very simple and specifically the anisotropy is 
aligned with the Cartesian mesh.  

%%%%%%%%%%%%%%%%%%%%%%%%%%%%%%%%%%%%%%%%%%%%%%%
\section{Numerical analysis} \label{SEC25}
%%%%%%%%%%%%%%%%%%%%%%%%%%%%%%%%%%%%%%%%%%%%%%%

Let us now perform a numerical analysis study of our schemes introduced for the resolution of \eqref{AV}, permitting to understand in detail the behavior observed in the last section. In particular we shall detail only the error-analysis of the standard IMEX-scheme and the {\it Asymptotic-Preserving} Lagrange-multiplier scheme. The error study of the other schemes is very similar. See \cite{holmes,leveque} for more details on this analysis part.

%%%%%%%%%%
\subsection{ IMEX scheme}
%%%%%%
We begin by recalling  the full discretized form of the IMEX scheme :
\be \label{El_2}
(IMEX)_\epsilon \quad\,\, {f^{\epsilon,n+1}_{i,j}-f^{\epsilon,n}_{i,j} \over \Delta t} + a \; \frac{f_{i,j}^{\epsilon,n} - f_{i-1,j}^{\epsilon,n}}{\Delta x} + \frac{b}{\epsilon} \;  \frac{f_{i,j}^{\epsilon,n+1}-f_{i,j-1}^{\epsilon,n+1}}{\Delta y}=0\,,  \quad \forall (n,i,j) \in  Q_{h} \,.
\ee
\debthm \label{th_i}
The IMEX scheme \eqref{El_2} is consistent with the Vlasov equation \eqref{AV}, and first order accurate in space and time. Furthermore, the local truncation error writes
$$
\mathcal{T}_I(t_n,x_i,y_j,\Delta t,\Delta x,\Delta y) = -\nabla \cdot({D_I \nabla f^{\epsilon}}) +  \mathcal{O}(\Delta t ^2)  + \mathcal{O}(\Delta x ^2) + \mathcal{O}(\Delta y ^2),
$$
with 
$$
D_I := \left(\begin{array}{cc}\ds \frac{a \Delta x}{2}(1-\alpha)& 0 \\0 & \ds  \frac{b \Delta y }{2 \epsilon}\Big(1+ \frac{\beta}{\epsilon}\Big)\end{array}\right), \quad \ds \alpha := \frac{a \Delta t}{\Delta x}, \quad \ds \beta := \frac{b \Delta t}{\Delta y}.
$$

Finally, we observe that the IMEX scheme \eqref{El_2} is a second-order scheme for the modified Vlasov equation
\be \label{modified}
\partial_t g^{\epsilon} + a \; \partial_x g^{\epsilon} + \frac{b}{\epsilon} \; \partial_y g^{\epsilon}  - \frac{a\Delta x}{2} (1-\alpha) \; \partial_{xx} g^{\epsilon} - \frac{b \Delta y}{2 \epsilon}(1+ \frac{\beta}{\epsilon}) \; \partial_{yy} g^{\epsilon}=0.
\ee
\finthm

\debproof
The local truncation error of the method \eqref{El_2} is defined by
\begin{align*}
\mathcal{T}_I(t,x,y,\Delta t,\Delta x, \Delta y) = & {f^{\epsilon}(t+\Delta t,x,y)-f^{\epsilon}(t,x,y) \over \Delta t} + a \; \frac{f^{\epsilon}(t,x,y) - f^{\epsilon}(t,x-\Delta x,y)}{\Delta x}   \\ +  &\frac{b}{\epsilon}  \;\frac{f^{\epsilon}(t+\Delta t,x,y)-f^{\epsilon}(t+\Delta t,x,y-\Delta y)}{\Delta y}\,.
\end{align*}
\noindent Supposing that $f^{\epsilon}$ is sufficiently smooth in order to apply a Taylor development, we find 
\begin{align*}
\mathcal{T}_I(t_n,x_i,y_j,\Delta t,\Delta x,\Delta y) = &\partial_t f^{\epsilon} + \frac{\Delta t}{2} \; \partial_{tt} f^{\epsilon} + \frac{b \Delta t }{\epsilon} \; \partial_{yt}f^{\epsilon} + a \; \partial_x f^{\epsilon} - \frac{a \Delta x}{2} \; \partial_{xx}f^{\epsilon} +\frac{b}{\epsilon} \; \partial_{y}f^{\epsilon}\\ & - \frac{b\Delta y}{2 \epsilon}  \;\partial_{yy} f^{\epsilon}  + \mathcal{O}(\Delta t ^2)  + \mathcal{O}(\Delta x ^2) + \mathcal{O}(\Delta y ^2),
\end{align*}
\noindent where $f^{\epsilon}$ is taken in $(t_n,x_i,y_j)$. Since $f^{\epsilon}$ satisfies the Vlasov equation \eqref{AV}, the $\mathcal{O}(1)$ terms drop out.  Moreover, by differentiating the Vlasov equation along $t$, $y$ and $x$, we express the partial derivatives $\partial_{tt} f$ and $\partial_{ty}f$ as functions of $\partial_{xx} f$ and $\partial_{yy} f$. We find thus
$$
\partial_{tt}f^{\epsilon} = a^2 \; \partial_{xx}f^{\epsilon} + 2 \frac{ab}{\epsilon} \; \partial_{xy} f^{\epsilon} + \frac{b^2}{\epsilon ^2} \partial_{yy}f^{\epsilon}, \quad \partial_{yt}f^{\epsilon} = - a \; \partial_{xy}f^{\epsilon} - \frac{b}{\epsilon} \; \partial_{yy} f^{\epsilon}.
$$
\noindent   The local truncation error writes finally
$$
\mathcal{T}_I(t_n,x_i,y_j,\Delta t, \Delta x, \Delta y)= -\frac{a\Delta x}{2} (1-\alpha) \; \partial_{xx} f^{\epsilon} - \frac{b \Delta y}{2 \epsilon}(1+ \frac{\beta}{\epsilon}) \; \partial_{yy} f^{\epsilon} + \mathcal{O}(\Delta t ^2,\Delta x ^2,\Delta y ^2).
$$
\finproof

\debrem
The modified equation \eqref{modified} is an advection/diffusion equation. Note that the diffusion is stronger in the $y$-direction due to the term $ 1/\epsilon$. These diffusion terms are responsible for the damping that we observed in the numerical simulations (see Fig. \ref{figure4} (B)), damping which tends towards infinity in the $y$-direction, as $\epsilon \to 0$. Note also that the diffusion coefficient is positive if $\alpha \leq 1$. This is precisely the stability condition of the upwind scheme, as we will see afterwards. If this condition is not respected, the diffusion becomes negative, leading to an ill-posed problem with exponentially growing solutions.
\finrem

\debthm
The IMEX scheme is stable in the Von Neumann sense if and only if  the CFL-condition $\ds \frac{a \Delta t}{\Delta x} \leqslant 1$ is satisfied.
\finthm

\debproof To study the stability of our scheme, let us inject in \eqref{El_2} for fixed $n \in \NN$ a plane wave of the form 
\begin{align*}
f^{\epsilon,n}_{i,j} =  e^{ {\mathbf i} k x_i} e^{{\mathbf i} l y_j} \qquad  \forall (i,j) , 
\end{align*}
with $ \ds k,l \in \ZZ$ two arbitrary modes, and look how it evolves from one time-step to the other. Let us denote by $\xi_I$  the amplification factor for this passage $t_n \rightarrow t_{n+1}$, meaning
\begin{align*}
f_{i,j}^{\epsilon,n+1} = \xi_I\, f_{i,j}^{\epsilon,n}=\xi_I \; e^{{\mathbf i} k x_i} e^{{\mathbf i} l y_j}\,, \qquad  \forall (i,j) .
\end{align*}
Inserting now these terms in the discretized equation \eqref{El_2}, yields, after simplification
\begin{align*}
\xi_I \Bigg [1 + \frac{b \Delta t}{\epsilon \Delta y}(1-e^{-{\mathbf i}l \Delta y}) \Bigg] = \Bigg [ 1 - \frac{a \Delta t}{\Delta x}(1- e^{-{\mathbf i} k \Delta x}) \Bigg].
\end{align*}

\noindent A scheme is said to be stable in the Von Neumann sense, if the amplification factor satisfies $|\xi_I| \leq 1$,  such that the modes are not amplified from one time-step to the other. Straightforward computations yield 
\begin{align*}
|\xi_I | =  \epsilon \sqrt{ \frac{1 -4 \alpha (1-\alpha) \,\sin^2 \Big(\frac{k\, \Delta x}{2} \Big)}{\epsilon^2+ 4 \beta (\epsilon + \beta)\, \sin^2 \Big(\frac{l\,\Delta y}{2} \Big)}}\,, \quad \forall k,l \in \ZZ.
\end{align*}
\noindent Then, a necessary and sufficient condition to have the Von Neumann stability is : 
\begin{align*}
\frac{a \Delta t}{\Delta x} \leqslant 1.
\end{align*}
\finproof
\debrem
Note that in the case $l\neq 0$, when $\epsilon$ tends towards $0$, the amplification factor converges towards $0$. This means that for injected waves with mode $l\neq 0$, the scheme becomes more and more diffusive and attenuates completely the oscillations. 
\finrem
%%%%%%%%%%
\subsection{Lagrange-multiplier scheme}
%%%%%%%%%
We do now the same work for the Lagrange-multiplier scheme, {\it i.e.}
\be \label{La_2}
(La)_\epsilon \quad \left\{
\begin{array}{l}
\ds {f^{\epsilon,n+1}_{i,j}-f^n_{i,j} \over \Delta t} + a \; \frac{f_{i,j}^{\epsilon,n} - f_{i-1,j}^{\epsilon,n}}{\Delta x} + b \; \frac{q_{i,j}^{\epsilon,n+1}-q_{i,j-1}^{\epsilon,n+1}}{\Delta y}=0\,, \quad \forall (n,i,j) \in Q_h\\[6mm]
\ds  \frac{f_{i,j}^{\epsilon,n+1}-f_{i,j-1}^{\epsilon,n+1}}{\Delta y} = \epsilon \; \frac{q_{i,j}^{\epsilon,n+1}-q_{i,j-1}^{\epsilon,n+1}}{\Delta y}\,, \quad \forall (n,i,j) \in Q_h \\[6mm]
\ds q^{\epsilon,n}_{i,1} = 0\,, \quad \forall (n,i) \in [0,N_t] \times [0,N_x].
\end{array}
\right.
\ee

\debthm
The Lagrange-multiplier scheme \eqref{La_2} is consistent with the Vlasov equation \eqref{AV}, and first order accurate in space and time. Furthermore, the local truncation error  writes 
$$
\mathcal{T}_L(t_n,x_i,y_j,\Delta t, \Delta x, \Delta y) = -\nabla \cdot({D_L \nabla f^{\epsilon}})  + \mathcal{O}(\Delta t ^2)  + \mathcal{O}(\Delta x ^2) + \mathcal{O}(\Delta y ^2),
$$
\noindent with 
$$
D_L := \left(\begin{array}{cc}\ds \frac{a \Delta x}{2}(1-\alpha)& 0 \\0 & \ds  \frac{b \Delta y }{2 \epsilon}\Big(1+ \frac{\beta}{\epsilon}\Big)\end{array}\right), \quad \ds \alpha := \frac{a \Delta t}{\Delta x}, \quad \ds \beta := \frac{b \Delta t}{\Delta y}.
$$
\finthm

\debproof
In order to prove the result, we write the local truncation error of the first equation. We find that
\begin{align*}
\mathcal{T}_{L}(t_n,x_i,y_j,\Delta t, \Delta x, \Delta y) = &\frac{\Delta t}{2} \partial_{tt} f^{\epsilon}   - a \frac{\Delta x}{2} \partial_{xx}f^{\epsilon} - b \frac{\Delta y}{2} \partial_{yy} q^{\epsilon} +  b \Delta t \partial_{yt} q^{\epsilon}    \\ &+ \mathcal{O}(\Delta t ^2)  + \mathcal{O}(\Delta x ^2) + \mathcal{O}(\Delta y ^2).
\end{align*}

\noindent  Since the first equation of \eqref{La_2} is verified by $(f^{\epsilon},q^{\epsilon})$, we have
$$
\partial_{tt} f^{\epsilon} = -a \partial_{xt}f^{\epsilon} - b \partial_{yt} q^{\epsilon}, \quad \partial_{xt} f^{\epsilon} = - a \partial_{xx} f^{\epsilon} - b \partial_{xy} q^{\epsilon}\,, \quad \partial_{ty}f^{\epsilon} = - a\partial_{xy}f^{\epsilon} -b \partial_{yy}q^{\epsilon}\,.
$$
\noindent Then, 
\begin{align*}
\mathcal{T}_{L}(t_n,x_i,y_j,\Delta t, \Delta x, \Delta y) = &\frac{\Delta t}{2} (a^2 \partial_{xx} f^{\epsilon}+ ab \partial_{xy}q^{\epsilon})   - a \frac{\Delta x}{2}\partial_{xx} f^{\epsilon} - b \frac{\Delta y}{2}  \partial_{yy}q^{\epsilon} +  b\frac{\Delta t}{2} \partial_{ty} q^{\epsilon}  \\ &+ \mathcal{O}(\Delta t ^2)  + \mathcal{O}(\Delta x ^2) + \mathcal{O}(\Delta y ^2).
\end{align*}

\noindent Since the second equation of \eqref{La_2} is verified, we have  
$$
\partial_{ty}q^{\epsilon} = \frac{1}{\epsilon} \partial_{ty}f^{\epsilon}, \quad \partial_{yy}q^{\epsilon} = \frac{1}{\epsilon} \partial_{yy}f^{\epsilon}, \quad \partial_{xy}q^{\epsilon} = \frac{1}{\epsilon} \partial_{xy}f^{\epsilon}, \quad 
$$
\noindent such that we find  the same expression as for the IMEX scheme, {\it i.e.}
$$
\mathcal{T}_L(t_n,x_i,y_j,\Delta t , \Delta x, \Delta y) = -\nabla \cdot({D_L \nabla f^{\epsilon}}) +  \mathcal{O}(\Delta t ^2)  + \mathcal{O}(\Delta x ^2) + \mathcal{O}(\Delta y ^2).
$$
\finproof

\noindent The just proved  result confirms what we have seen on the numerical plots. Indeed, the IMEX and Lagrange-multiplier schemes have the same behavior when regarding the convergence and the asymptotic behavior.

\debthm
The Lagrange-multiplier scheme is stable in the Von Neumann sense if and only if the CFL condition $\ds \frac{a \Delta t}{\Delta x} \leqslant 1$ is satisfied.
\finthm

\debproof
Here, we have two unknown functions $f^{\epsilon}$ and $q^{\epsilon}$. To study the Von Neumann stability, we write 
$$
q^{\epsilon,n+1}_{i,j} = \xi_q \; q^{\epsilon,n}_{i,j}\,, \quad  f^{\epsilon,n+1}_{i,j} = \xi_f \; f^{\epsilon,n}_{i,j}\,,
$$
with the two amplification factors $\xi_q$ and $\xi_f$. As usual, we insert these expressions in the discretized Lagrange-multiplier equations. We obtain a linear system where the unknowns are $\xi_q$ and $\xi_f$. This system writes 
$$
\left(\begin{array}{cc}1 & \beta \big(1-e^{-{\mathbf i}k_m \Delta y}\big)  \\1 & -\epsilon \end{array}\right)\left(\begin{array}{c}\xi_f \\ \xi_q\end{array}\right) = \left(\begin{array}{c} \alpha\big(1-e^{-{\mathbf i}k_n \Delta x}\big)\\0\end{array}\right)\,,
$$
and is easy to invert. Computing the amplification factor $\xi_f$, we remark that it is identical to the one calculated for the IMEX scheme. 
\finproof
%%%%%%%%%%
\subsection{AP-properties} \label{Disc}
%%%%%%
We are now able to explain the numerical results obtained in Section \ref{SEC24}, in particular to explain why the IMEX-scheme, even if being not an AP-scheme, gives in this simple field-aligned test case, good results up to a value of $\epsilon=10^{-14}$. For this, let us recall that two types of errors arise during a numerical resolution of the Vlasov equation \eqref{AV}. First of all we have the truncation errors, estimated in the last subsections, and secondly one has also the round-off errors, arising at each elementary computation. To be more precise, one has to consider the three linear systems, corresponding to \eqref{SYST_IM}:
$$
\mathcal{A } \;  \mathcal{F}_{ex} = \mathcal{B} + \epsilon \mathcal{T}\,, \qquad \mathcal{A } \;  \mathcal{F} = \mathcal{B}\,, \qquad (\mathcal{A }+ \delta \mathcal{A }) \;  \mathcal{F}_{num} = \mathcal{B}+ \delta \mathcal{B}\, ,
$$
where to simplify notation we omitted all the time and space indices. Here we denoted by $\mathcal{F}_{ex}$ the exact solution of the Vlasov equation \eqref{AV}, satisfying the linear system \eqref{SYST_IM} up to a truncation error $\mathcal{T}$, $\mathcal{F}$ is the exact solution of the linear system \eqref{SYST_IM}, supposing exact arithmetics, and finally $\mathcal{F}_{num}$ is the solution to the linear system \eqref{SYST_IM} obtained via a computer, hence contaminated with round-off errors. The error we are interested in, can be estimated as follows
$$
||\mathcal{F}_{ex}-\mathcal{F}_{num}||\le ||\mathcal{F}_{ex}-\mathcal{F}||+ ||\mathcal{F}-\mathcal{F}_{num}||\,.
$$
Stability and consistency permit to show that the first error term is of the order of the truncation error. For the estimate of the second error term, we have to take into account the condition number of the matrix, in particular one has the estimate \cite{Tref}
$$
{||\mathcal{F}-\mathcal{F}_{num}||\over ||\mathcal{F}||} \le { \textrm{cond}(\mathcal{A}) \over 1-||\mathcal{A }^{-1}||\,||\delta \mathcal{A}||  } \, \left( {||\delta \mathcal{A}||\over ||\mathcal{A}||}+ {||\delta \mathcal{B}||\over ||\mathcal{B}||} \right) \,.
$$
Performing our computations in double precision (machine accuracy of $10^{-16}$), and as long as the condition number is not exceeding a value of $10^{12}$ (see Fig. \ref{cond}), the second error term is not so dangerous. For larger condition numbers, this term can give rise to erroneous results. In our test case, it is however rather the first error-term which leads to trouble, as the truncation error is $1/\epsilon$-dependent. In the first toy-model \eqref{AV}, the large truncation error impacts only  the $y$-direction, leading to a large diffusion along the axes-aligned anisotropy and hence to the limit-model. We shall see a drastic difference in the second, not-field aligned toy-model.

\section{Second Vlasov toy-model with variable coefficients}\label{SEC26}
%%%%%%%%%%%%%%%%%%%%%%%%%%%%%%%%%%%%%%%%%%%%%%%

Finally, let us come now in this section to the second Vlasov toy model, given by : 
\be \label{eq_cir_0}
\partial_t f^{\epsilon} + \frac{1}{\epsilon}(\mathbf{v} \times \mathbf{B})\cdot{\nabla_v f^{\epsilon}}=0, 
\ee
with $\epsilon \ll 1$ and the magnetic field $\ds \mathbf{B} = \mathbf{e_z}$. This model is a simplified version of the anisotropic Vlasov equation \eqref{V_C} in not-field aligned Cartesian coordinates. Denoting, for notational simplicity, the velocity-variable as $v=(x,y,z)$, we have $\mathbf{v}\times \mathbf{B} = (y \;,  -x \;,  0)^t$, such that the previous equation writes :
\be \label{eq_cir}
(G)_\epsilon\,\,\, 
\left\{
\begin{array}{l}
\ds \partial_t f^{\epsilon} + \frac{y}{\epsilon} \partial_x f^{\epsilon} - \frac{x}{\epsilon}\partial_y f^{\epsilon} = 0, \quad \quad \forall (t,x,y) \in [0,T] \times \Omega. \\[3mm]
f^{\epsilon}(0,x,y) = f_{in}(x,y), \quad \quad \forall (x,y) \in \Omega\,,
\end{array}
\right.
\ee
where this time our velocity-domain is given by $\Omega:=[-L_x,L_x]\times [-L_y,L_y]$. Again we will consider a doubly-periodic framework.

%%%%%%%%%%
\subsection{Exact solution by the characteristic method}
%%%%%%%%%
The exact solution of the equation \eqref{eq_cir} is simply determined  via the characteristic method. The characteristic curve $\mathcal{C}_{\epsilon}^{x,y}(s) := \Big(X(s), Y(s)\Big)$ passing at instant $t$ through $(x,y)$, solves the ODE : 
$$
\left\lbrace
\begin{array}{ll}
\displaystyle \dot{X}(s) = \frac{Y(s)}{\epsilon}, \\ \\
\displaystyle \dot{Y}(s) = - \frac{X(s)}{\epsilon},
\end{array}\right.
\label{charac}
\quad \quad  (X(t),Y(t)) = (x,y).
$$
\noindent We can write this system under matrix form : 
$$
\left(\begin{array}{c}\dot{X} \\\dot{Y}\end{array}\right)  = \frac{1}{\epsilon} \, A\, \left(\begin{array}{c}X \\Y\end{array}\right), \qquad A:=\left(\begin{array}{cc}0 & 1 \\-1 & 0\end{array}\right)
$$
\noindent leading to 
$$
\mathcal{C}_{\epsilon}^{x,y}(s) := \ds \left(\begin{array}{c}X \\Y\end{array}\right)(s)  =  e^{A \frac{s-t}{\epsilon}} \left(\begin{array}{c}x \\y\end{array}\right).
$$
Denoting the rotation matrix by $ \mathcal{R}_{\epsilon}(y) := \ds e^{A \frac{y}{\epsilon}}$, one has
$$
\mathcal{R}_{\epsilon}(s-t) = e^{A \frac{s-t}{\epsilon}} =
\left(\begin{array}{cc}\ds \cos\Big(\frac{s-t}{\epsilon}\Big) & \ds \sin \Big(\frac{s-t}{\epsilon}\Big) \\ \\\ds -\sin \Big(\frac{s-t}{\epsilon}\Big) & \ds \cos \Big(\frac{s-t}{\epsilon}\Big)\end{array}\right) \,.
$$
\noindent We can easily verify that the characteristic curve passing through the point $(x,y)$ is a spiral, whose projection on the $(x,y)$-plane is the circle with radius $R:= \sqrt{x^2+y^2}$ and center $(0,0)$. All characteristics are $2\, \pi\, \epsilon$-periodic (in $t$).

\noindent The exact solution $f^{\epsilon}$ of \eqref{eq_cir} is now simply the advection of the initial condition along these characteristic curves, such that
$$
f^{\epsilon}(t,x,y) = f_{in}(X(0,t,x,y),Y(0,t,x,y))= f_{in} \Big(\cos \Big( \frac{t}{\epsilon} \Big)x - \sin \Big(\frac{t}{\epsilon}\Big)y,\sin \Big(\frac{t}{\epsilon}\Big)x + \cos \Big( \frac{t}{\epsilon} \Big)y \Big).
$$
%%%%%%%%
\subsection{Limit solution of the problem}
%%%%%%%%
The next step is to obtain the limit solution of the problem \eqref{eq_cir}, as $\epsilon \to 0$. Keeping in mind that $f^{\epsilon}$ is constant along the characteristic curves, we integrate \eqref{eq_cir} along $\mathcal{C}_{\epsilon}^{x,y}$, to get 
$$
\partial_t \int_{\mathcal{C}_{\epsilon}^{x,y}}  f^{\epsilon} d \sigma + \frac{1}{\epsilon} \int_{\mathcal{C}_{\epsilon}^{x,y}} (y, \, -x)^t \cdot{\nabla f^{\epsilon}}(t,x,y)d \sigma= 0,
$$
\noindent leading to 
$$
\partial_t \int_{\mathcal{C}_{\epsilon}^{x,y}}  f^{\epsilon} d\sigma + \frac{1}{\epsilon} \int_{t}^{t+2\pi \epsilon} (Y(s),\, \quad -X(s))^t \cdot{\nabla f^{\epsilon}}(t,X(s),Y(s))\,{\sqrt{x^2+y^2}\over \epsilon}ds = 0\,.
$$
Furthermore, as
$$
\int_{t}^{t+2\pi \epsilon} (Y(s),\, -X(s))^t \cdot{\nabla f^{\epsilon}}(t,X(s),Y(s))ds = \int_t^{t+2\pi \epsilon}{d\over ds} \Big[f^{\epsilon}\Big(t,X(s),Y(s))\Big)\Big] =0,
$$
\noindent which comes from the periodicity of the characteristics, and denoting the average along a curve by $\ds \langle {f^{\epsilon}} \rangle := \frac{1}{|\mathcal{C}_{\epsilon}^{x,y}|} \int_{\mathcal{C}_{\epsilon}^{x,y}}f^{\epsilon}d\sigma$, with $|\mathcal{C}_{\epsilon}^{x,y}|=2\, \pi\, \epsilon$, we have :
$$
\partial_t \langle f^{\epsilon} \rangle =0\,.
$$
\noindent Letting now formally $\epsilon \to 0$, we obtain the following limit problem associated to \eqref{eq_cir}:
\be \label{eq_cir_lim} 
(G)_0\,\,\qquad 
\langle f^{0} \rangle = \langle f_{in}\rangle.
\ee

%%%%%%%%%%%%%%%%%%%%%%%%%%%%
\subsection{Numerical schemes for the second Vlasov toy model }
%%%%%%%%%%%%%%%%%%%%%%%%%%%%%
Let us now discretize the second Vlasov toy model \eqref{eq_cir} via the IMP (fully implicit scheme this time) and Lagrange-multiplier schemes. The time semi-discretizations read
\be \label{st_IMEX}
(IMP)_{\epsilon}\quad
\frac{f^{\epsilon,n+1}-f^{\epsilon,n}}{\Delta t} +\frac{y}{\epsilon} \;  \partial_x f^{\epsilon,n+1} - \frac{x}{\epsilon}  \;\partial_y f^{\epsilon,n+1} =0, \quad \quad \forall n \geqslant 0, \quad \forall (x,y) \in \Omega,
\ee 
as well as
\be \label{Lag_cir_0}
(La)_\epsilon \quad \left\{
\begin{array}{l}
\ds  \frac{f^{\epsilon,n+1}-f^{\epsilon,n}}{\Delta t} + y \;  \partial_x q^{\epsilon,n+1} - x \; \partial_y q^{\epsilon,n+1} = 0,\\[6mm]
\ds y \; \partial_x f^{\epsilon,n+1} - x\; \partial_y f^{\epsilon,n+1} = \epsilon \Big(y\;  \partial_x q^{\epsilon,n+1} - x \; \partial_y q^{\epsilon,n+1}\Big)-  (\Delta x \Delta y)^\gamma \; q^{\epsilon,n+1}
\end{array}
\right.
 \forall n \geqslant 0\,.
\ee
\noindent The term $(\Delta x \Delta y)^\gamma \; q^{\epsilon,n}$ in \eqref{Lag_cir_0} is a stabilization term permitting to have the uniqueness of the solution $(f^{\epsilon}, q^{\epsilon})$. In the former "field-aligned" example, we fixed $q^{\epsilon}$ on the anisotropy lines by setting $q^{\epsilon}_{|\Gamma_{in}}=0$, but here it is more arduous from a numerical point of view. The stabilization aims equally to fix $q^\epsilon$, however in a different manner. It is very delicate to choose the magnitude of this term, in order not to destroy the problem, in particular we took here $\gamma=0.91$. First it is a small perturbation of the equation, of the order of the truncation error. Secondly, averaging the second equation of the Lagrange-multiplier scheme along the anisotropy lines, permits to obtain
\begin{align*}
(\Delta x \Delta y)^\gamma \;  \langle q^{\epsilon,n+1} \rangle = 0, 
\end{align*}
\noindent which means that $q^{\epsilon}$ is unique, by having zero average along the field lines. A more detailed study of this stabilization technique was performed in \cite{LNN} for the elliptic framework.

\noindent For the spatial discretization, we use again an upwind scheme, observing that this time the equation has no more constant coefficients. Thus, we define :
$$
x_i^{+} := \max_i (x_i,0), \quad x_i^{-} := \min_i(0,x_i), \quad y_j^{+} := \max_j (y_j,0), \quad y_j^{-}=\min_j(0,y_j),  \; \forall (i,j) \in \mathbb{N}^2.
$$
\noindent The full discretization of the IMP scheme writes now
\begin{align*} \label{eq_euler}
(IMP)_{\epsilon}\,\,\, 
& f^{\epsilon,n+1}_{i,j}+ \frac{1}{\epsilon}\Big[\big(r_x (y_j^{+}-y_j^{-})+r_y(x_i^{+}-x_i^{-} \big)f_{i,j}^{\epsilon,n+1} - r_x (y_j^{+}f^{\epsilon,n+1}_{i-1,j} - y_j^{-} f^{\epsilon,n+1}_{i+1,j}) - \\  & r_y(x_i^{+}f^{\epsilon,n+1}_{i,j+1} - x_i^{-}f^{\epsilon,n+1}_{i,j-1})\Big] = f^{\epsilon,n}_{i,j}, \quad \forall (n,i,j) \in Q_h,
\end{align*}
\noindent with $ \ds r_x = \frac{\Delta t}{\Delta x}$ and $ \ds r_y = \frac{\Delta t}{\Delta y}$. And for the Lagrange-multiplier scheme, we have : 
$$
(La)_\epsilon\,\,\, 
\left\lbrace
\begin{array}{ll}
\displaystyle f^{\epsilon,n+1}_{i,j}+ \frac{1}{\epsilon}\Big[\big(r_x (y_j^{+}-y_j^{-})+r_y(x_i^{+}-x_i^{-} \big)q_{i,j}^{\epsilon,n+1} - r_x (y_j^{+}q^{\epsilon,n+1}_{i-1,j} - y_j^{-} q^{\epsilon,n+1}_{i+1,j}) - \\[3mm] r_y(x_i^{+}q^{\epsilon,n+1}_{i,j+1} - x_i^{-}q^{\epsilon,n+1}_{i,j-1})\Big] = f^{\epsilon,n}_{i,j},  \quad \forall (n,i,j) \in Q_h, \\ \\
\displaystyle \frac{1}{\Delta t}\Big[\big(r_x (y_j^{+}-y_j^{-})+r_y(x_i^{+}-x_i^{-} \big)f_{i,j}^{\epsilon,n+1} - r_x (y_j^{+}f^{\epsilon,n+1}_{i-1,j} - y_j^{-} f^{\epsilon,n+1}_{i+1,j}) - \\[3mm]  r_y(x_i^{+}f^{\epsilon,n+1}_{i,j+1} - x_i^{-}f^{\epsilon,n+1}_{i,j-1})\Big] 
= \displaystyle \frac{\epsilon}{\Delta t}\Big[\big(r_x (y_j^{+}-y_j^{-})+r_y(x_i^{+}-x_i^{-} \big)q_{i,j}^{\epsilon,n+1} - \\[3mm] r_x (y_j^{+}q^{\epsilon,n+1}_{i-1,j} - y_j^{-} q^{\epsilon,n+1}_{i+1,j}) - r_y(x_i^{+}q^{\epsilon,n+1}_{i,j+1} - x_i^{-}q^{\epsilon,n+1}_{i,j-1})\Big] - (\Delta x \Delta y)^\gamma \; q^{\epsilon,n+1}_{i,j},  \quad \forall (n,i,j) \in Q_h.
\end{array}\right.
\label{charac}
$$
%%%%%%%%
\subsection{Numerical simulations}
%%%%%%%%%%
\noindent Here we present our simulations corresponding to both numerical schemes. We consider $\Omega = [-3,3]^2$, $T=1$ and the discretization parameters $N_t = 64$ and $N_x = N_y = 160$. The initial data is defined by a Gaussian function :
$$
f_{in}(x,y) = \exp \Bigg(-\frac{x^2 + y^2}{2 \sigma^2}\Bigg), \qquad \sigma = 0.5, \quad \forall(x,y) \in \Omega.
$$
As we showed before, the exact solution is known thanks to the characteristic method. In the present simple test case, one can easily prove that 
\be \label{EX_FF}
f_{ex}^{\epsilon}(t,x,y) = f_{in}(x,y) = \exp \Bigg(-\frac{x^2 + y^2}{2\sigma^2}\Bigg),
\ee
\noindent in other words, the exact solution is a stationary solution, independent of $\epsilon$, the initial condition being constant along the anisotropy field. This simple test case permits in a very simple way to compare both methods with respect to the $\epsilon$-dependence of the results, in particular to show that the IMP-scheme is not an {\it Asymptotic-Preserving} scheme. We shall investigate in a future paper a more involved, physical test-case, where we shall adapt the here introduced Lagrange-multiplier-method, which seems to be the most appropriate method for our singularly-perturbed Vlasov problem \eqref{V_C}, to second-order schemes and test more thoroughly its AP-properties.\\

In Figure \ref{cond_2toy} we first plot the condition-number $cond(A):=||\mathcal{A }^{-1}||_2\, ||\mathcal{A }||_2$ associated to the two schemes. As for the first toy-model, one remarks the $\epsilon$-independent condition-number of the Lagrange-multiplier-scheme, whereas, as expected, the IMP scheme has an $1/\epsilon$-dependent condition-number.
\begin{figure}[h]
\begin{center}
\psfrag{IMEX}[][][0.5]{IMP}
\includegraphics[scale = 0.5]{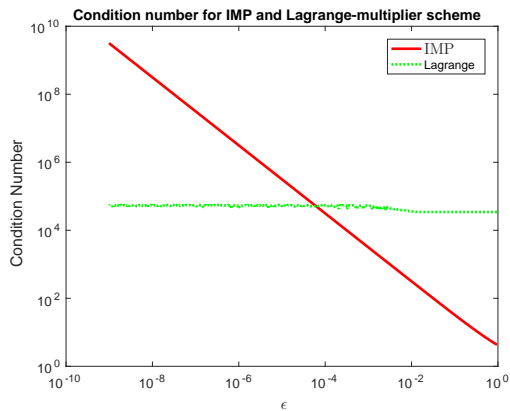} 
\end{center}
\caption{\small{Condition number $cond(A)$ as a function of $\epsilon$ in  log-log scale. The two 
curves correspond to the IMP and Lagrange-multiplier schemes.}}
\label{cond_2toy}
\end{figure}

\bigskip
\noindent Then, in Figure \ref{figure14}, we show the numerical solution $f^{\epsilon}$ at the final time $T$ and computed for several values of $\epsilon$, with both IMP and Lagrange-multiplier schemes. For $\epsilon =1$ and $\epsilon = 0.1$, we do not distinguish any difference. However for smaller $\epsilon$ values, the solution obtained with the Lagrange-multiplier scheme seems to be $\epsilon$-independent, contrary to the IMP scheme, which diffuses more and more as $\epsilon \to 0$.  Indeed, the IMP solution is completely damped as $\epsilon \to 0$ and leads towards the zero-solution, whereas the Lagrange-multiplier scheme keeps the form of the Gaussian, with a usual $\epsilon$-independent $(\Delta x, \Delta y)$-diffusion. This permits to conclude that the Lagrange-multiplier scheme is an AP-scheme contrary to the IMP scheme. 

\bigskip

In order to distinguish much better this AP-property, we plot on Figure \ref{figure_gauss} a cut of the previous curves at the point $x=0$. We observe clearly the diffusion in the IMP scheme which depends of $\epsilon$ contrary to the Lagrange-multiplier scheme. 

\begin{figure}[ht]
\centering
	\subfloat[IMP  ]{
      \includegraphics[width=0.44\textwidth]{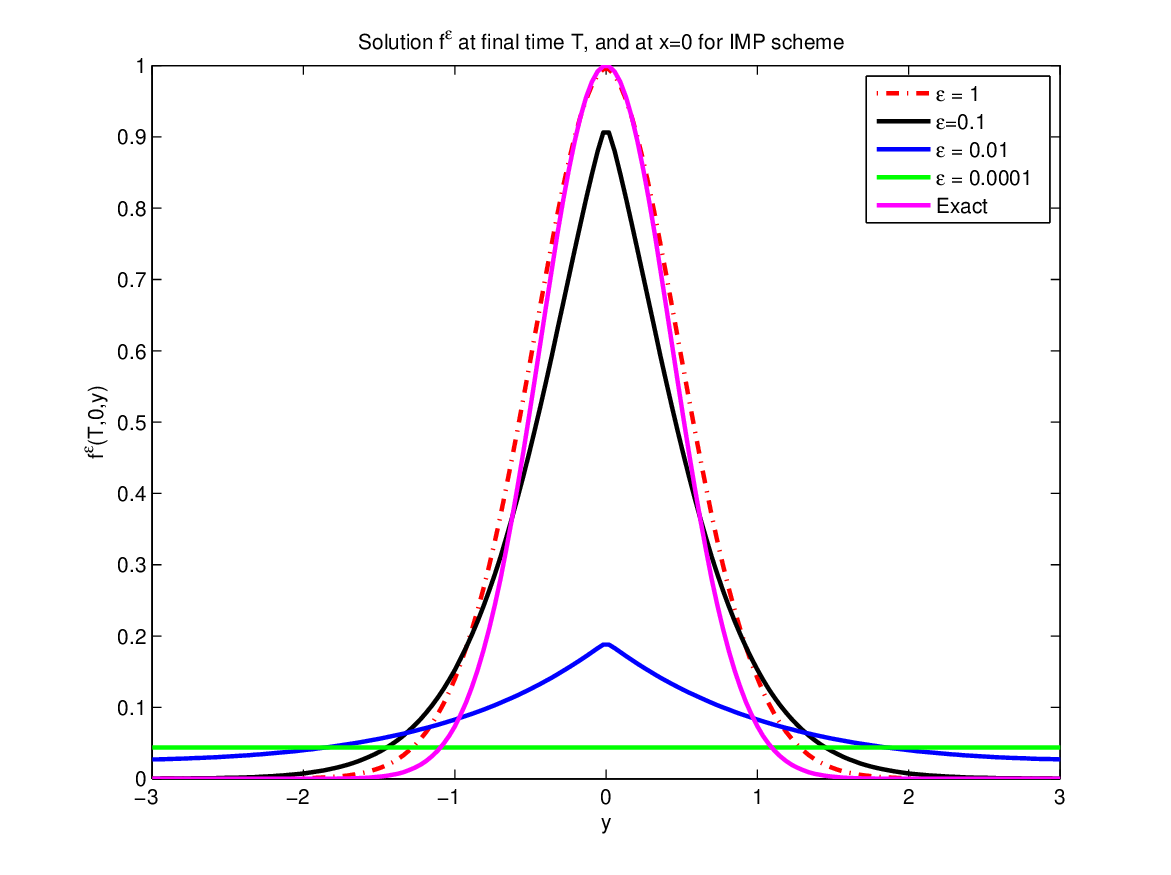}}
      \label{f0}    
	\subfloat[Lagrange-multiplier]{
      \includegraphics[width=0.44\textwidth]{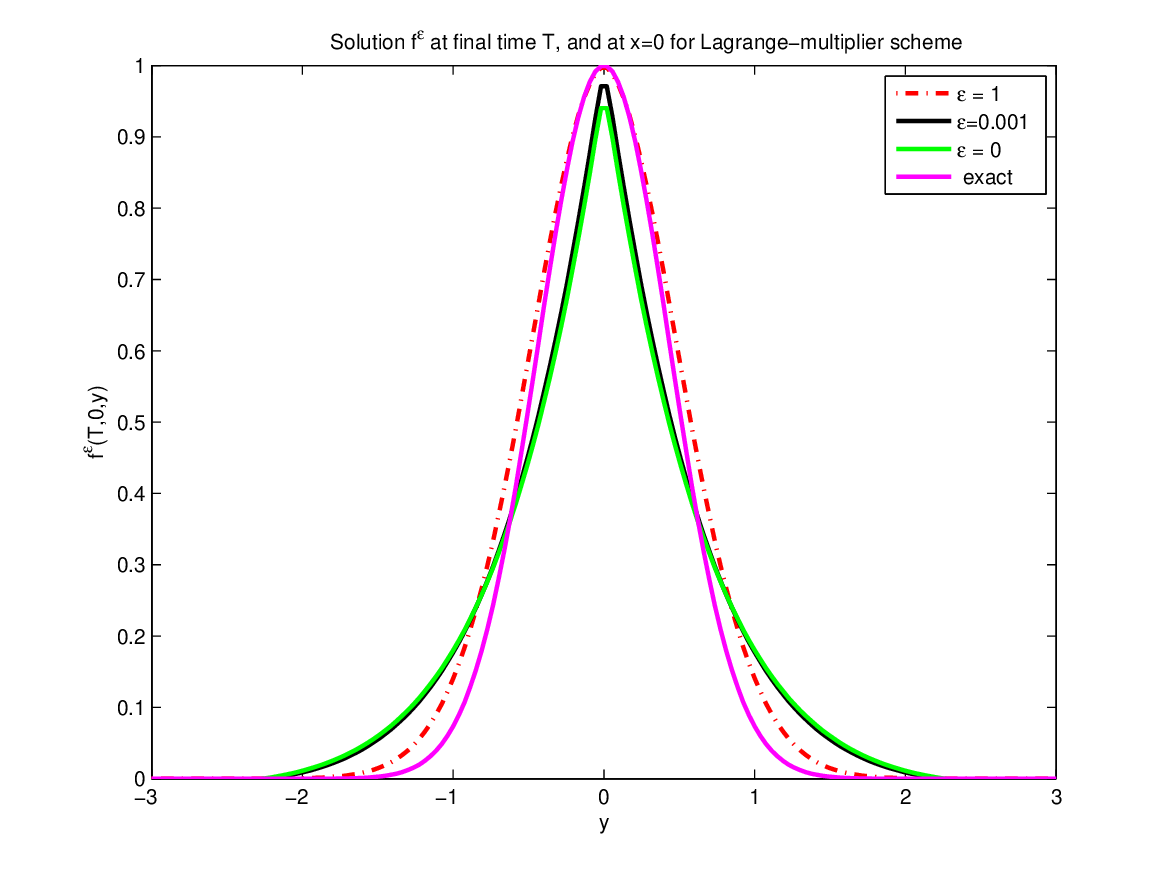}}
      \label{f_eps}   
      \caption{\small{Representation of a cut at $x=0$ of $f^{\epsilon}_{num}$ at the final time $T$ for the IMP and Lagrange-multiplier schemes, and several values of $\epsilon$.}}
\label{figure_gauss}
\end{figure}

\begin{figure}[ht]
\centering
	\subfloat[IMP ; $\epsilon=1$ ]{
      \includegraphics[width=0.44\textwidth]{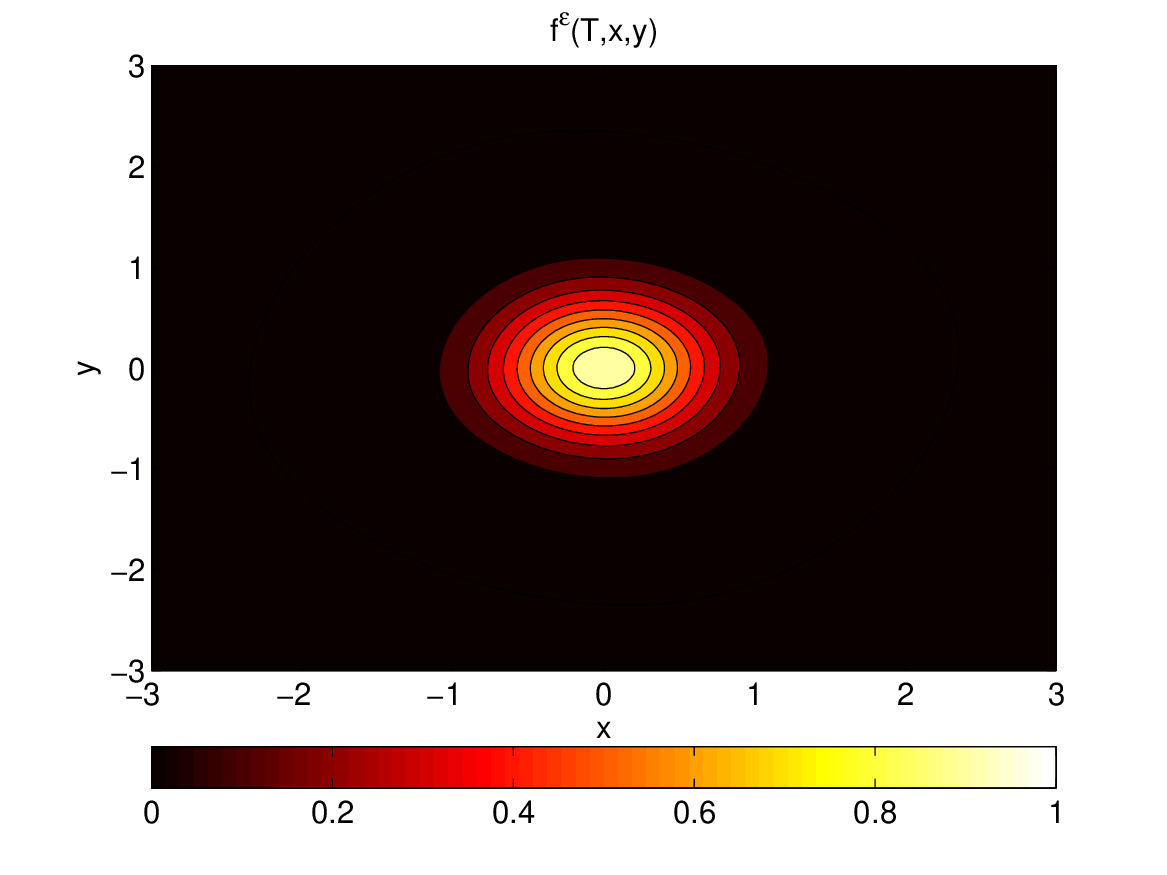}}
      \label{f0}    
	\subfloat[Lagrange-multiplier ; $\epsilon=1$ ]{
      \includegraphics[width=0.44\textwidth]{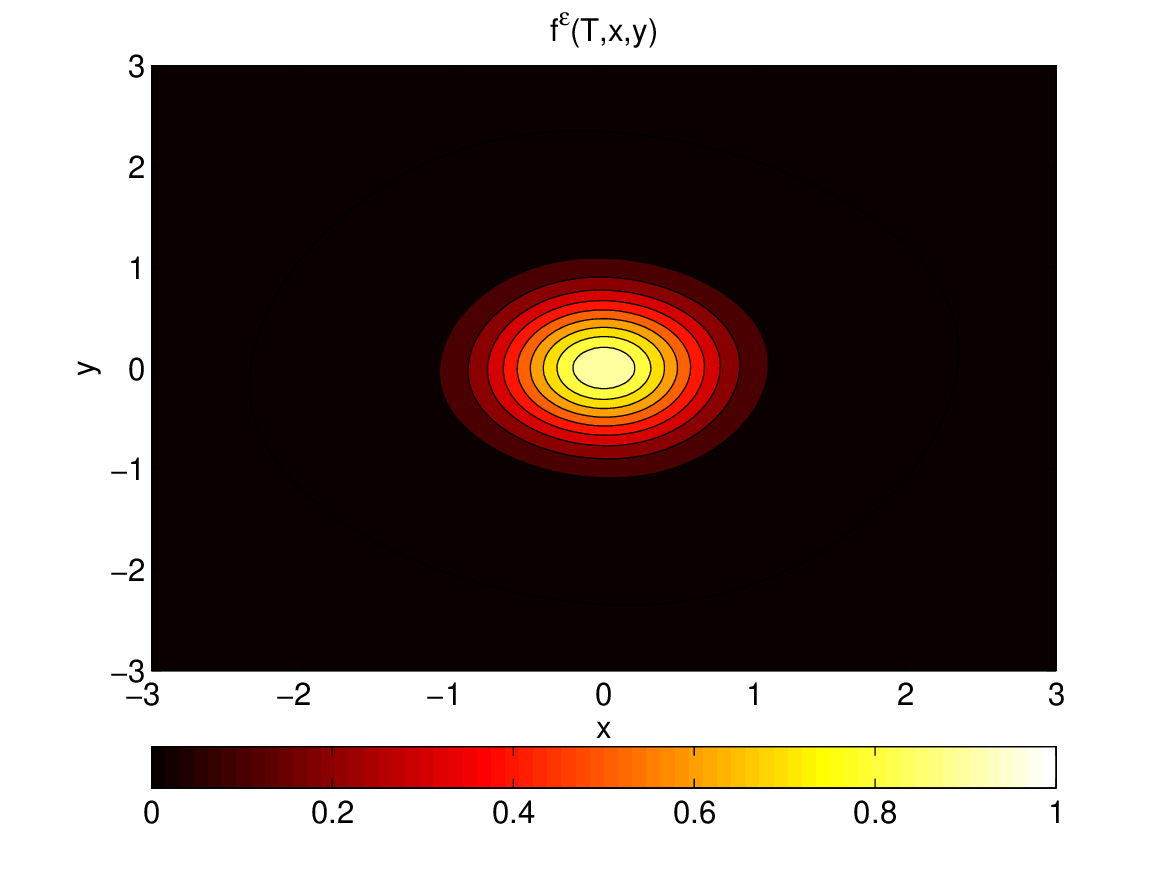}}
      \label{f_eps} 
      \vfill
      \subfloat[IMP ; $\epsilon=0.01$]{
      \includegraphics[width=0.44\textwidth]{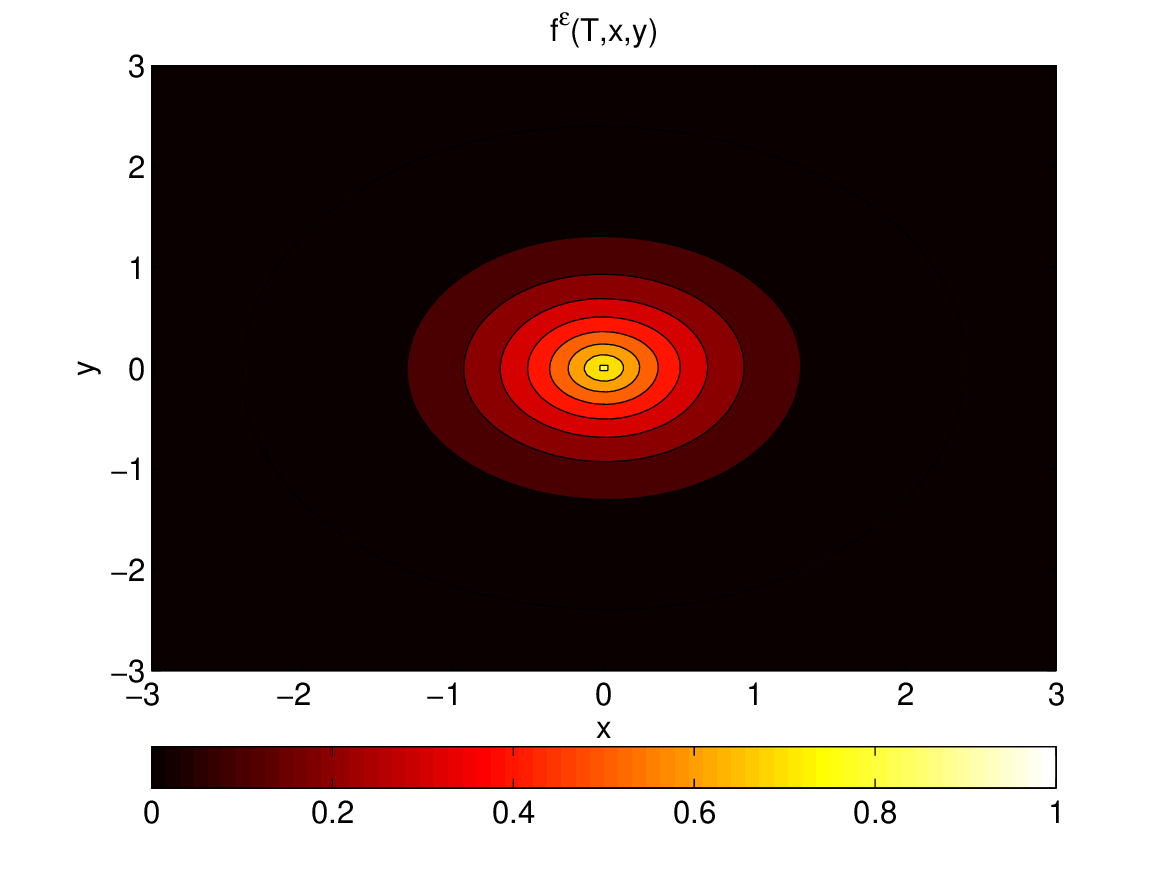}}
      \label{f0}    
	\subfloat[Lagrange-multiplier ; $\epsilon=0.01$]{
      \includegraphics[width=0.44\textwidth]{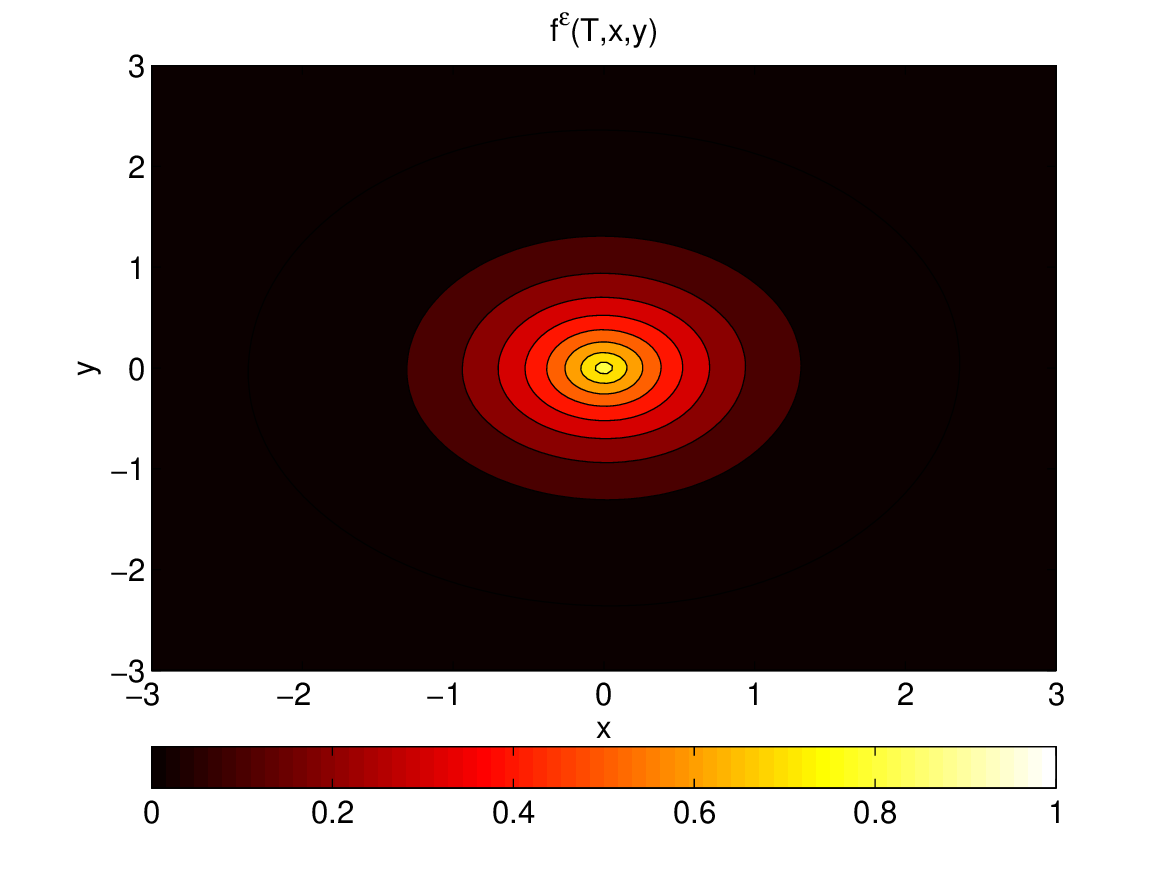}}
      \label{f_eps} 
      \vfill
      \subfloat[IMP ; $\epsilon=5.10^{-4}$]{
      \includegraphics[width=0.44\textwidth]{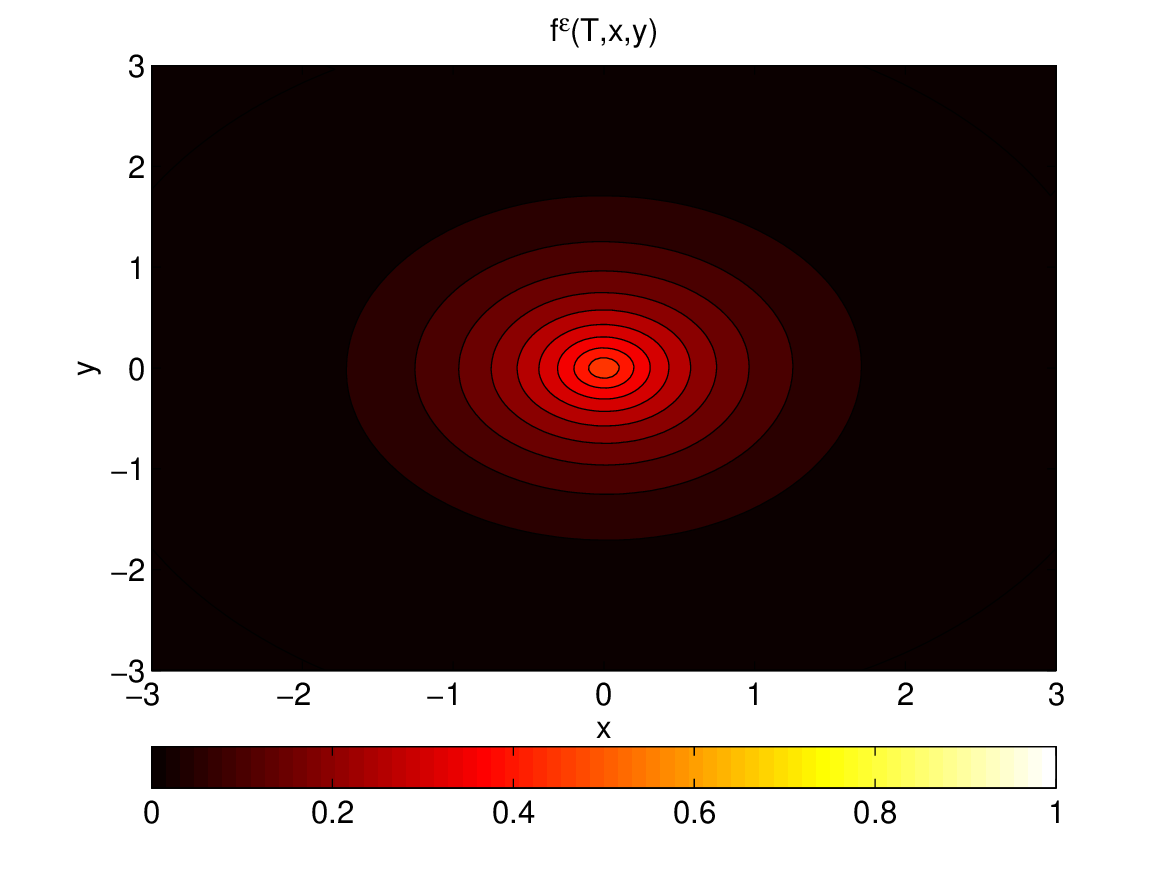}}
      \label{f0}    
	\subfloat[Lagrange-multiplier ; $\epsilon=5.10^{-4}$]{
      \includegraphics[width=0.44\textwidth]{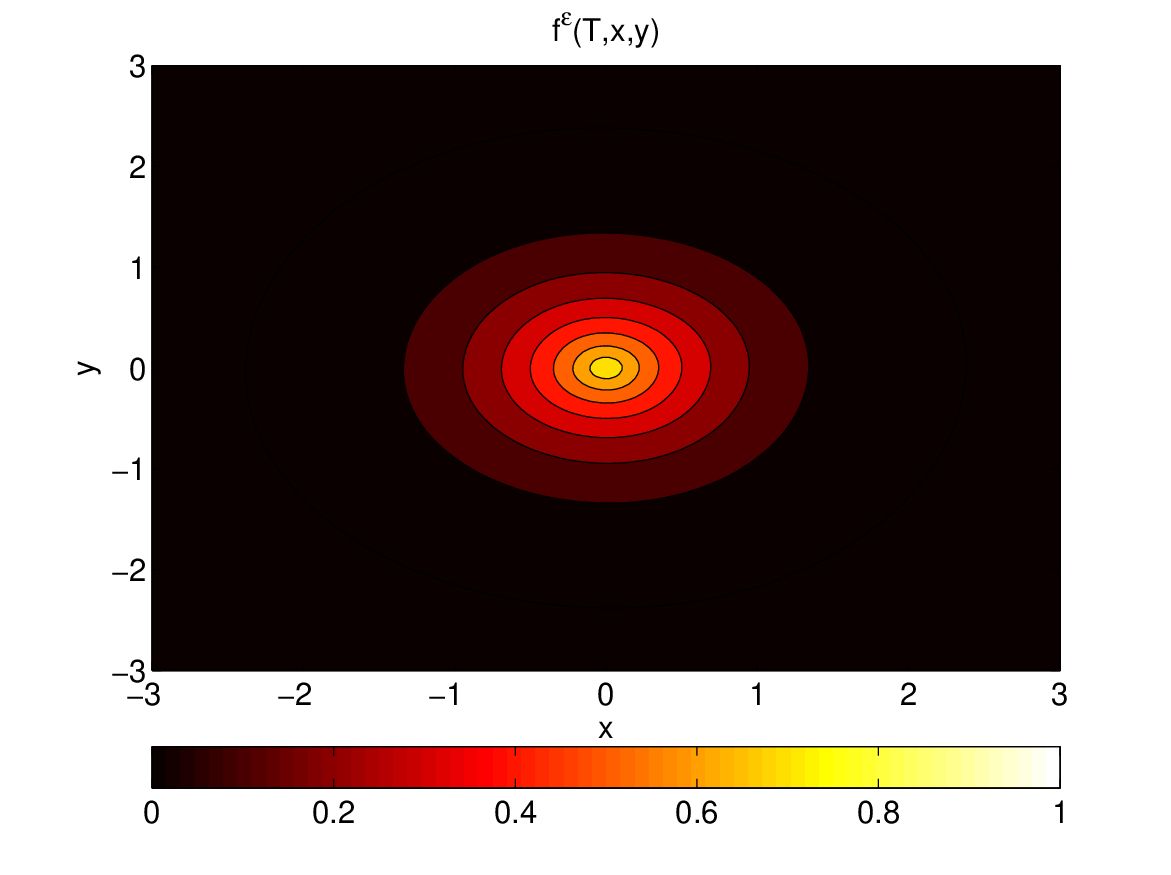}}
      \label{f_eps} 
      \vfill
      \subfloat[IMP ; $\epsilon=10^{-10}$]{
      \includegraphics[width=0.44\textwidth]{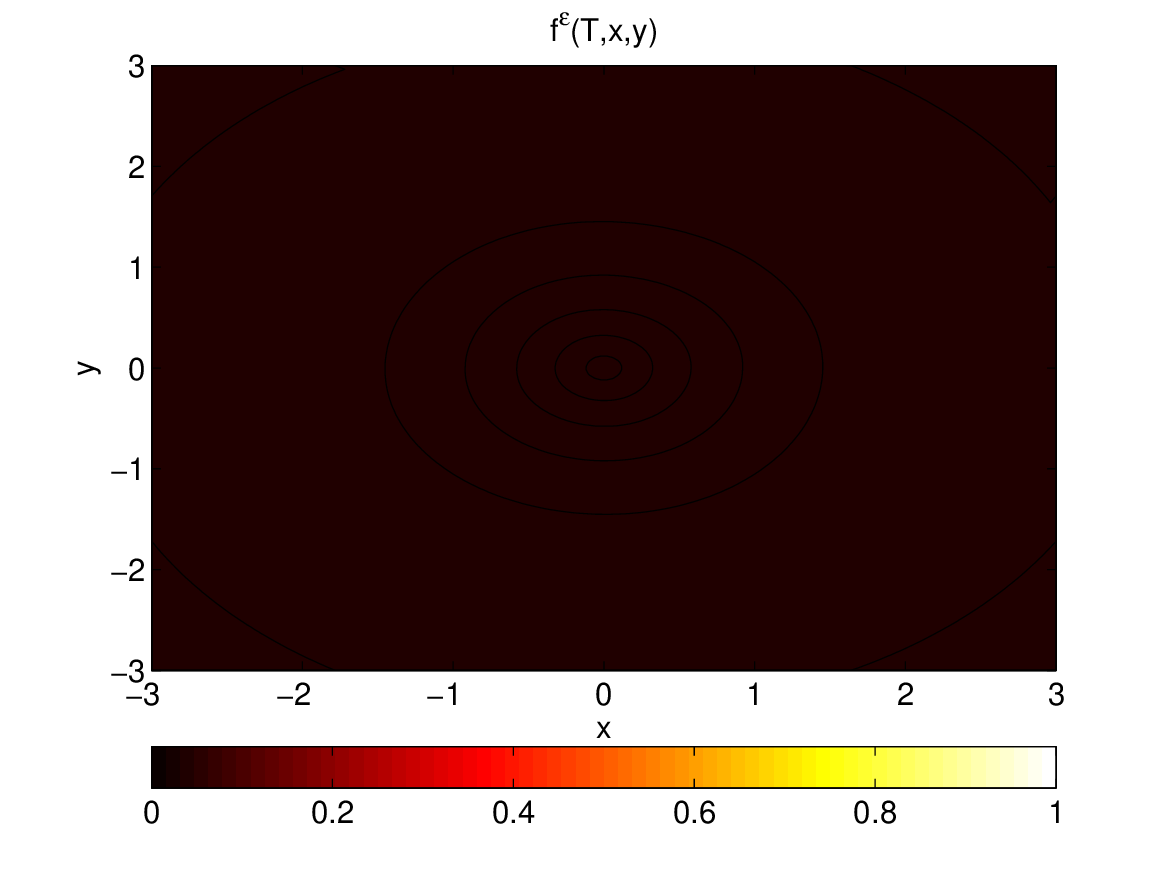}}
      \label{f0}    
	\subfloat[Lagrange-multiplier ; $\epsilon=0$]{
      \includegraphics[width=0.44\textwidth]{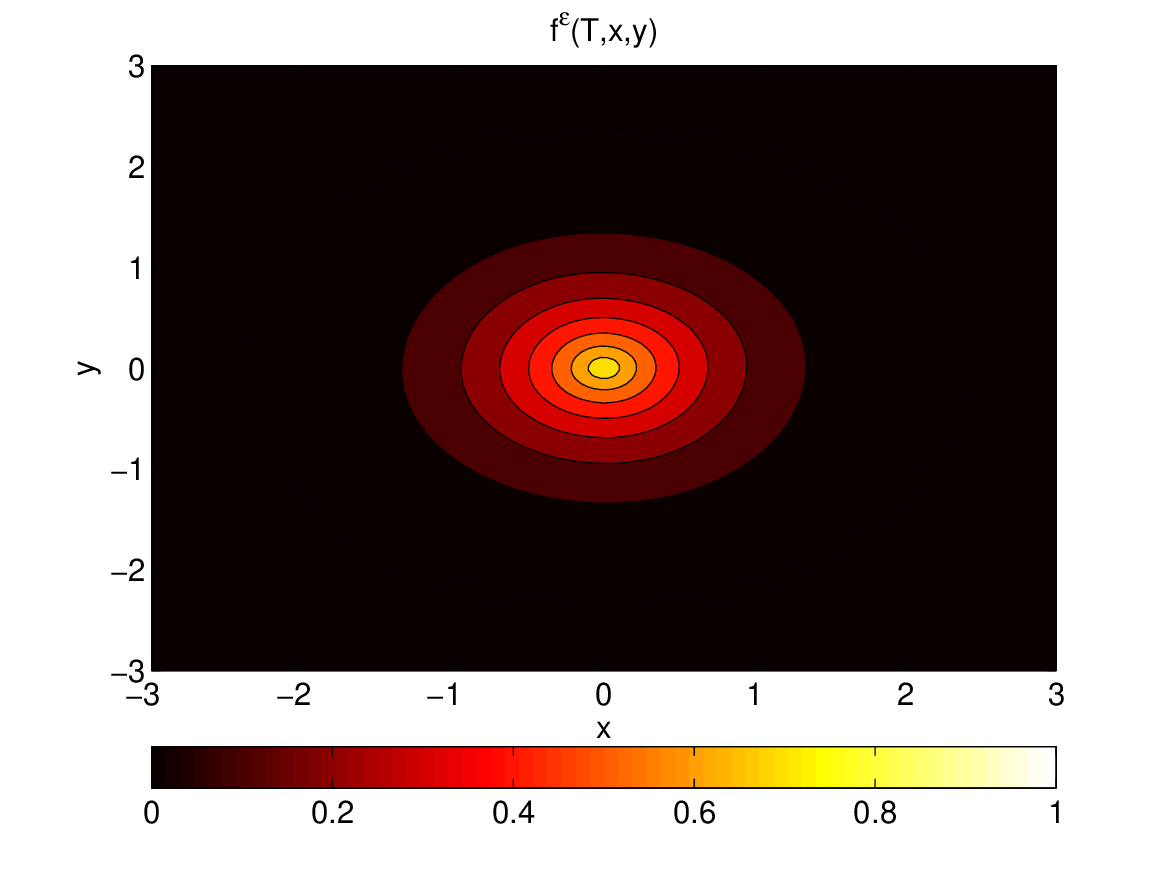}}
      \label{f_eps} 
\caption{Representation of the function $f^{\epsilon}$ at the final time $T$ for the IMP and Lagrange-multiplier scheme, with several values of $\epsilon$.}
\label{figure14}
\end{figure}

%%%%%%%%%%
\subsection{Numerical analysis}
%%%%%%%%%%%
The aim of this section is to explain the plots presented before. In particular we will investigate why the IMP scheme does not work for small $\epsilon$-values, whereas the Lagrange-multiplier scheme preserves the asymptotics. First of all, we compute the local truncation error of both schemes. We shall consider only the case $x \geq 0$ and $y \geq 0$, the remaining cases changing nothing in the following reasoning.
%%%%%
\subsubsection{IMP scheme}
%%%%%
We begin by recalling the expression of the full discretized expression of this scheme : 
\be \label{Euler_2}
(IMP)_\epsilon \quad\,\, {f^{\epsilon,n+1}_{i,j}-f^{\epsilon,n}_{i,j} \over \Delta t} + \frac{y_j}{\epsilon} \; \frac{f_{i,j}^{\epsilon,n+1} - f_{i-1,j}^{\epsilon,n+1}}{\Delta x} - \frac{x_i}{\epsilon}  \; \frac{f_{i,j+1}^{\epsilon,n+1}-f_{i,j}^{\epsilon,n+1}}{\Delta y}=0\,, \quad \forall (n,i,j) \in Q_h \,.
\ee

\debthm
The IMP scheme \eqref{Euler_2} is consistent with the second Vlasov problem \eqref{eq_cir_0}, first order accurate in time and in space. Moreover, the local truncation error writes
$$
\mathcal{T}_I (t_n,x_i,y_j,\Delta t, \Delta x,\Delta y) = - \nabla \cdot{ \Big[D_I \nabla f^{\epsilon}\Big] }+ \mathcal{O}(\Delta t ^2)  + \mathcal{O}(\Delta x ^2) + \mathcal{O}(\Delta y ^2).
$$
where 
$$
D_I:= {1 \over \epsilon}\, \left(\begin{array}{cc} \ds  \frac{y_j \Delta x}{2}\Bigg(1 + \frac{\alpha_j}{\epsilon}\Bigg) & \ds \frac{-x_iy_j \Delta t}{2 \epsilon}  \\ \ds \frac{-x_iy_j \Delta t}{2 \epsilon} &\ds  \frac{x_i \Delta y }{2}\Bigg(1 + \frac{\beta_i}{\epsilon}\Bigg)\end{array}\right), \quad \alpha_j := \frac{y_j \Delta t}{\Delta x}, \quad  \ds \beta_i := \frac{x_i \Delta t}{\Delta y}.
$$
%
%Finally, we observe that the IMEX scheme \eqref{Euler_2} is a second-order scheme for the modified Vlasov equation
%\begin{align*}
% &\partial_t g^{\epsilon} + \frac{y}{\epsilon} \; \partial_x g^{\epsilon} - \frac{x}{\epsilon} \; \partial_y g^{\epsilon}  -\frac{\Delta t}{2 \epsilon^2} \; \Big [2xy  \; \partial_{xy}g^{\epsilon} +y \; \partial_y g^{\epsilon} + x\; \partial_x g^{\epsilon}\Big] - \frac{y \Delta x }{2 \epsilon}(1+\alpha(y)) \; \partial_{xx}g^{\epsilon} \\ &- \frac{x\Delta y}{2 \epsilon}(1+ \beta(x)) \;\partial_{yy} g^{\epsilon}=0\,, \numberthis \label{modified_2}
%\end{align*}
%where $\alpha(y):= \frac{y \Delta t}{\Delta x} $ and $\beta(x):= \frac{x \Delta t}{\Delta y}$.
\finthm

\debproof
This proof is very similar to the proof of Theorem \ref{th_i}.
\finproof

\debrem
Contrary to the first toy-model, where the diffusion was $1/\epsilon$-dependent only in  the anisotropy-direction, which was aligned with the coordinate system, in the present case, the diffusion-matrix is scaled by a $1/\epsilon$ factor, meaning that this time we have a very strong $1/\epsilon$-dependent diffusion in all directions. This large diffusion leads rapidly to a damping of the solution towards zero, as $\epsilon$ becomes smaller, and leads thus to completely erroneous results.
\finrem

%%%%%%%%%%
\subsubsection{Lagrange-multiplier scheme}
%%%%%%%
We use the same reasoning for the Lagrange-multiplier scheme
\be \label{Lag_cir}
(La)_\epsilon \quad \left\{
\begin{array}{l}
\ds \partial_t f^{\epsilon} + y \;  \partial_x q^{\epsilon} - x \; \partial_y q^{\epsilon} = 0,\\[6mm]
\ds y \partial_x f^{\epsilon} - x\; \partial_y f^{\epsilon} = \epsilon \Big(y\;  \partial_x q^{\epsilon} - x \; \partial_y q^{\epsilon}\Big) - (\Delta x \Delta y)^\gamma \; q^{\epsilon}. 
\end{array}
\right.
\ee

\noindent  Supposing $y\geqslant0$ and $x\geqslant0$, we have the full discretization of $(La)_\epsilon$
\be \label{lag_2}
(La)_\epsilon \quad \left\{
\begin{array}{l}
\ds {f^{\epsilon,n+1}_{i,j}-f^{\epsilon,n}_{i,j} \over \Delta t} + y_j \; \frac{q_{i,j}^{\epsilon,n+1} - q_{i-1,j}^{\epsilon,n+1}}{\Delta x} -x_i \; \frac{q_{i,j+1}^{\epsilon,n+1}-q_{i,j}^{\epsilon,n+1}}{\Delta y}=0\,, \\[6mm]
\ds  y_j \; \frac{f_{i,j}^{\epsilon,n+1} - f_{i-1,j}^{\epsilon,n+1}}{\Delta x} -x_i \;  \frac{f_{i,j+1}^{\epsilon,n+1}-f_{i,j}^{\epsilon,n+1}}{\Delta y} = \epsilon \; \Bigg(y_j \; \frac{q_{i,j}^{\epsilon,n+1} - q_{i-1,j}^{\epsilon,n+1}}{\Delta x} -x_i  \;\frac{q_{i,j+1}^{\epsilon,n+1}-q_{i,j}^{\epsilon,n+1}}{\Delta y}\Bigg)\\ \ds \hspace{2cm}- (\Delta x \Delta y)^\gamma \; q^{\epsilon,n+1}_{i,j}. 
\end{array}
\right.
\ee

\debthm
The Lagrange-multiplier scheme \eqref{lag_2} is consistent with the second Vlasov model \eqref{eq_cir_0} and first order accurate in time and in space. Furthermore,  the local truncation error writes 
\begin{align*}
&\left(
\begin{array}{c}
\mathcal{T}_{L1} \\\mathcal{T}_{L2}
\end{array}
\right) = 
\left(
\begin{array}{cc}
\nabla \cdot & 0 \\
0 & \nabla \cdot \end{array}\right)
\, {\left(\begin{array}{cc}0 & D_{L_1} \\
D_{L_2} & -\epsilon D_{L_2}
\end{array}
\right)
\left(\begin{array}{c}
\nabla f^{\epsilon} \\ \nabla q^{\epsilon}
\end{array}
\right) } +\mathcal{O}(\Delta t ^2, \Delta x ^2,\Delta y ^2) \\[4mm]
&\hspace{1.7cm} =\left(
\begin{array}{c}
\displaystyle \nabla\,(D_{L_1} \, \nabla q^{\epsilon})\\[3mm]
\displaystyle \nabla\,(D_{L_2} \, \nabla f^{\epsilon}) - \epsilon \nabla\,(D_{L_2} \, \nabla q^{\epsilon})
\end{array}
\right)+\mathcal{O}(\Delta t ^2, \Delta x ^2,\Delta y ^2)
\end{align*}
where 
$$
D_{L1} :=  \left(\begin{array}{cc} \ds  \frac{y_j \Delta x}{2}\Bigg(1 + \frac{\alpha_j}{\epsilon}\Bigg) & \ds \frac{-x_iy_j \Delta t}{2 \epsilon}  \\ \ds \frac{-x_iy_j \Delta t}{2 \epsilon} &\ds  \frac{x_i \Delta y }{2}\Bigg(1 + \frac{\beta_i}{\epsilon}\Bigg)\end{array}\right), \quad D_{L2} = \left(\begin{array}{cc} \ds \frac{y_j \Delta x}{2} &0  \\ 0 &\ds \frac{x_i \Delta y}{2}\end{array}\right). 
$$
\finthm

\debproof
We begin by the computation of the $\mathcal{T}_{L1}$ term. Supposing sufficient regularity for the functions $f^{\epsilon}$ and $q^{\epsilon}$, we use Taylor series expansion to get
\begin{align*}
\mathcal{T}_{L1}(t_n,x_i,y_j,\Delta t, \Delta x, \Delta y) = &\frac{\Delta t}{2} \; \partial_{tt}f^{\epsilon}   + y_j \Delta t \; \partial_{xt}q^{\epsilon}-  \frac{y_j  \Delta x}{2} \;  \partial_{xx} q^{\epsilon}- \frac{x_i \Delta y}{2} \; \partial_{yy} q^{\epsilon}  \\ &- x_i \Delta t \; \partial_{ty}q^{\epsilon}  +  \mathcal{O}(\Delta t ^2)  + \mathcal{O}(\Delta x ^2) + \mathcal{O}(\Delta y ^2).
\end{align*}

\noindent Since the first equation \eqref{Lag_cir} is verified, we can write 
$$
\partial_{tt} f^{\epsilon} = -y_j \; \partial_{xt} q^{\epsilon} + x_i \; \partial_{yt} q^{\epsilon},
$$
\noindent And we differentiate in time the second equation of \eqref{Lag_cir} to obtain 
\begin{align*}
\mathcal{T}_{L1}(t_n,x_i,y_j,\Delta t, \Delta x, \Delta y) = &  \frac{\Delta t}{2\epsilon} \; \partial_t (y_j \; \partial_x f^{\epsilon}- x_i \; \partial_y f^{\epsilon})- \; \frac{y_j  \Delta x}{2} \partial_{xx} q^{\epsilon}- \frac{x_i \Delta y}{2} \; \partial_{yy} q^{\epsilon}  \\ & + \mathcal{O}(\Delta t ^2)  + \mathcal{O}(\Delta x ^2) + \mathcal{O}(\Delta y ^2).
\end{align*}

\noindent We have the following relations  
$$
\partial_{tx}f^{\epsilon} = x \; \partial_{yx} q^{\epsilon} - y \; \partial_{xx}q^{\epsilon} + \partial_y q^{\epsilon}, \quad \partial_{ty}f = x \; \partial_{yy} q^{\epsilon} - \partial_x q^{\epsilon} - y \; \partial_{xy} q^{\epsilon}.
$$

\noindent The local truncation error writes finally
\begin{align*}
\mathcal{T}_{L1}(t_n,x_i,y_j,\Delta t, \Delta x, \Delta y) = &  \frac{\Delta t}{2\epsilon}  (y_j \;  \partial_y q^{\epsilon}  + x_i \;  \partial_x q^{\epsilon} +2 x_i y_j \; \partial_{xy}q^{\epsilon})-  \; \frac{y_j \Delta x}{2}(1+\alpha_j /\epsilon) \; \partial_{xx} q^{\epsilon} \\ &- \frac{x_i \Delta y}{2}(1+\beta_i/\epsilon) \; \partial_{yy} q^{\epsilon}  + \mathcal{O}(\Delta t ^2)  + \mathcal{O}(\Delta x ^2) + \mathcal{O}(\Delta y ^2).
\end{align*}
\noindent With an analogous reasoning, we compute the truncation error of the second equation:

\begin{align*}
\mathcal{T}_{L2}(t_n,x_i,y_j,\Delta t, \Delta x,\Delta y) = &-  \; \frac{y_j \Delta x}{2} \; \partial_{xx} f^{\epsilon} -  \;  \frac{x_i\Delta y}{2} \; \partial_{yy} f^{\epsilon} - \epsilon \;  \Big(  \; \frac{y_j  \Delta x}{2} \; \partial_{xx} q^{\epsilon} + \frac{x_i \Delta y}{2} \;  \partial_{yy} q^{\epsilon} \Big)  \\&+ \mathcal{O}(\Delta t ^2)  + \mathcal{O}(\Delta x ^2) + \mathcal{O}(\Delta y ^2).
\end{align*}
\finproof

\debrem
In contrast to the first Vlasov toy-model \eqref{AV}, the IMP and Lagrange-multiplier schemes do not have the same behavior with respect to the local truncation error. More particularly, the dependence on $\epsilon$ is very different.  The IMP-scheme is diffusing in all directions, diffusion proportional to $1/\epsilon$. The only $1/\epsilon$-dependent diffusion in the Lagrange-multiplier scheme arises in relation with the auxiliary unknown $q^\epsilon$, {\it i.e.} in the term $\nabla\,(D_{L_1} \, \nabla q^{\epsilon})$. And one can immediately verify that the $1/\epsilon$-dependence arises only aligned with the anisotropy field lines, and not perpendicular to them. Indeed, one gets immediately for the diffusion along resp. perpendicular to the field lines:
$$
\begin{array}{lll}
\ds (y\,,\, -x)\, D_{L_1} \, (y\,,\, -x)^T&=&\ds   {y^3 \, \Delta x \over 2} + {x^3 \, \Delta y \over 2}+ { \Delta t \over 2} (x^2+y^2)^2\\[3mm]
\ds (x\,,\, y)\, D_{L_1} \, (x\,,\, y)^T&=&\ds  {x\, y \over 2} \left[x\, \Delta x + y\, \Delta y \right] \,.
\end{array}
$$
\finrem
\section{Concluding remarks}
%%%%%%
To conclude, let us summarize here the knowledge we acquired about the resolution of anisotropic Vlasov equations of the type \eqref{V_C} arising in fusion plasma modelisation. Two types of techniques can be adopted from the beginning. One can decide to pass directly to polar coordinates in velocity and get hence a field-aligned formulation as for ex. \eqref{V_P_bis}. In this case, a simple IMEX-scheme is the most appropriate scheme to be used, being simple enough and giving rise to accurate results up a sufficiently small $\epsilon$-value. However, the disadvantage is that one has to change coordinate system, which can be rather cumbersome if the magnetic field is variable, in time and space.\\
The second technique is rather simple, as it avoids to pass to field-aligned coordinates and remains in a nice Cartesian framework. The drawback is that in this case it is no more sufficient to implicit the stiff term and take the other terms explicitly. Indeed, for small $\epsilon$-values (already $\epsilon =10^{-4}$), meaning strong magnetic fields as in tokamak plasmas, an IMEX scheme would lead to erroneous results. An Asymptotic-Preserving reformulation like our "Lagrange-multiplier-method" is more adequate, leading in the limit $\epsilon \to 0$ towards the right Limit-problem. This Lagrange-multiplier-method is indeed usable for all $\epsilon \ge 0$ and gives accurate and stable results independently on $\epsilon$.  However there is a disadvantage, namely the fact that it is more time-consuming, as it involves an additional unknown $q^\epsilon$.\\
Solving an anisotropic Vlasov equation of the type \eqref{V_C} needs hence an {\it a priori} decision, which of these two techniques to follow. 
The first technique is at the moment the basis of several codes.
%the GYSELA-code \cite{Grand}.
The second technique has not be tested up to now, and its rigorous validation and comparison with the first one will be the aim of a forthcoming paper, in a more physical context.

\bigskip

%%%%%%%%%%%%%%%%%%%%%%%%%%%%%%%%%%%%%%%%%%%%%%%%%%%%%%%%%%%%%%%%%%%
\noindent {\bf Acknowledgments.} The authors would like to acknowledge support from the ANR PEPPSI  (Plasma Edge Physics and Plasma-Surface Interactions, 2013-2017). Furthermore, this work has been carried out within the framework of the EUROfusion Consortium and has received funding from the Euratom research and training program 2014-2018 under grant agreement No 633053. The views and opinions expressed herein do not necessarily reflect those of the European Commission.

%%%%%%%%%%%%%%

\begin{thebibliography}{99}
\footnotesize{

%\bibitem{allaire} Allaire G., {\it Homogenization and two-scale convergence}, SIAM J. Math. Anal. {\bf 23.6} (1992), pp 1482--1518.

\bibitem{mb} Bostan M.:  Transport equations with disparate advection fields. Application to the gyrokinetic models in plasma physics, SIAM J. Sci. Comp. {\bf 31}(1) 334-368 (2008).

%\bibitem{mbflr} 
%Bostan, M.:  
%The Vlasov-Poisson system with strong external magnetic field. Finite Larmor radius regime, 
%Asymptot. Anal.  {\bf 61}(2) 91-123 (2009).

%\bibitem{Briz} Brizard A.J., Hahm T.S.: Foundations of nonlinear gyrokinetic theory. Rev. Mod. Phys. Vol. 79, 421--468 (2007).

\bibitem{CHENF}
Chen F. F.: Plasma Physics and controlled fusion, Springer Verlag. New York, (2006).

\bibitem{NL} Crouseilles N., Lemou M., {\it An asymptotic preserving scheme based on a micro-macro decomposition for collisional Vlasov equations: diffusion and high-field scaling limits}, Kinetic Related Models {\bf 4} (2011), 441--477. 

\bibitem{NICOO} Crouseilles N., Lemou M., M\'ehats F., {\it Asymptotic-Preserving schemes for oscillatory Vlasov-Poisson equations}, JCP {\bf 248} (2013), pp 287--308.

\bibitem{DNP} De Cecco A., Negulescu C., Possanner S., {\it Asymptotic transition from kinetic to adiabatic electrons along magnetic field lines}, SIAM MMS (Multiscale Model. Simul.) \textbf{15} (2017), no. 1, 309--338.


\bibitem{DDLNN} Degond P., Deluzet F., Lozinski A., Narski J.,
  Negulescu C., {\it Duality based Asymptotic-Preserving Method for
    highly anisotropic diffusion equations}, Communications in
  Mathematical Sciences {\bf 10} (2012), no. 1, 1--31.

\bibitem{DLNN} Degond P.,  Lozinski A., Narski J.,
  Negulescu C. {\it
    An Asymptotic-Preserving method for highly anisotropic elliptic
    equations based on a micro-macro decomposition}, Journal of
  Computational Physics {\bf 231} (2012), no. 7, 2724--2740. 

%\bibitem{frenod0} Fr\'enod E., {\it Two-Scale Convergence}, Esaim Proceedings, {\bf 38} (2012), pp 1--35.

%\bibitem{frenod1}
%Fr\'enod E., Sonnendr\"ucker E.:  Homogenization of the Vlasov Equation and of the Vlasov-Poisson System with a Strong External Magnetic Field. Asymp. Anal., Vol. 18, No 3-4, pp 193--214, (1998).

%\bibitem{frenod2}
%Fr\'enod E., Sonnendr\"ucker E.: Long Time Behavior of the Vlasov Equation with Strong External Magnetic Field. Math. Mod. Meth. Appl. Sciences, Vol. 10, No 4, pp 539--553, (2000).

%\bibitem{frenod3}
%Fr\'enod E., Sonnendr\"ucker E.:  The finite Larmor radius approximation. SIAM J. Math. Anal., Vol. 32, No 6, pp 1227--1247, (2001).

\bibitem{filjin2} Filbet F., Jin S., \textit{An Asymptotic Preserving Scheme for the ES-BGK model of the Boltzmann equation}, J. Sci. Computing {\bf 46} (2011), no. 2, 204-224. 

%\bibitem{Fri} Frieman E.A., Chen L.: Nonlinear gyrokinetic equations for low-frequency electromagnetic waves in general plasma equilibria. Physics of Fluids {\bf 25}, 502--5 (1982).

\bibitem{Xav} Garbet X., Idomura Y., Villard L., Watanabe T.: Gyrokinetic simulations of turbulent transport. Nuclear Fusion Vol. 50, No 4, (2010).

\bibitem{Ruther}
Goldston R. J., Rutherford P.H.,  Plasma Physics. Taylor $\&$ Francis Group, (1995).

\bibitem{Grand} Grandgirard V.,  Sarazin Y.,  Garbet X.,  Dif-Pradalier G., Ghendrih Ph., Crouseilles N., Latu G.,  Son-
nendr\"ucker E., Besse N., Bertrand P., {\it GYSELA, a full-f global gyrokinetic Semi-Lagrangian code for ITG turbulence simulations}, Theory of Fusion Plasmas
{\bf 871} (2006), American Institute of Physics Conference Series, pp 100--111.

\bibitem{Max}
Ghendrih Ph., Hauray M., Nouri A.: Derivation of a gyrokinetic model. Existence and uniqueness of soecific stationary solutions. Kinetic and Related Models, Vol. 2, No 4, pp 707-725, (2009).

\bibitem{hazeltine_meiss}
Hazeltine R.D., Meiss J.D.: Plasma confinement. Dover Publications, New York (2003).

\bibitem{holmes} Holmes M.H.: Introduction to Numerical Methods in Differential Equations, Springer, New York, (2007).

\bibitem{Jin_rev} Jin S., \textit{Asymptotic preserving (AP) schemes for multiscale kinetic and hyperbolic equations: a review}, Rivista di Matematica della Universita di Parma \textbf{3}, 177-216 (2012).

\bibitem{lm} Lemou M., Mieussens L.:  A new asymptotic preserving scheme based on micro-macro formulation for linear kinetic equations in the diffusion limit. J. Differential Equations, {\bf 249},  pp. 1620-1663 (2010). 

\bibitem{leveque} LeVeque R.J.: Finite Difference Methods for Ordinary and Partial Differential Equations. Siam, Philadelphia, (2007).

\bibitem{LNN}  Lozinski A., Narski J., Negulescu C., {\it Highly anisotropic temperature balance equation and its asymptotic-preserving resolution}, M2AN (Mathematical Modelling and Numerical Analysis) 48 (2014) 1701--1724.


\bibitem{Majda} Majda A.J., Bertozzi A.L., {\it Vorticity and incompressible flow}, Cambridge University Press, 2002.

\bibitem{MN} Mentrelli A., Negulescu C., {\it Asymptotic-Preserving scheme for highly anisotropic non-linear diffusion equations}, Journal of Comp. Phys. {\bf 231} (2012), 8229--8245. 

\bibitem{CN} Negulescu C. {\it Kinetic modelling of strongly magnetized tokamak plasmas with mass disparate particles. The electron Boltzmann relation}, submitted.

\bibitem{PE}  Negulescu C. {\it Asymptotic-Preserving schemes. Modeling, simulation and mathematical analysis of magnetically confind plasmas}, Riv. Mat. Univ. Parma. {\bf 4}(2) (2013) 265--343.

%\bibitem{Roma}  Romanelli F., Briguglio S., {\it Toroidal semicollisional microinstabilities and anomalous electron and ion transport}, Physics of Fluids B: Plasma Physics {\bf 2}, no. 4 (1990), pp.754-763.



%\bibitem{golse}
%Golse F.,  Saint-Raymond L.: The Vlasov-Poisson system with strong magnetic field. J. Math. Pures Appl.,
%{\bf78}, p. 791-817, (2001).



%\bibitem{negulescu1} Negulescu C.: Modeling, Simulation and Mathematical Analysis of Magnetically Confined Plasmas, Riv. Mat. Univ. Parma. 4 (2013), no.2, pp. 265--343.

\bibitem{Tref} Trefethen L., Bau D., {\it Numerical linear algerbra}, SIAM Philadelphia, 1997.


}
\end{thebibliography}
\end{document}